\documentclass[aop, preprint, 12pt]{imsart}

\RequirePackage[OT1]{fontenc}
\RequirePackage{amsthm,amsmath}
\RequirePackage[numbers]{natbib}
\RequirePackage[colorlinks,citecolor=blue,urlcolor=blue]{hyperref}
\usepackage{amsfonts}
\usepackage{color}
\usepackage{pifont}
\usepackage{amssymb}
\usepackage[english]{babel}
\usepackage[T1]{fontenc}
\usepackage{graphicx}
\usepackage{subfigure}
\usepackage{latexsym}
\usepackage{amstext}
\usepackage{mathrsfs}
\usepackage{hyperref}
\usepackage{dsfont}
\usepackage{graphics}
\usepackage{enumerate}
\usepackage{paralist}
\usepackage{color}
\usepackage{fullpage}
\newtheorem{prop}{Proposition}[section]
\newtheorem{assumption}{Assumption}[section]

\newtheorem{rem}[prop]{Remark}
\newtheorem{lem}[prop]{Lemma}

\numberwithin{equation}{section}

\newcommand{\beq}{\begin{eqnarray}}
\newcommand{\beqq}{\begin{eqnarray*}}
\newcommand{\eeq}{\end{eqnarray}}
\newcommand{\eeqq}{\end{eqnarray*}}
\newcommand{\eps}{\varepsilon}
\newcommand{\cmark}{\ding{51}}%
\newcommand{\xmark}{\ding{55}}

\newcommand{\absxi}{ {\xi^{|\cdot|}}}

\usepackage{mathtools}
\usepackage[english]{babel}
\usepackage[utf8]{inputenc}
\usepackage{amsmath}
\usepackage{graphicx}
\usepackage[colorinlistoftodos]{todonotes}
\usepackage{polynom}
\usepackage{amsfonts}
\usepackage{dsfont}
\usepackage{amsthm}
\usepackage{hyperref}
\usepackage{graphicx}
\newtheorem{theorem}{Theorem}[section]

\newtheorem{defn}{Definition}[section]
\newtheorem{lemma}{Lemma}[section]

\usepackage{fullpage}
\usepackage[T1]{fontenc}
\usepackage{hyperref}
\usepackage{xcolor}
\definecolor{link-color}{rgb}{0.15,0.4,0.15}
\hypersetup{
  colorlinks,
  linkcolor = {link-color}, citecolor = {link-color},
}

\usepackage{graphicx}
\usepackage{hyperref}
\usepackage{amsmath,amsthm,amssymb}
\usepackage{bbm}
\usepackage{tabularx}

\newcommand{\R}{\mathbb{R}}
\newcommand{\N}{\mathbb{N}}
\renewcommand{\P}{\mathbb{P}}
\newcommand{\E}{\mathbb{E}}
\newcommand{\1}{\mathbbm{1}}

\DeclareMathOperator{\p}{\mbox{\rm I\hspace{-0.02in}P}}

    \def\d{{\textnormal d}}

\newcommand{\iu}{\mathrm{i}} 

\newenvironment{eqnarr}{\begin{IEEEeqnarray}{rCl}}{\end{IEEEeqnarray}\ignorespacesafterend}

\renewcommand{\eqref}[1]{\hyperref[#1]{(\ref*{#1})}}

\newcommand{\rhohat}{\hat{\rho}}

\newcommand{\dd}{\mathrm{d}}

    \def\beq{\begin{eqnarr}}
    \def\eeq{\end{eqnarr}}
    \def\beqq{\begin{eqnarray*}} 
    \def\eeqq{\end{eqnarray*}} 

    \def\p{{\mathbb P}}

        \def\d{{\rm d}}

    \def\d{{\textnormal d}}

\newcommand*{\pref}[1]{\hyperref[#1]{(\ref*{#1})}}
\newcommand*{\refpref}[2]{\hyperref[#2]{\ref*{#1}(\ref*{#2})}}

\startlocaldefs
\numberwithin{equation}{section}
\theoremstyle{plain}

\endlocaldefs

\begin{document}

\begin{frontmatter}
\title{Entrance and exit at infinity \\ for stable jump diffusions}

\runtitle{Entrance of stable SDE from the infinite boundary}

\begin{aug}

\author{\fnms Leif D\"oring\thanksref{t2}
\ead[label=e2]{doering@uni-mannheim.de}}
\and
\author{\fnms{Andreas E. Kyprianou}\thanksref{t3}\ead[label=e1]{a.kyprianou@bath.ac.uk}}

\thankstext{t3}{Supported by EPSRC grants EP/L002442/1 and EP/M001784/1}

\thankstext{t2}{Supported by Ambizione Grant of the Swiss Science Foundation}

\affiliation{University of Mannheim, University of Bath}

\address{
University of Bath\\
Department of Mathematical Sciences \\
Bath, BA2 7AY\\
 UK.
\printead{e1}
}

\address{University of Mannheim\\
Institute of Mathematics\\
68131 Mannheim\\
Germany. 
\printead{e2}
}
\end{aug}

\begin{abstract}\hspace{0.1cm}
In his seminal work from the 1950s, William Feller classified all one-dimensional diffusions on $-\infty\leq a<b\leq \infty$ in terms of their ability to access the boundary (Feller's test for explosions) and to enter the interior  from the boundary. Feller's technique is restricted to diffusion processes as the corresponding differential generators allow  explicit computations and the use of Hille-Yosida theory. In the present article we study exit and entrance from infinity for the most natural generalization, that is, jump diffusions of the form
\begin{align*}
	dZ_t=\sigma(Z_{t-})\,dX_t, 
\end{align*}
driven by stable L\'evy processes for $\alpha\in (0,2)$. Many results have been proved for jump diffusions, employing a variety of techniques developed after Feller's work but exit and entrance from infinite boundaries has long remained open. We show that the presence of jumps implies features not seen in the diffusive setting without drift. Finite time explosion is possible for $\alpha\in (0,1)$, whereas entrance from different kinds of infinity is possible for $\alpha\in [1,2)$. Accordingly we derive necessary and sufficient conditions on $\sigma$.

Our proofs are based on very recent developments for path transformations of stable  processes via the Lamperti--Kiu representation and new Wiener--Hopf factorisations for L\'evy processes that lie therein. The arguments draw together original and intricate applications of results using the Riesz--Bogdan--\.Zak transformation, entrance laws for self-similar Markov processes, perpetual integrals of L\'evy processes and fluctuation theory, which have not been used before in the SDE setting, thereby allowing us to employ classical theory such as Hunt--Nagasawa duality and Getoor's characterisation of transience and recurrence.


\end{abstract}

\begin{keyword}[class=MSC]
\kwd[Primary ]{60H20, 60G52}
\end{keyword}

\begin{keyword}
\kwd{SDEs, entrance, explosion, Kelvin transform, duality, stable L\'evy processes}
\end{keyword}

\end{frontmatter}
\tableofcontents

\section{Introduction}
	 In his seminal work in the 1950s, Feller \cite{feller1, feller2}  classified one-dimensional diffusion processes and their boundary behaviour on an interval $[a,b]$ with  $-\infty\leq a<b\leq \infty$. Feller identified four types of boundaries of the domain. The definition of each is given in terms of combinations of two fundamental properties (or the absence thereof), namely accessibility, i.e. reachable in finite time from within $(a,b)$, and enterability, i.e. the diffusion started at that point can enter $(a, b)$. The four types of boundary points are: regular,
if it is both accessible and enterable; exit, if it is accessible but not enterable; entrance, if it is enterable but not accessible; natural if it is neither accessible nor enterable. Feller's definitions and proofs are purely analytic, using Hille-Yosida theory to characterize all possible subdomains of $C([a,b])$, the space of continuous functions on $[a,b]$, for second order differential operators $\mathcal A :=\kappa \frac{d}{dx}+\frac{\sigma^2}{2} \frac{d^2}{dx^2}$ to generate a Feller semigroup. Feller's study can be recovered probabilistically using stochastic differential equations (SDEs) and excursion theory to construct so-called sticky boundary behavior; a historical summary can be found in \cite{peskir}. In the present article we will not discuss sticky behavior so we focus on SDEs of the form
	\begin{align}\label{BB}
		dZ_t=\kappa(Z_t)\,{dt}+\sigma(Z_t)\,dB_t, \quad Z_0=z\in \R,
	\end{align}
where  $(B_t,t\geq 0)$ is a standard Brownian motion. A simple change of space allows to simplify the degree of generality in the choices of $\kappa$. Indeed, transforming space with the so-called scale function allows a reduction of \eqref{BB} to the driftless SDE
	\begin{align}\label{B}
		dZ_t=\tilde\sigma(Z_t)\,dB_t,\quad Z_0=z\in\R,
	\end{align}
	on a new interval $(\tilde a, \tilde b)$. In the setting of the entire real line, i.e. $a=-\infty$ and $b=+\infty$, the notion of entrance (in applications also called coming down from infinity) and exit (explosion) for \eqref{B} becomes interesting as they necessitate the range of the diffusion to be infinite over an almost surely finite period of time, a property not seen for the Brownian motion alone. It is a standard property (random time-change of a recurrent process) that solutions to \eqref{B} cannot explode in finite time, hence, neither $+\infty$ nor $-\infty$ are accessible. This can also be verified by plugging-into Feller's test for explosions, see for instance Karatzas and Shreve \cite{KaratzasShreve}, Section 5.5.C. On the other hand, depending on the growth of $\sigma$ at infinity the infinite boundary points can be of entrance type. Feller's results for this scenario imply that $+\infty$ is an entrance boundary if and only if  
	\begin{align}\label{Test}
		\int^{+\infty} x\,\sigma(x)^{-2}\,{dx}<\infty,
	\end{align}
	i.e. $\sigma$ growth slightly more than linearly at infinity. An analogous integral test at $-\infty$ holds in the case that $-\infty$ is an entrance point.\smallskip

	In the present article we study a new type of boundary behaviour, namely simultaneously exit and entrance from $+\infty$ and $-\infty${\color{black}; See Figure \ref{fig1}}. The simultaneous infinite boundary point will be denoted by $\pm\infty$. We define entrance (resp. explosion at a finite random time $T$) from $\pm\infty$ if almost surely {liminf}$_{t\downarrow 0} Z_t=-\infty$ and {limsup}$_{t\downarrow 0} Z_t=+\infty$ (resp. {liminf}$_{t\uparrow T} Z_t=-\infty$ and {limsup}$_{t\uparrow T} Z_t=+\infty$). Entrance and exit at $\pm\infty$ are forced by an alternation of increasingly big jumps that avoid compact sets in $\R$. 
\begin{figure}[h]
        \includegraphics[scale=0.5]{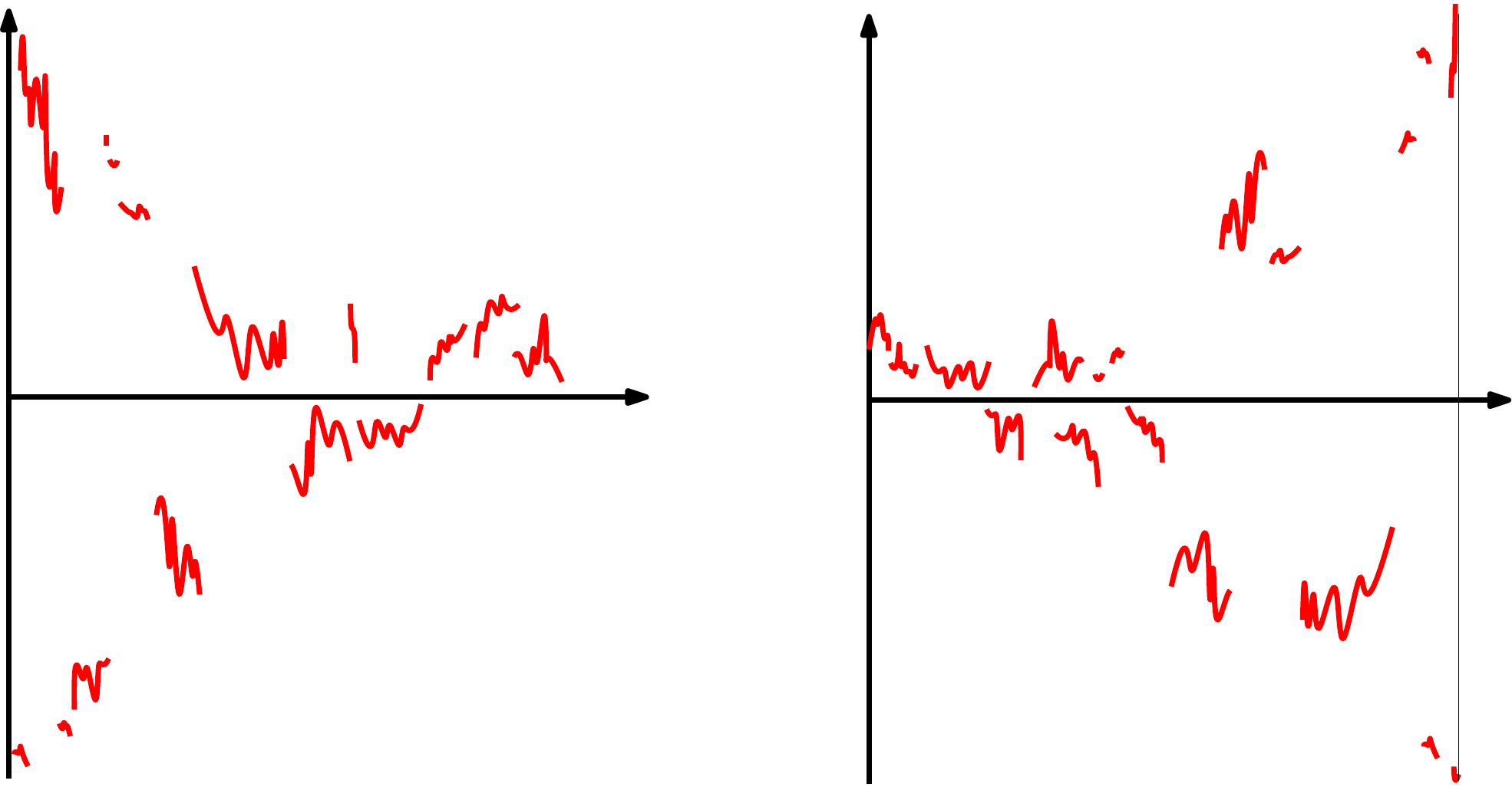}
\caption{Entrance from $\pm \infty$ and exit at $\pm\infty$}
\label{fig1}
\end{figure}

	We focus our study on so-called stable jump diffusions, i.e. stochastic differential equations
	      	\begin{align}\label{2}
		dZ_t=\sigma(Z_{t-})\,dX_t,\quad Z_0=z\in\R,
	\end{align}
driven by a  stable L\'evy process $  (X_t,t\geq 0)$ with index $\alpha\in (0,2)$ up to a (possibly infinite) explosion time. The boundary case $\alpha=2$ corresponds to the Brownian case studied by Feller. More precisely, we derive necessary and sufficient conditions on $\sigma$ so that (i) non-exploding solutions exist and (ii) the corresponding transition semigroup of $Z$ extends to an entrance point at `infinity' in an appropriate Fellerian way.

\section{Main results}
Before stating the results let us clarify our notation. A stable process is a L\'evy process with the additional property that, for all $c>0$ and $x\in\mathbb{R}$, 
\[
(cX_{c^{-\alpha}t}, t\geq 0) \text{ under } \mathbb{P}_x \text{ is equal in law to } (X_t, t\geq 0) \text{ under }\mathbb{P}_{cx},
\]
where $(\mathbb{P}_x, x\in\mathbb{R})$ are the probabilities of $X$ and $\alpha\in(0,2)$. As a L\'evy process, a stable process is a Feller process and the semigroup of $X$ is entirely characterised by its characteristic exponent. More precisely,  $\Psi(z): = -\log
\mathbb{E}\left[ \mathrm{e}^{\mathrm{i}{z}   X_{1}}\right]$ satisfies
 \begin{equation}
 \Psi(z) = 
 |z|^{\alpha} \left( e^{\pi {\mathrm i} \alpha (\frac{1}{2}-\rho)} {\bf 1}_{\{z>0\}}+ 
  e^{-\pi {\mathrm i} \alpha (\frac{1}{2}-\rho)} {\bf 1}_{\{z<0\}}\right), \quad z\in\mathbb{R},
 \label{Psi_alpha_rho_parameterization_process}
 \end{equation}
where we have reserved the special notation  $\mathbb{P}$, with expectation operator $\mathbb{E}$, for the law of $X$ when issued from the origin and 
$\rho={\mathbb P}(X_1>0)$
is the positivity parameter. 
The L\'evy measure associated with $\Psi$ can be written in the form 
\begin{equation}
\Pi({dx})/dx = \Gamma(1+\alpha)  \frac{\sin(\pi \alpha \rho)}{\pi} \frac{1}{x^{1+\alpha}}\mathbf{1}_{(x>0)}+
 \Gamma(1+\alpha)  \frac{\sin(\pi \alpha\hat\rho)}{\pi}\frac{1}{|x|^{1+\alpha}}\mathbf{1}_{(x<0)}, \quad x\in\mathbb{R},
\label{jumpmeas}
\end{equation}
where $\hat\rho := 1-\rho$. In the case that $\alpha = 1$, we take $\rho = 1/2$, meaning that $X$ corresponds to the Cauchy process. If $X$ has only upwards (resp. downwards) jumps we say $X$ is spectrally positive (resp. negative). If $X$ has jumps in both directions we say $X$ is two-sided. A spectrally positive (resp. negative) stable process with $\alpha\in(0,1)$ is necessarily increasing (resp. decreasing). See for example the recent review Kyprianou \cite{KALEA}  for more on this parametric  classification of stable processes. %
\smallskip

For a driving stable process $X$ on a filtered probability space $(\Omega, \mathcal A, \mathcal F_t,\P)$ and an initial value $z\in\R$, a stochastic process $Z$ on $(\Omega, \mathcal A,\P)$ is called a solution to \eqref{2} up to an explosion time if $Z$ is $\mathcal F_t$-adapted, has almost surely c\`adl\`ag  sample paths and, with $T^n=\inf\{t: |Z_s|\geq n\}$, the stopped integral equation
\begin{align}\label{2b}
	Z_{t\wedge T^n}=z+\int_0^{t \wedge T^n} \sigma(Z_{s-})\,dX_s,\quad t\geq 0,
\end{align}
is satisfied almost surely for all $n\in \N$. We denote by $T=\lim_{n\to \infty} T^n$ the (finite or infinite) explosion time of the solution. We note that with this notion, a solution $Z$ is  a 
`local solution on $(-n,n)$' (in the sense of Zanzotto \cite{z2} or \cite{z}) for all $n\in\N$.\smallskip

When $\alpha\in(1,2)$, the study of weak existence and uniqueness of solutions to \eqref{2} (resp. \eqref{2b}) in $\R$ is due to Zanzotto \cite{z}, complementing the classical Engelbert--Schmidt theory for one-dimensional Brownian SDEs (see Chapter 5.5 of \cite{KaratzasShreve}). In fact, the main difficulty is to understand existence and uniqueness for solutions at zeros of $\sigma$. The focus of the present article lies on finite time explosion and entrance from infinity, so we always work under the following simplifying assumption that avoids all difficulties in the interior of $\R$.
\begin{assumption}\label{A}
	$\sigma$ is continuous and strictly positive.
\end{assumption}
Time-change techniques are a useful tool in the study of one-dimensional diffusions, see for instance Karatzas and Shreve \cite{KaratzasShreve}, Section 5.5.A. For stable SDEs time-change was the main tool in the study of Zanzotto \cite{z2}, \cite{z}. Under weaker assumptions than our Assumption \ref{A}, Zanzotto proved that for all $n\in \N$ there is a unique local solution on $(-n,n)$ so that  $(Z_t, t\leq T^n)=(X_{\tau_t},t\leq T^n)$ in distribution, where $\tau_t = \inf\{s>0 : \int_0^s \sigma(X_s)^{-\alpha}ds > t\}$. Note that continuity of $\sigma>0$ on $\R$ implies that $\sigma$ is bounded away from zero on all intervals $(-n,n)$. Local solutions have a simple consistency property. For $m>n$, a local solution on $(-m,m)$ stopped at $T^n$ is a local solution on $(-n,n)$. Hence, we immediately obtain the following time-change representation of (possibly exploding) solutions.
\begin{prop}\label{pr}
	Suppose that $\sigma$ satisfies Assumption \ref{A} and $z\in\R$. Then there is a unique (possibly exploding) weak solution $Z$ to the SDE \eqref{2} and $Z$ can be expressed as time-change under $\P_z$ via
	\begin{align}\label{timechangesolution}
		Z_t:=X_{\tau_t},\quad t<T, 
	\end{align}
	where 
	\begin{align}\label{7}
			\tau_t = \inf\left\{s>0 : \int_0^s \sigma(X_s)^{-\alpha}ds > t\right\}
	\end{align}
	and the finite or infinite explosion time is $T= \int_0^\infty \sigma(X_s)^{-\alpha}ds$. 
\end{prop}

Henceforth, the law of the unique solution $Z$ as a process on $\mathbb{D}([0,\infty), \mathbb{R})$  will be denoted by $\emph{\rm P}_z$, $z\in\mathbb{R}$, where $\mathbb{D}([0,\infty), \mathbb{R})$ is the space of c\`adl\`ag paths mapping $[0,\infty)$ to $\mathbb{R}$, equipped with the Borel $\sigma$-algebra induced by the Skorokhod topology. \textcolor{black}{We call a finite-time explosion a Feller explosion if the explosion time T is weakly continuous in the Skorokhod topology with respect to the initial condition  on $\mathbb{R}$ and T converges weakly to zero as $|x|\to\infty$.
}

\smallskip

In the following two sections we present and discuss tests for \textcolor{black}{Feller} explosion and \textcolor{black}{Feller} entrance from infinity. All proofs are based solely on the time-change representation \eqref{timechangesolution} for solutions of the jump diffusion \eqref{2}, no further SDE calculus is used. \textcolor{black}{The main focus of our constructions lies on entrance from infinity.}

\subsection{(Non-)Explosion of stable jump diffusions}\label{S1}
 In theory, the question of finite time explosion could be resolved immediately from \eqref{timechangesolution} and \eqref{7} by studying finiteness vs. infiniteness of the so-called perpetual integral  $\int_0^\infty \sigma(X_s)^{-\alpha}ds$ for the stable process $X$. This is trivial for $\alpha\geq 1$ and the Brownian case due to (set-)recurrence of $X$. For $\alpha\in(0,1)$ the transience of $X$ implies that finiteness of $\int_0^\infty \sigma(X_s)^{-\alpha}ds$ will depend on the growth of $\sigma$ at infinity. Except for a general 0-1 law for perpetual integrals of L\'evy processes, which implies that finite time explosion is an event of probability $0$ or $1$ (see Lemma 5 of \cite{DK}, the stronger assumptions of \cite{DK} are not used for the 0-1 law), we are not aware of a sufficiently general result for perpetual integrals that is helpful in this respect. 
 Our first main theorem gives necessary and sufficient conditions for \textcolor{black}{ Feller explosion } of stable SDEs \eqref{2} and identifies the infinite almost sure limit at the explosion time which is either $+\infty$, $-\infty$ or $\pm\infty$. Divergence of the solution $Z$ to $\pm\infty$ at the explosion time means $\limsup_{t\uparrow T} Z_t=+\infty$ and $\liminf_{t\uparrow T} Z_t=-\infty$ almost surely.\smallskip

In terms of other work that are in close proximity to our own, we are only aware of the recent article Li \cite{Pei} for continuous state polynomial branching processes which coincides with our Theorem \ref{zthrm} below for polynomials $\sigma(x)=x^\theta$ and spectrally positive driving stable process. However, the use of technology for branching processes excludes generalizations of that article to two-sided jumps.

\smallskip

 In the {\color{black}Table 1} a tick stands for \textcolor{black}{Feller} explosion to the corresponding infinite boundary point, a cross for almost sure non-explosion. We use the symbols $\uparrow$, $\downarrow$ and $\uparrow \& \downarrow$ to indicate the direction of jumps of the driving stable process. The table is complemented with a final row ($\alpha=2$) representing Feller's test for explosions for Brownian SDEs.
	\begin{theorem}\label{zthrm} 
Suppose that $\sigma$ satisfies Assumption \ref{A} and let
\begin{align*}
	I^{\sigma,\alpha}(A) = \int_A \sigma(x)^{-\alpha}|x|^{\alpha -1}d x.
\end{align*}
Then {\color{black}Table 1} exhaustively summarises \textcolor{black}{Feller} explosion for the SDE \eqref{2} issued from any $z\in \R$, depending only on $\alpha, \sigma$ and the directions of jumps of the stable driving L\'evy process.
\begin{footnotesize}
\begin{table}[h!]
\label{table1}
\caption{\rm Necessary and sufficient conditions for exit at infinite boundary points}
 \hspace{-0.4cm}
\begin{tabular}{|c|c | l| l| l|}
\hline
$\alpha$ &\scriptsize{Jumps}&$+\infty$&$-\infty$&$\pm\infty$\\
\hline\hline
 &only $\downarrow$& \xmark &
\cmark \scriptsize{ iff } $I^{\sigma,\alpha}(\R_-)<\infty$
&\xmark \\
$<1$&only $\uparrow$& \cmark \scriptsize{ iff } $I^{\sigma,\alpha}(\R_+)<\infty $
&\xmark&\xmark \\
&$\uparrow \& \downarrow$&\xmark 
&\xmark &\cmark \scriptsize{ iff } $I^{\sigma,\alpha}(\R)<\infty$\\
\hline\hline
$\scriptsize{=1}$&$\uparrow  \& \downarrow$&\xmark &\xmark &\xmark \\
\hline\hline
 &only $\downarrow$&\xmark & \xmark&\xmark \\
$>1$&only $\uparrow$&\xmark&\xmark  &\xmark \\
& $\uparrow \& \downarrow$&\xmark&\xmark& \xmark \\
\hline \hline
$=2$&\text{none}&\xmark&\xmark& \xmark \\
\hline
\end{tabular}

\end{table}

\end{footnotesize}  



 \end{theorem}

\subsection{Entrance from infinity}\label{S2} 
After the characterization of infinite exit points we continue with the characterization of entrance from infinity. By analogy to the three types of infinite boundary points for explosion, we distinguish entrance points $+\infty$, $-\infty$ and $\pm\infty$. Alternating entrance from $\pm\infty$ is a new phenomenon. A rigorous formulation will be given in terms of semigroup extensions under which trajectories enter continuously from infinity. Although  in the spirit of Feller's work  our construction is completely different. Feller constructed semigroups through the Hille-Yosida theorem (which gives a Markov process through the Riesz representation theorem) whereas we give explicit probabilistic constructions and then prove the corresponding semigroup is Feller. It is not clear if the Hille-Yosida approach for diffusions can be extended to jump diffusions as it involves the need to understand the resolvent equations $(\mathcal A-\lambda I)f=g$. Those are ordinary differential equations for the diffusive case and can be solved using the variation of constants formula. For jump diffusions the resolvent equations are intego-differential equations for which explicit solutions are not available.\smallskip
	

 Suppose that $S$ is a locally compact metrizable topological space. 
 We write $C_b(S)$ for the space of bounded continuous functions mapping $S$ to $\R$. Then $C_b(S)$ is a Banach space with the supremum norm $||\cdot||$. 
 \begin{defn}\rm \label{Fellersemigroupdef}
 A $C_b$-Feller semigroup is a collection of linear operators $\mathcal{P}=(\mathcal{P}_t, t\geq 0)$ mapping $C_b(S)$ into $C_b(S)$ satisfying
\begin{itemize}
	\item[(i)] $\mathcal{P}_t 1\leq 1$ for all $t\geq 0$ (contraction),
	\item[(ii)] $\mathcal{P}_t f\geq 0$ for all $f\geq 0$ and $t\geq 0$ (positivity),
	\item[(iii)] $\mathcal{P}_0={\rm id}$ and $\mathcal{P}_{t+s}=\mathcal{P}_t \mathcal{P}_s$ for all $t,s\geq 0$ (semigroup),
	\item[(iv)] $\lim_{t\to 0}\mathcal{P}_t f(x)=f({x})$ for all $f\in C_b(S)$ and ${x}\in S$ (weak continuity).
\end{itemize}
Additionally, $\mathcal P$ is called conservative (or not killed) if
\begin{itemize}
	\item[(i')] $\mathcal{P}_t 1=1$ for all $t\geq 0$.
\end{itemize}
\end{defn}
Semigroups are the natural language with which to describe the transitions of a Markov process. A (possibly killed) Markov process $(Y_t,t\geq 0)$ on $S$ with cemetary state $\Delta\notin S$ is a collection $(\texttt{P}_{y},{y}\in S)$ of probability laws on the c\`adl\`ag trajectories $\mathbb{D}([0,\infty), S\cup \{\Delta\})$, mapping $[0,\infty)$ to $S\cup \{\Delta\}$, equipped with the Borel $\sigma$-algebra induced by the Skorokhod topology, such that the canonical process $Y_t(\omega):=\omega_t$, $t\geq 0$, is absorbed at $\Delta$ and satisfies
\[
	\texttt{E}_{y}[f(Y_t)\,|\sigma(Y_u,u\leq s)]=\texttt{E}_{y}[f(Y_t)\,|\sigma(Y_s)],\qquad \texttt{P}_{y}\text{-almost surely},
\]
 for all ${y}\in S$, $0\leq s\leq t$ and $f\in C_b(S)$. If $\mathcal P$ is conservative, then the killing time is infinite almost surely. If we define from a Markov process $(\texttt{P}_{y},{y}\in S)$ the so-called transition operators
\begin{equation}
	\mathcal{P}_{t} f({y}): = \texttt{E}_{y}[f(Y_t)], \qquad t\geq 0, {y}\in S, f\in C_b(S),
\label{semigroupMarkov}
\end{equation}
then conditions (i)-(iii) hold. However, it is not necessarily the case that $\mathcal P_tf$ is continuous and (iv) holds. Conversely, for a given Feller semigroup $\mathcal{P}$ there is a (possibly killed) strong Markov process $(\mathcal P_x:{x\in S})$ on $S$ with transition semigroup $\mathcal{P}$ in the sense of \eqref{semigroupMarkov}. In that case we refer to $Y$ as a (conservative) Feller process. We refer the reader for instance to Chapter 17 of Kallenberg \cite{Kallenberg} for a full account of the theory.\smallskip

The main finding of this article is that there are three types of infinite entrance boundaries under the presence of jumps. In this respect, let us denote
\begin{align}
	\overline{\R} := \mathbb{R}\cup\{\infty\}, \quad \underline{\R} : = \mathbb{R}\cup\{-\infty\}\quad \text{and }\quad\underline{\overline{\R}} : = \mathbb{R}\cup\{\pm\infty\}
	\label{3R}
\end{align}
with the usual extensions of the Eucledian topology, i.e. the smallest topology containing all open sets of $\mathbb{R}$ and sets 
\begin{align}\label{4R}
	(c,+\infty] \text{ for } \overline \R,\quad [-\infty,c)\text{ for }\underline \R \quad \text{ and }\quad [-\infty, c)\cup (d,+\infty]\text{ for }\underline{\overline{\R}}.
\end{align}	
Note that all these sets are metrizable as they are homeomorphic to intervals. It will later play a role that in this way $\underline{\overline{\R}}$ is the one-point compactification of $\R$. 
\begin{defn}\rm\label{Fellerdef}
We say that $+\infty$ is a (continuous) entrance point for a Feller process $({\texttt P}_x:{x\in \R})$ if there is an extension $({\texttt P}_x:{x\in \overline \R})$ on the Skorokhod space, specifically, meaning Skorokhod  continuity in the initial position, so that
\begin{itemize}
\item[(i)] the point $+\infty$ is not accessible under ${\texttt P}_x$ for all $x\in\R$,
\item[(ii)] the corresponding transition semigroup $\mathcal P$ is Feller on $C_b(\overline{\R})$,
\item[(iii)] there is continuous entrance in the sense that ${\texttt P}_{+\infty}(\lim_{t\downarrow 0} Y_t=+\infty)=1$.
\end{itemize}
Analogously, we define entrance from $-\infty$ as extension to $C_b(\underline{\R})$ and entrance from $\pm\infty$ as extension to $C_b(\underline{\overline{\R}})=C(\underline{\overline{\R}})$.
\end{defn}

Our next result extends Feller's characterization of infinite (continuous) entrance points to stable jump diffusions. In {\color{black} Table 2} below a tick stands for entrance from the corresponding infinite boundary point, a cross for no entrance point. We use the symbols $\uparrow$, $\downarrow$ and $\uparrow \& \downarrow$ to indicate the direction of jumps of the driving stable process. The table is complemented with a final row representing Feller's criterion for entrance from infinity in the Brownian case. We also note that, when the driving noise only has positive jumps, the necessary and sufficient condition in the table is a special form of the one given by (1.25) in \cite{Pei}, where some other equivalent conditions are also given. 
\begin{theorem}\label{main}
Suppose that $\sigma$ satisfies Assumption \ref{A} and let
\begin{align*}
	I^{\sigma,\alpha}(A) = \int_A \sigma(x)^{-\alpha}|x|^{\alpha -1}d x \quad\text{ and } \quad I^{\sigma,1}= \int_\R \sigma(x)^{-1}\log |x|d x.
\end{align*}
Then {\color{black} Table 2}  exhaustively summarizes entrance points at infinity depending only on $\alpha, \sigma$ and the directions of jumps of the stable driving L\'evy process.
\begin{footnotesize}
\begin{table}[h!]
\label{table2}
\caption{\rm Necessary and sufficient conditions for entrance from infinite boundary points}
 \hspace{-0.4cm}
\begin{tabular}{|c| c ||l| c || l|c || l|c |}
\hline
$\alpha$ &\scriptsize{Jumps}&$+\infty$&\scriptsize{Proof}& $-\infty$&\scriptsize{Proof}&$\pm\infty$&\scriptsize{Proof}\\
\hline\hline
 &only $\downarrow$& \xmark &\tiny{(\S \ref{proof6})}&
\xmark 
& \tiny{(\S \ref{proof8})}
&\xmark &\tiny{(\S \ref{proof6})}\\
$<1$&only $\uparrow$& \xmark 
& \tiny{(\S \ref{proof4.5})}
&\xmark & \tiny{(\S \ref{proof4})}&\xmark &\tiny{(\S \ref{proof4})}\\
&$\uparrow \& \downarrow$&\xmark& \tiny{(\S \ref{proof0})}&\xmark &\tiny{(\S \ref{proof1})}&\xmark &\tiny{(\S \ref{proof2})}\\
\hline\hline
$\scriptsize{=1}$&$\uparrow  \& \downarrow$&\xmark &\tiny{(\S \ref{proof1})}&\xmark &\tiny{(\S \ref{proof1})}&\cmark\,\scriptsize{ iff $ I^{\sigma,1}<\infty$} &\tiny{(\S \ref{a=1})}
\\
\hline\hline
 &only $\downarrow$&\xmark &\tiny{(\S \ref{proof6})}& \cmark\,\scriptsize{ iff  $I^{\sigma,\alpha}(\R_-)<\infty$ }&\tiny{(\S \ref{proof5})}
&\xmark &\tiny{(\S \ref{proof6})}\\
$>1$&only $\uparrow$&\cmark\, \scriptsize{iff $I^{\sigma,\alpha}(\R_+)<\infty$}& \tiny{(\S \ref{proof5})}
&\xmark  &\tiny{(\S \ref{proof4})}&\xmark & \tiny{(\S \ref{proof4})}\\
& $\uparrow \& \downarrow$&\xmark& \tiny{(\S \ref{proof0})}&\xmark& \tiny{(\S \ref{proof1})}& \cmark\,\scriptsize{ iff 
$I^{\sigma,\alpha}(\R)<\infty$} &\tiny{(\S \ref{proof3})}
\\
\hline\hline
\scriptsize{$=2$}\text{  }&\text{ none }& \cmark\, \scriptsize{iff $I^{\sigma,2}(\R_+)<\infty$} &\text{  Feller }&\cmark\, \scriptsize{iff $I^{\sigma,2}(\R_-)<\infty$}& \text{  Feller }
&\xmark \quad\quad\quad\quad\quad\quad\quad\text{ }&\text{  ------ }\\
\hline
\end{tabular}
\end{table}
\end{footnotesize}  
\end{theorem}
	Without loss of generality, throughout the article we will study entrance from infinity for the SDE \eqref{2} killed upon first hitting the origin, denoted by $Z^\dagger$. The time-change representation from Proposition \ref{pr} holds unchanged, replacing the stable process $X$ by the stable process killed at the origin $X^\dagger$. The additional killing is crucial to apply stochastic potential theory (killing makes solutions transient) but does not restrict the generality of our results for the following reasons.
	\smallskip
	
	(i) If $\alpha\leq 1$, then solutions almost surely do not hit the origin, hence, no killing occurs. This is a consequence of the time-change representation \eqref{timechangesolution} and the fact that points are polar for stable processes with $\alpha\leq 1$.
	\smallskip
	
	(ii) If $\alpha>1$, then solutions to \eqref{2} might be killed at zero in finite time. For all initial conditions, solutions are weakly unique, non-explosive and known to be $C_b$-Feller on $\R$ (see for instance van Casteren \cite{Cas}, but note that his statement is stronger than his proofs and ${\mathcal P}_tf$ does not necessarily vanish at infinity). To construct a Markov process without killing at $0$ from the killed solution, one proceeds as follows. Take the killed process up to the killing time and `glue' thereafter a new unkilled solution with $Z_0=0$. The reader should keep in mind that constructing Markov processes by `glueing' two processes is far from easy and the literature is limited. For continuous processes we refer the reader to Nagasawa \cite{N2}, the gluing results needed for the present article can be found for instance in Werner \cite{Flo}, Theorem 1.6. The process obtained by gluing is not only Markov but also Feller, which follows directly from inspecting the resolvent operator of the Markov process obtained by gluing.

\subsection{Proof strategy for entrance from infinity}
	
	{\it (i) Direct arguments using the time-change representation \eqref{timechangesolution}.} No entrance from infinity in the impossible cases is argued as follows. If the stable process itself `diverges at fixed levels' for large starting conditions (i.e. does not hit intervals or has diverging overshoots), then the time-change in the representation \eqref{pr} cannot prevent  solutions to the SDE \eqref{2} having the same property. Such arguments explain all the crosses in the tables.
	
	\smallskip

	{\it (ii) Spatial-inversion.} To construct the semigroup extensions at infinity we proceed similarly in all cases. In Section \ref{sec_transforms} we prove an extension of the so-called Riesz--Bogdan--\.{Z}ak transformation, which we then use to write solutions to \eqref{2} as a time-change of the spatial inversion $x\mapsto 1/x$ of a certain $h$-transformed process $\hat X^\circ$. As such, starting the SDE from infinity is equivalent to starting $\hat X^\circ$ from $0$ and insisting on a suitably well behaved time-change. For the first of these two,   we can reduce the entrance of the auxiliary process $\hat X^\circ$ at 0 to recent result on self-similar Markov processes, the time-change can be controlled using recent explicit potential formulas.\smallskip
	
	{\it (iii) Time-reversal.} If infinity is an entrance point, we show (using the strong Markov property as a consequence of the Feller property) that $0$ for $\alpha>1$ (resp. $(-1,1)$ for $\alpha=1$) is hit in finite time. For $\alpha>1$ we use time-reversal to derive the integral test from a perpetual integral of a well-behaved L\'evy process. For $\alpha=1$ we apply an extended version of a transience result due to Getoor, see the Appendix.
	
\subsection{Proof strategy for explosion}
All recurrent cases are easily dealt with since the explosion time $T= \int_0^\infty \sigma(X_s)^{-\alpha}ds$ is obviously infinite almost surely if $X$ is recurrent. Since explosion is equivalent to entrance from $0$ of the space-inverted time-reversal (which can be identified as a time-change of the stable process itself), in the transient case $\alpha\in(0,1)$ we can work with a transience argument for $X$.

\section{Self-similar Markov processes and stable processes}\label{background}
The techniques we use to prove Theorem \ref{main} make significant use of the fact that  the driving stable L\'evy process in the SDE \eqref{2} is also a self-similar Markov process. We will  use a lot of facts and theory that have only very recently been developed in the field of self-similar Markov processes. As our desire is to keep the article mostly self-contained, we devote this section  to a brief overview of the recent results  that are needed. In particular, we look at how the theory of self-similar Markov processes plays  into the setting of stable L\'evy processes.

	\smallskip

	The reader will quickly realise that  there are many different types of processes that are involved in our analysis, least of all in this section. For this reason, we include as an annex at the end of this article, a glossary of mathematical symbols.

\subsection{Positive self-similar Markov processes}\label{pssMpsect}
A regular strong Markov family ${P}_z$, $z> 0$, with c\`adl\`ag paths on the state space $(0,\infty)$, with $0$ being an absorbing cemetery state,
 is called positive self-similar Markov process of index $\alpha>0$ (briefly pssMp) if the scaling property holds
 \begin{align}\label{self_sim}
   \text{The law of $(c\mathcal{X}_{c^{-\alpha}t},t\geq 0)$ under ${P}_z$ is ${P}_{cz}$},
 \end{align}
for all $z, c>0$. The analysis of positive self-similar processes is fundamentally based on the seminal work of Lamperti \cite{L72}  (see also  Chapter 13 of \cite{Kbook} for an overview). Lamperti's result gives a bijection between the class of pssMps and the class of L\'evy processes, possibly killed at an independent exponential time with cemetery state $-\infty$, such that, under ${P}_z$, $z>0$,
\begin{align}
\label{pssMpLamperti}
	\mathcal{X}_t=\exp(\xi_{\varphi_t}),\qquad t\leq I_\infty : = \int_0^\infty \exp(\alpha\xi_u)du,
\end{align}
where $\varphi_t=\inf\{s>0: \int_0^s \exp(\alpha \xi_u)du >t \}$ and the L\'evy process $\xi$  is started in $\log z$.

	 \smallskip

 It is a consequence of the Lamperti representation \eqref{pssMpLamperti} that pssMps can be split into conservative and non-conservative regimes. If $\zeta$ denotes the first hitting time of 0 by $\mathcal{X}$, then
\begin{align}\label{999}
\begin{split}
	 \quad\,\,\,{P}_z(\zeta <\infty)=1\text{ for all }z>0\quad &\Longleftrightarrow\quad\xi \text{ drifts to }-\infty\text{ or is killed},\\
	\quad\,\,\,{P}_z(\zeta<\infty)=0\text{ for all }z>0\quad&\Longleftrightarrow\quad\xi\text{ drifts to }+\infty \text{ or oscillates}.
	\end{split}
\end{align}
The dichotomy \eqref{999} can be used to decide if a pssMps is transient or recurrent by examining the corresponding L\'evy process $\xi$. We will see this methodology employed in later sections.\smallskip

%
%

For the present article we shall need one of several continuations of Lamperti's work. For the conservative case ${P}_z(\zeta<\infty)=0$, an important question to ask is: When is it possible to treat $0$ as an entrance point? More precisely, one asks for a Feller extension $({P}_z,z\geq 0)$ of $({P}_z,z> 0)$. It  was shown incrementally in Bertoin and Yor \cite{BY02},  Caballero and Chaumont \cite{CC}, Chaumont et al. \cite{CKPR} and also in Bertoin and Savov \cite{BS} that, if the ascending ladder height process of $\xi$ is non-lattice, $0$ is an entrance point for $\mathcal{X}$ if and only if  the overshoot distribution of $\xi$ over asymptotically large levels converges. That is to say, if $(\mathbf{P}_x, x\in\R)$ are the distributions of $\xi$ and $\bf P=\bf P_0$, then $0$ is an entrance point for $\mathcal{X}$ if and only if
\begin{align}\label{C2}
	\lim_{x\uparrow \infty} {\bf P}(\xi_{\varsigma^+_x}-x\in dy), \qquad y\geq 0,
\end{align}
exists in the sense of weak convergence, where $\varsigma^+_x:=\inf\{t>0:\xi_t\geq x\}$. If \eqref{C2} holds then one says the L\'evy process $\xi$ has stationary overshoots. The probabilistic condition which is equivalent to stationary overshoots is complicated to verify directly but has an explicit analytic counterpart in terms of the L\'evy triplet (see for instance Chapter 7 of \cite{Kbook}). In this paper, when encountering the need to verify stationary overshoots as such, we will do so directly.



\subsection{Real-valued self-similar Markov processes}\label{sec:rssMp}
A real self-similar Markov process (rssMp) extends the notion of a pssMp albeit the requirement that the process is positive  is dropped allowing the exploration of $\mathbb{R}$ until absorption in the cemetery state 0 (if at all).  
%
 Significant effort has been invested in the last few years to extend the theory of pssMp to the setting of $\mathbb{R}$.
 The description below is the culmination of the work in \cite{GV, VG, Kiu, Chy-Lam} with more recent clarity given in Chaumont et al. \cite{CPR}, Kuznetsov et al. \cite{KKPW} and Dereich et al. \cite{DDK}. 
  \smallskip

Analogously to Lamperti's representation, for a real self-similar Markov process $\mathcal{X}$ there is a Markov additive process $((\xi_t,J_t), t\geq 0)$ on $\R\times \{-1,1\}$ such that
\begin{align}\label{LK}
   \mathcal{X}_t =  J_{\varphi_t}\exp\bigl( \xi_{\varphi_t}\bigr) ,\quad  t\leq I_\infty : = \int_0^\infty e^{\alpha\xi_s}ds,
\end{align}
where $\varphi_t=\inf\{s>0 : \int_0^s \exp(\alpha \xi_u)du >t\}$ and $(\xi_0,J_0)=(\log |z|,[z])$ with
$$
[z]=\begin{cases} 1 &\mbox{ if } z>0, \\-1 &\mbox{ if }z<0.\end{cases}
$$
The representation \eqref{LK} is known as the Lamperti--Kiu transform. 
Here, by  Markov additive process (MAP), we mean  the regular strong Markov process with probabilities ${\bf P}_{x,i}$, $x\in\mathbb{R}$, $i\in\{-1,1\}$, such  that $(J_t, t\geq 0)$ is a continuous time Markov chain on $\{-1,1\}$ (called the modulating chain) and, for any $i\in \{-1,1\}$ and $s,t\geq 0$,
\begin{align*}
	&\text{given }\{J_t=i\},\text{ the pair }(\xi_{t+s}-\xi_t,J_{t+s})_{s\geq 0} \text{ is independent of the past}\\
	&\qquad\text{ and has the same distribution as }(\xi_s, J_s)_{s\geq 0}\text{ under }{\bf P}_{0,i}.
\end{align*}
If the MAP is killed, then $\xi$ is sent  to the cemetery state $\{-\infty\}$. All background results for MAPs that relate to the present article can be found in the Appendix of Dereich et al. \cite{DDK}.\smallskip

The mechanism behind the Lamperti--Kiu representation is thus simple. The modulation $J$ governs the sign and, on intervals of time for which there is no change in sign, the Lamperti--Kiu representation effectively plays the role of the  Lamperti representation of a pssMp. In a sense, the MAP formalism gives a concatenation of signed Lamperti representations between times of sign change.
\begin{rem}\rm\label{rssMpispssMp} 
	Typically one can assume the Markov chain $J$ to be irreducible as otherwise the corresponding self-similar Markov processes only switches signs at most once and can therefore be treated using the theory of  pssMp.
\end{rem}
Analogously to L\'evy processes one knows that an unkilled MAP $(\xi,J)$ either drifts to $+\infty$ (i.e. $\lim_{t\uparrow\infty}\xi_t=+\infty$), drifts to $-\infty$ (i.e. $\lim_{t\uparrow\infty}\xi_t=-\infty$) or oscillates (i.e. $\liminf_{t\uparrow\infty}\xi_t=-\infty$ and $\limsup_{t\uparrow\infty}\xi_t=+\infty$), in the almost sure sense. Moreover, just in the case of  pssMp a simple 0-1 law for rssMp holds, distinguishing the case of conservative processes from non-conservative processes. We have 
	\begin{align*}
{P}_z(\zeta<\infty)=1\text{ for all }z\neq 0\quad&\Longleftrightarrow\quad(\xi,J) \text{ drifts to }-\infty \text{ or is killed,} \\{P}_z(\zeta<\infty)=0\text{ for all }z\neq 0\quad&\Longleftrightarrow\quad (\xi,J)\text{ drifts to }+\infty \text{ or oscillates},
\end{align*}
where $\zeta=\inf\{t>0: \mathcal{X}_t=0\}$. Generalizing the results for pssMps, the existence of 0 as an entrance point was addressed by Dereich et al. \cite{DDK}. It was shown that a necessary and sufficient condition for the existence of a Feller extension $(\P_z,z\in \R)$ under which trajectories leave $0$ continuously  of $(\P_z,z\neq 0)$ in terms of the underlying MAP is weak convergence of the overshoots; that is,
\begin{align}\label{239}
	\lim_{a\to+\infty} {\bf P}_{0,i}&(\xi_{\varsigma_a^+}-a \in dy, J^+_{\varsigma_a^+} = j), \qquad y\geq 0, i,j\in\{-1,1\},
\end{align}
exists in the sense of weak convergence independently of $i\in \{-1,1\}$ and is non-degenerate, where $\varsigma_a^+=\inf\{t>0:\xi_t\geq a\}$.  
Just as in the pssMp setting, this can be thought of as a natural condition for similar reasons. As for L\'evy processes, there is an analytic condition for \eqref{239} in terms of the generalized triplet for MAPs, see Theorem 5 of \cite{DDK}.

%

\subsection{Stable processes and their path functionals as  rssMp}\label{sec:stablerssMp}

Stable processes and certain types of  conditioned stable processes are linked  to the theory of self-similar Markov processes. A little care is needed since the definition of a real self-similar Markov process given above asks for $0$ to be absorbing.

\begin{itemize}
\item[\bf (1)]  To  discuss stable process in the light of self-similarity we should remind ourselves of the accessibility of the single point set $\{0\}$, see for instance Chapter 7 of \cite{Kbook}. If we set $\tau^{\{0\}} = \inf\{t>0: X_t = 0\}$, then, for all $x\neq0$,
\begin{align*}
\mathbb{P}_x\big(\tau^{\{0\}}<\infty \big) =
 \begin{cases}
0&\text{ if } \alpha\in (0,1]\\
1&\text{ if } \alpha\in (1,2)
\end{cases},
\end{align*}
for all $x\neq 0$. In other words, $\{0\}$ is polar if and only if $\alpha\leq 1$. In response to this observation, it is $(X^\dagger_t , t\geq 0)$ which conforms to our definition of a positive or real self-similar Markov process, where 
\begin{equation}
X^\dagger_t : = X_t\mathbf{1}_{(t<\tau^{\{0\}})},\qquad  t\geq 0.
\label{Xdagger}
\end{equation}
 This is clearly the case when $0$ is polar as $X^\dagger = X$. However, when 0 is not polar, a little more detail is deserving in order to verify the scaling property. Indeed, suppose momentarily we write $(X^{(x)}_t, t\geq 0) $, $x\neq 0$, to indicate the initial value of the process, i.e. $X^{(x)}_0 = x$. Then, for $c>0$,
\begin{align*}
\tau^{\{0\}}&=\inf\{t>0: X^{(x)}_t =0\}  \\
&=c^{-\alpha}\inf\{c^{\alpha}t>0: cX^{(x)}_{c^{-\alpha} c^{\alpha}t} =0\} \\
&
=:c^{-\alpha}\inf\{s>0: \tilde{X}^{(cx)}_{s} =0\}\\
&=:c^{-\alpha}\tilde{\tau}^{\{0\}},
\end{align*}
where $\tilde{X}^{(cx)}_{s}: = cX^{(x)}_{c^{-\alpha}s}$, $s\geq 0$, is equal in law to $X^{(cx)}_s$, $s\geq 0$. With this in hand, we now easily verify that, for  
$c>0$,
\[
cX^{(x)}_{c^{-\alpha}t}\mathbf{1}_{(c^{-\alpha}t<\tau^{\{0\}})} = \tilde{X}^{(cx)}_{t}\mathbf{1}_{(t< \tilde{\tau}^{\{0\}})}, \qquad t\geq 0,
\]
and, as such, the right-hand side is equal in law to $(X^\dagger, \mathbb{P}_{cx})$.

\smallskip

From Section \ref{sec:rssMp} we see that there is a family of MAPs corresponding to the family of killed stable processes through the Lamperti--Kiu representation. A characterisation of this family $(\xi, J)$ was uncovered in Chaumont et al. \cite{CPR} (see also Kuzentsov et al. \cite{KKPW}). From the characterization it can be deduced that
\begin{itemize}
\item $(\xi,J)$ drifts to $+\infty$ if $\alpha\in(0,1)$ and $\xi$ is not the negative of a subordinator,
\item $(\xi, J)$ drifts to $-\infty$ if $\alpha\in(0,1)$ and $\xi$ is the negative of a subordinator
\item $(\xi, J)$ oscillates if $\alpha = 1$,
\item $(\xi, J)$ drifts to $-\infty$ if  $\alpha\in(1,2)$. 
\end{itemize}
This path behaviour is consistent through the Lamperti--Kiu represenation with the  fact that, as a Markov process, a stable L\'evy process
\begin{itemize}
\item  is transient when $\alpha \in(0,1)$, in which case 0 is polar, hence, $\lim_{t\to\infty}|X_t| = \infty$,
\item  is recurrent when $\alpha =1$ but points are polar, hence, $\limsup_{t\to\infty}|X_t| =\infty$ and $\liminf_{t\to\infty}|X_t| = 0$, 
\item  almost surely hits zero when $\alpha\in(1,2)$ and $\lim_{t\to\tau^{\{0\} }}|X_t| = 0$.
\end{itemize}
See for instance the discussion around Theorems 7.4 and 7.5 in \cite{Kbook} for these facts.

\item[\bf (2)]   
More examples of a rssMps that can be derived from stable processes emerge through special kinds of conditioning. We confine this remark  to the setting of two-sided jumps. When $\alpha\in(0,1)$, it was shown in Kyprianou et al. \cite{KRS} that, for $x\neq 0$, $A\in\mathcal{F}_t: = \sigma(X_s, s\leq t)$ and each $a>0$, 
\begin{align}\label{attract}
	 \mathbb{P}^{\circ}_x\big(A{\color{black}\,\cap\,\{t<\tau^{(-a,a)}\}}\big)=\lim_{\varepsilon\to 0}\mathbb{P}_x\big(A{\color{black}\,\cap\,\{t<\tau^{(-a,a)}\}}\,\big|\,\tau^{(-\varepsilon,\varepsilon)}<\infty\big),
\end{align}
where $\tau^{(-a,a)}=\inf\{t>0: |X_t|<a \}$,
defines a consistent family of probability laws such that $(X, \mathbb{P}^{\circ}_x)$, $x\neq 0$, defines a rssMp, referred to as the stable process conditioned to continuously absorb at the origin. Additionally they show that, irrespective of the point of issue, the absorption time is almost surely finite and  that $\P^\circ$ is an $h$-transform of $\P$ via
\begin{equation}\label{updownCOM}
\left.\frac{{\rm d}\mathbb{P}^\circ_x}{{\rm d}\mathbb{P}_x}\right|_{\mathcal{F}_t} = \frac{h(X_t)}{h(x)}, \qquad t\geq 0, x\in\mathbb{R}\backslash\{0\},
\end{equation}

where 
\begin{align}\label{h}
h(x): =
-\Gamma(1-\alpha) \left(\dfrac{\sin(\pi\alpha\hat\rho)}{\pi}\mathbf{1}_{(x\geq 0)}+ \dfrac{\sin(\pi\alpha\rho)}{\pi}\mathbf{1}_{(x<0)}\right) |x|^{\alpha-1}, \quad x\in\mathbb{R}.
\end{align}

Moreover, when $\alpha\in(1,2)$, it was also shown in Kyprianou et al. \cite{KRS} as well as in Chaumont et al. \cite{CPR} that, for $x\neq 0$ and $A\in\mathcal{F}_t$,
\begin{align}\label{avoid}
	\mathbb{P}^{\circ}_x(A) =\lim_{a\to \infty}\mathbb{P}_x\big( A{\color{black}\,\cap\,\{t<\tau^{(-a,a)^{\rm c}}\}}\,\big|\,\tau^{(-a,a)^{\texttt c}}<\tau^{\{0\}}\big),
\end{align}
where $\tau^{(-a,a)^c}=\inf\{t>0: |X_t|\geq a \}$,
also defines a consistent family of probability laws such that $(X, \mathbb{P}^{\circ}_x)$, $x\neq 0$, defines rssMp referred to as  the stable process conditioned to avoid the origin. Moreover, 
the absolute continuity \eqref{updownCOM} is still valid, albeit with $X$ replaced by $X^\dagger$. 

\smallskip

It is a straightforward exercise to show that expectations of the form $\mathbb{E}^\circ_x[f(cX_{c^{-\alpha}s}, s\leq t)]$, where $f$ is bounded and measurable, transform to $\mathbb{E}^\circ_{cx}[f(X_s, s\leq t)]$ thanks to the shape of the $h$-transform and the inherent scaling of the stable process. Said another way, the process $(X^\circ, \mathbb{P}_x)$, $x\neq 0$, is a rssMp.

\smallskip

The reader may be left wondering if either of these two conditionings applies when $\alpha = 1$ even though the $h$-transform becomes trivial. Clearly conditioning to avoid the origin is meaningless as 0 is inaccessible for the Cauchy process. It also turns out that conditioning the Cauchy process to continuously absorb at the origin cannot be made sense of. In this way the Cauchy process asserts itself again as a distinguished intermediary case in the class of stable processes.

\end{itemize}
Amongst the above examples of rssMp, i.e. the stable process $X$ killed on hitting the origin, the stable process conditioned to continuously absorb at the origin and the stable process conditioned to avoid the origin, we can examine the existence of $0$ as an entrance point.  Clearly when $\alpha \in (0,1]$ we already know that $X$ is well-defined as entering from 0 (it never hits 0 again). Moreover, when $\alpha\in(1,2)$, the process $X$ is instantaneously absorbed at 0 when issued there and, hence, the origin cannot serve as an entrance point. When $\alpha\in(0,1)$, it is also clear that 0 cannot serve as an entrance point for the stable process conditioned to continuously absorb at the origin; cf. Kyprianou \cite{KALEA}.
 However, when $\alpha\in(1,2)$, it is meaningful to check whether $0$ is an entrance point in the sense that there is a Feller extension $\P^\circ_x$, $x\in \R$ of $\P^\circ_x$, $x\neq 0$.
\begin{lemma}\label{zeroenter} 
	When $ \alpha\in(1,2)$ and the stable process has two-sided jumps, then $0$ is an entrance boundary of $\mathbb{P}^\circ_x$, $x\neq 0$. 
\end{lemma}
\begin{proof}
As remarked above, the conditioned processes are also rssMps so we can apply the results on entrance from $0$ for rssMps (see Section \ref{sec:rssMp} and the necessary and sufficient condition \eqref{239} in particular). 
From the Lamperti--Kiu representation there is a corresponding MAP that we denote by $(\xi^\circ,J^\circ)$ for which we need to check convergence of overshoots \eqref{239}. One could either try to appeal to  the analytic condition of Dereich et al. \cite{DDK} for the convergence of overshoots or invoke known formulas for overshoots for stable processes.

\smallskip

To carry out the second option, we   note that, due to the Lamperti--Kiu representation, the range of $J^\circ_t\exp(\xi^\circ_t)$, $t\geq 0$, agrees with that of the conditioned process. Therefore, using the absolute continuity relation \eqref{updownCOM}, \eqref{239} is equivalent to the existence of the weak limit
\begin{equation}\label{othercheck}
\lim_{|x|\to0}\mathbb{P}^\circ_x(X_{\tau^{(-1,1)^c}}\in dy)=\lim_{|x|\to0}\frac{h(y)}{h(x)}\mathbb{P}_x\big(X_{\tau^{(-1,1)^c}}\in dy, \tau^{(-1,1)^c}<\tau^{\{0\}}\big), \qquad |y|\geq 1,
\end{equation}
in the sense of weak convergence, where $\tau^{(-1,1)^c} = \inf\{t>0:| X_t|\geq1\}$. Fortunately, there are fluctuation identities known in explicit form in existing literature, which enables us to deal with the righthand side of \eqref{othercheck} directly. 
 Indeed, in this case, we may appeal to Corollary 2 of Kyprianou \cite{deep1} which tells us, e.g. when $y>1$ and $x\in(0,1)$, for all $\alpha\in(0,2)$,
 \begin{align}\label{2sideexit}
& \mathbb{P}_x\big(X_{\tau^{(-1,1)^c}}\in dy, \tau^{(-1,1)^c}<\tau^{\{0\}}\big)/dy\notag\\
&=\frac{\sin(\pi\alpha\rho)}{\pi}(1+x)^{\alpha\rhohat}(1-x)^{\alpha\rho}(1+y)^{-\alpha\rhohat}(y-1)^{-\alpha\rho}( y-x)^{-1}\notag\\
&\quad-c_\alpha \frac{\sin(\pi\alpha\rho)}{\pi}(1+y)^{-\alpha\rhohat}(y-1)^{-\alpha\rho} y^{-1}x^{\alpha-1} \int_1^{1/x} (t-1)^{\alpha\rho-1} (t+1)^{\alpha\rhohat-1}\, d t,
 \end{align}
 where $c_\alpha = \max\{(\alpha-1),0\}$.
Recalling the definition of $h$ we use the above identity together with L'H\^opital's rule to deduce the righthand side of \eqref{othercheck} exists. The details for these and  other combinations of $x$ and $y$ are left to the reader (see also Remark 6 in Profeta and Simon \cite{PS}). Hence, overshoots of $(\xi, J)$ converge and the theory of rssMps implies the claim.
 \end{proof}
To complete this section, we recall a remarkable  result which gives a pathwise connection between $X$ and the conditioned processes given in \eqref{attract} and \eqref{avoid}.
In the following result, which is due to Bogdan and \.Zak \cite{BZ}, we write $\hat{\mathbb{P}}_{x}$ for the law of $-X$ under $\P_x$. Under $\hat{\P}_x$ the canonical process is again a stable process, the so-called dual stable process.

\begin{theorem}[Riesz--Bogdan--\.{Z}ak transform]\label{RBZthrm} 
Suppose that $X$ under $\mathbb{P}_x$ has two-sided jumps and
\begin{align}
\eta_t = \inf\left\{s>0 : \int_0^s |X^\dagger_u|^{-2\alpha}{\rm d}u >t\right\},\quad t\geq 0. 
\label{firsteta}
\end{align}
Then, for all $x\neq 0$, the law of $(1/X^\dagger_{\eta_t})_{t\geq 0}$ under $\hat{\mathbb{P}}_{x}$ is $\mathbb{P}_{1/x}^\circ$.
\end{theorem}

In  words, this theorem gives a pathwise link (spatial inversion and time-change) between the killed stable process $X$ and the conditioned process ($h$-transform).

\subsection{Stable processes and their path functionals as  pssMp}\label{pssmpexamples}

Stable processes are also naturally linked to pssMp by looking at different functionals of $X$. Three pertinent cases in point are that of the censored stable, the radial process and the stable process conditioned to stay positive.

\begin{itemize}

\item[\bf (1)]  For the (positive) censored stable process, define the occupation time of $(0,\infty)$,
\begin{align*}
	A_t = \int_0^t \mathbf{1}_{(X^\dagger_s > 0)} \, d s,
\end{align*}
and let $\gamma_t = \inf\{ s \ge 0 : A_s > t \}$ be its right-continuous inverse.
The process $(X^\dagger_{\gamma_t})_{t\geq 0}$ is what is understood to be the (negatively) censored stable process. In words, this is the process
formed by erasing the negative components of $X^\dagger$ and shunting together  the resulting sections of trajectory so that the temporal gaps are closed.
The L\'evy process that underlies its Lamperti representation, say $\xi^{>}$, was found in Theorem 5.5 of Kyprianou et al. \cite{KPW}; up to a multiplicative constant, its characteristic exponent has the form 
\begin{equation}
\Psi^{>}(z)= \frac{\Gamma(\alpha\rho - \iu{z})}{\Gamma(-\iu{z})}
    \frac{\Gamma(1 - \alpha\rho + \iu{z})}{\Gamma(1 - \alpha + \iu{z})} , \qquad z\in\mathbb{R}.
    \label{censoredpsi}
\end{equation}
Note, here we use the convention that  $\Psi^{>}(z) = -t^{-1}\log\mathbf{E}^>[\exp({\rm i }z\xi^>_t)]$, $t>0$, and we consistently use this arrangement when citing characteristic exponents of other L\'evy processes.
It is not difficult to imagine that one may also consider the analogue of this process when we censor away the positive components of $X$. In that case, the roles of $\rho$ and $\rhohat$  are exchanged on the right-hand side of  \eqref{censoredpsi}.

\smallskip

It is also worthy of note at this point that the censoring procedure of $X^\dagger$ leading to a pssMp is not specific to the stable case. Indeed, any rssMp can be censored in the same way and will still result in a pssMp. (We leave it as an exercise to verify this fact, however the proof is essentially the same as in the stable setting, see Kyprianou et al. \cite{KPW}) We will see such an example later in this  exposition.

\medskip

\item[\bf (2)]  The radial process of $X$ is nothing more than $|X|$. In general, $|X|$ is not a Markov process as one needs to know the sign of $X$ to determine its increments. However, when $X$ is symmetric, that is to say $\rho = 1/2$, then |X| is Markovian. The same is true of $|X^\dagger|$ since $X = X^\dagger$. Moreover $|X^\dagger|$ is also a pssMp. The latter can be deduced from symmetry and the Lamperti--Kiu transformation \eqref{LK}; see the discussion in Chapter 13 of \cite{Kbook}. 
The associated L\'evy process, $\absxi$, that underlies the Lamperti transform has characteristic exponent given by
\begin{equation}
\Psi^{|\cdot|}(z) =\frac{\Gamma(\frac{1}{2}(-{\rm i}z +1 ))}{\Gamma(-\frac{1}{2}{\rm i}z)}\frac{\Gamma(\frac{1}{2}({\rm i}z +1))}{\Gamma(\frac{1}{2}{\rm i}z)}, \qquad z\in\mathbb{R},
\label{a}
\end{equation}
up to a multiplicative constant. 
See Caballero et al. \cite{CPP} for further details.

\medskip

\item[\bf (3)]  The stable process conditioned to stay positive is only of interest for our purposes when $X$ does not have monotone paths. Introduced in Chaumont \cite{C96}, it arises from the limiting procedure (which is indeed valid as a definition for any L\'evy process conditioned to stay positive) \begin{equation}
	\mathbb{P}_x^\uparrow(A): = \lim_{q\downarrow0}\mathbb{P}_x\big(A, \, t<q^{-1} \mathbf{e}\big| X_s
\geq 0,\, s\leq q^{-1}\mathbf{e}\big)
\label{CTSP}
\end{equation}
for $A\in\mathcal{F}_t: = \sigma(X_s, s\leq t)$, where $\mathbf{e}$ is an independent and exponentially distributed random variable with unit rate; see also Chaumont and Doney \cite{CD}. This defines a new family of probabilities on $\mathbb{D}(\mathbb{R}_+,\mathbb{R}_+)$ and the resulting process $(X, \mathbb{P}^\uparrow_x)$, $x>0$, is what we call the stable process conditioned to stay positive.

\smallskip

It turns out that the family $\mathbb{P}^\uparrow_x$, $x>0$, is absolutely continuous with respect to $\mathbb{P}_x$, $x>0$, on $(\mathcal{F}_t, t\geq 0)$ via the $h$-transform relation
\begin{equation}
\left.\frac{{\rm d}\mathbb{P}^\uparrow_x}{{\rm d}\mathbb{P}_x}\right|_{\mathcal{F}_t} = \frac{X_t^{\alpha\hat\rho}}{x^{\alpha\hat\rho}}\mathbf{1}_{(t<\tau^{(-\infty,0)})}, \qquad t\geq 0, x>0,
\label{CSP}
\end{equation}
where $\tau^{(-\infty,0)} = \inf\{t>0: X_t <0\}$. Note that when $X$ is spectrally negative, the $h$-function in the above $h$-transform is precisely the one given in \eqref{h}, that is
\begin{equation}
h(x) =-\Gamma(1-\alpha) \frac{\sin(\pi(\alpha-1))}{\pi}x^{\alpha-1} = \frac{1}{\Gamma(\alpha)}x^{\alpha-1}, \qquad x\geq 0,
\label{speconesidedh}
\end{equation}
on account of the fact that $\rho = 1/\alpha$ for spectrally negative stable processes.

\smallskip

Similarly to the conditioned process from the previous section, stable processes conditioned to be positive are self-similar. The L\'evy process $\xi^\uparrow$ that underpins the Lamperti transform was computed in Cabellero and Chaumont \cite{CC}, see also Section 13.4.2 of Kyprianou \cite{Kbook}), and takes the form
\begin{equation}
 \Psi^\uparrow(z)=\frac{\Gamma(\alpha\rho -{\rm i} {z})}{\Gamma(-{\rm i}{z})}
 \frac{\Gamma(1+{\rm i}{z} + \alpha\rhohat)}{\Gamma(1+{\rm i}{z}) },
 \qquad z\in\R.
 \label{Psiuparrow}
\end{equation}
They also proved that $\xi^\uparrow$ drifts to $+\infty$ so that according to \eqref{999}, $0$ is polar for the stable processes conditioned to be positive.\smallskip

\item[{\bf (4)}] Finally, we consider the setting of $\alpha\in(0,1)$ and that $X$ has monotone paths. Conditioning ascending (resp. descending) stable subordinator to stay  positive (resp. negative) is an uninteresting concept. However, what is more interesting is to consider an ascending (resp. descending) stable subordinator to approach the origin continuously from  below (resp. above). 

\smallskip

This was treated by Chaumont \cite{C96} and Kyprianou et al. \cite{KRSe}, where it was shown that for all $x>b >0$,
    \[
      \P^{\circ}_x(A, t< \tau^-_b): = \lim_{\varepsilon\downarrow0}\P_x(A, t< \tau^-_b \,|\, X_{\tau^-_0-}\leq \varepsilon), \qquad t\geq 0, A\in\mathcal{F}_t,
    \]
    is well-defined such that, for $x>0$,
  \begin{equation}
   \left. \frac{{\rm d}\P^{\circ}_x}{{\rm d}\P_x}\right|_{\mathcal F_t} = \frac{X_t^{\alpha-1}}{x^{\alpha-1}}\1_{\{X_t \geq 0\}}.
    \label{subCOM}
  \end{equation}
In the Lamperti representation of  $(X^\circ, \mathbb{P}_x)$, $x\geq 0$, it was also shown by \cite{KRSe} that  $\xi$ is the negative of a subordinator so that its  Laplace exponent is given by 
\[
-\frac{1}{t}\log\mathbf{E}_x[e^{\lambda \xi_t}]=\frac{\Gamma(\alpha+\lambda )}{\Gamma(\lambda)}, \qquad \lambda \geq 0.
\]
\end{itemize}

\smallskip
Similarly to the discussion on conditioned stable processes in Section \ref{sec:stablerssMp} we may ask whether 0 is an entrance point for the process conditioned to stay positive in the sense that there is a Feller extension  $\P_x^\uparrow$, $x> 0$ allowing the meaningful inclusion of $\mathbb{P}^\uparrow_0$.
\begin{lemma}\label{zeroenter2} 
	If $ \alpha\in(1,2)$, then $0$ is an entrance point for $\P_x^\uparrow$, $x>0$.
\end{lemma}

\begin{proof}
The proof is almost the same as the proof of Lemma \ref{zeroenter}. Analogously to \eqref{othercheck}, we may appeal to \eqref{C2} and the Lamperti transform \eqref{pssMpLamperti} to deduce that a necessary and sufficient condition for 0 to be an entrance point is that the righthand side of
\begin{equation}
\lim_{x\downarrow0}\mathbb{P}^\uparrow_x(X_{\tau^{(1,\infty)}} \in dy)=\lim_{x\downarrow0}\frac{y^{\alpha\rhohat}}{x^{\alpha\rhohat}}\mathbb{P}_x(X_{\tau^{(1,\infty)}}\in dy, \,\tau^{(1,\infty)}< \tau^{(-\infty,0)})
\label{check21}
\end{equation}
exists weakly.
Similarly to \eqref{othercheck}, we can verify this directly by appealing to already known explicit fluctuation identities. In this case, we need the two sided exit problem which was solved by Rogozin \cite{Rog}. For example, under the regime $\alpha\in(1,2)$ when  $0<\alpha\rho<1$ (which includes the case of spectral positivity),
\begin{align}\label{1stcase}
\begin{split}
&\quad\mathbb{P}_x(X_{\tau^{(1,\infty)}} \in d  y ;
 \,\tau^{(1,\infty)}< \tau^{(-\infty,0)})\\
 &= \frac{\sin(\pi\alpha \rho)}{\pi} (1-
x)^{\alpha\rho}x^{\alpha\rhohat}
(y-1)^{-\alpha\rho}y^{-\alpha\rhohat}(y-x)^{-1}d y, \qquad y>1,
\end{split}
\end{align}
and when $\rho = 1/\alpha$ (which is the case of spectral negativity), then necessarily $X_{\tau^{(1,\infty)}} =1$ and 
\begin{align}\label{2ndcase}
\begin{split}
&\quad\mathbb{P}_x(X_{\tau^{(1,\infty)}} =1;
 \,\tau^{(1,\infty)}< \tau^{(-\infty,0)})\\
 &= 1- \frac{\sin(\pi\alpha \rhohat)}{\pi} (1-
x)^{\alpha\rhohat}x^{\alpha\rho}\int_0^\infty
(y-1)^{-\alpha\rhohat}y^{-\alpha\rho}(y-x)^{-1}d y, \qquad y>1.
\end{split}
\end{align}
The limiting computation in \eqref{check21} is now trivial to verify using \eqref{1stcase} and, with a little care, straightforward to verify using \eqref{2ndcase} as well. 
\end{proof}


In a similar spirit to the previous section, we complete this section by providing another remarkable pathwise transformation of the process $X$, connecting it to its conditioned version $\mathbb{P}^\uparrow_x$, $x>0$, but only in the case that $X$ is spectrally positive and $\alpha\in(1,2)$. As before, we write $\hat{\mathbb{P}}_x$, $x\neq0$ for the probabilities of $-X$.

\begin{theorem}[Chaumont]\label{CRBZ} Suppose that $X$ is spectrally positive with $\alpha\in(1,2)$ and define 
\begin{equation}
\eta_t = \inf\left\{s>0 : \int_0^s (X_u^\dagger)^{-2\alpha}{\rm d}u >t\right\},\quad t\leq \int_0^\infty (X_u^\dagger)^{-2\alpha}{\rm d}u.
\label{gamma_t}
\end{equation}
For all $x>0$, the law of $(1/{X}^\dagger_{\eta_t})_{t\geq 0}$ under ${\mathbb{P}}_{x}$ is $\hat{\mathbb{P}}_{1/x}^\uparrow$.

\end{theorem}
\begin{proof}Strictly speaking, 
this result is a special case of Theorem 2.4.1 in Chaumont \cite{Loicnotes}, which demonstrates a more general result of this kind for pssMp. Indeed, suppose $X$ is a pssMp with associated L\'evy process $\xi$ via the Lamperti transform, and, in the same respect,  $\hat{X}$ is the pssMp associated to the L\'evy process $-\xi$. Then Theorem 2.4.1 of \cite{Loicnotes} states that $\hat{X}$, when issued from $y>0$ is equal in law to $(1/{X}_{\gamma_t}, t\geq 0)$ when issued from $1/y$, where the endogenous time-change $\gamma_t$  is structured as in \eqref{gamma_t}.
\smallskip

 The special case we are concerned with here  makes use of  the observation from Caballero and Chaumont \cite{CC} (see also Section 13.4.2 of Kyprianou \cite{Kbook}) that, if $X^\dagger$ is the spectrally positive stable process killed on hitting the origin with $\alpha\in(1,2)$, then its L\'evy process, say $\xi^\dagger$, underlying the Lamperti transform has characteristic exponent satisfying 
 \begin{equation}
\Psi^\dagger(z)  
 =  {\rm i}z \frac{\Gamma(\alpha-{\rm i}{z})}{\Gamma(1-{\rm i}{z})}
 ,\qquad {z}\in\mathbb{R},
\label{xipsi}
 \end{equation}
 up to a multiplicative constant.
One easily computes from this exponent that the mean of this L\'evy process at time 1 is equal to  $
 -{\rm i} \Psi^{\dagger\prime}(0) =  
 -\Gamma(\alpha)$, which is strictly negative, accounting for the almost sure hitting of the origin by $X$  as one would expect.
 On the other hand,  from \eqref{Psiuparrow}
  the L\'evy process underlying $(X, \hat{\mathbb{P}}^\uparrow_x)$, $x>0$, via the Lamperti transform, say $\hat{\xi}^\uparrow$
  takes the form 
 \begin{equation}
 \hat{\Psi}^\uparrow(z)= {\rm i}z\frac{\Gamma(\alpha+{\rm i}{z})}{\Gamma(1+{\rm i}{z})},\qquad {z}\in\mathbb{R},
   \label{xiuparrowpsi}
 \end{equation}
 up to a multiplicative constant.
Recall  that spectral positivity of $(X, \mathbb{P}_x)$, $x\in\mathbb{R}$,   means that $\rhohat = 1/\alpha$ and hence in the context of deriving \eqref{xiuparrowpsi}, where $\hat X$ is used, we have $\rho = 1/\alpha$. Note now that  
 $-{\rm i} \Psi^{\uparrow\prime}(0)=\Gamma(\alpha)$, which is  strictly positive and accounts for the fact that $(X, \hat{\mathbb{P}}^\uparrow_x)$ is transient to $+\infty$.
 The statement of Theorem \ref{CRBZ} now follows by comparing these two exponents and recalling that the law of a L\'evy process is entirely determined by its characteristic exponent and, moreover,  that the law of a pssMp is entirely determined by its underlying L\'evy process (via the Lamperti transform).
\end{proof}

\section{Fundamental transformations}\label{sec_transforms}
	In this section we consider combinations of classical transformations (change of measure, change of space, random change of time) related to the SDE \eqref{2}, resp. the time-change representation \eqref{timechangesolution}. These will be crucial in the main part of the proof to apply results for stable L\'evy processes and self-similar Markov processes.

\subsection{Time-space inversions}
Before we state and prove an extension of the Riesz--Bogdan--\.Zak transformation (Theorem \ref{RBZthrm} above) we recall a simple lemma on time-changes 
which is essentially a re-wording of Theorem 1.1 and the discussion above in Chapter 6 of Ethier and Kurtz  \cite{EthKur}, see also Proposition 3.5 of K\"uhner and Schnurr \cite{KSc}.
\begin{lemma}\label{lemm}
Suppose $(Y_t , t\geq 0)$ is a c\`adl\`ag trajectory, $f\geq 0$ continuous and
	\begin{align*}
		t_0=\inf\big\{t\geq 0: f(Y_t)=0\big\}\quad \text{and}\quad t_1=\inf\left\{t\geq 0: \int_0^t \frac{1}{f(Y_u)}\,du=\infty\right\}.
	\end{align*}
	If $t_0=t_1$, then the integral equation $\upsilon_t=\int_0^t f(Y_{\upsilon_s})\,ds$ has a unique solution which is of the form
	\begin{align*}
		\upsilon_t=\inf\left\{s\ge 0: \int_0^s \frac{1}{f(Y_u)}\,du>t\right\}\wedge t_0,\quad t\geq 0.
	\end{align*}
\end{lemma}

In what follows we set
\[
\beta(x) = \sigma(1/x)^{-\alpha}|x|^{-2\alpha}, \qquad x\in\mathbb{R}\backslash\{0\}
\]
and prove an extension of the Riesz--Bogdan--\.Zak Theorem.
\begin{prop} \label{prop} 
Assume the stable process $X$ with distribution $\mathbb{P}_x$, $x\in\mathbb{R}$, has two-sided jumps and $\sigma>0$ is continuous.

\smallskip

(i) Define the time-space transformation
\begin{align}	\label{RBZ}
	{Z}^\dagger_t=\frac{1}{\hat{X}^\circ_{\theta_t}}, \qquad t< \int_0^\infty \beta(\hat{X}^\circ_u){\,d u},
\end{align}
where 
\[
\theta_t =\inf\left\{s> 0 : \int_0^s\beta(\hat{X}^\circ_u){\,d u}>t\right\}.
\]
If $\hat{X}^\circ$ has law $\hat{\mathbb{P}}^\circ_{1/x}$, $x\neq 0$, then $Z^\dagger$ is the time-changed process \eqref{timechangesolution} under $\mathbb{P}_x$ killed at the origin.

\smallskip

(ii) Define the time-space transformation 
	\begin{align}
		X^\circ_t=\frac{1}{\hat{Z}^\dagger_{\vartheta_t}}, \quad t\geq 0,
		\label{XfromZ}
  	\end{align}	
	where 
	\[
	\vartheta_t = \inf\left\{s>0 : \int_0^s\frac{1}{\beta(1/\hat{Z}^\dagger_u)}{\,d u} >t\right\}.
	\]
	If $\hat{Z}^\dagger$ is the time-changed process from \eqref{timechangesolution} under $\hat{\mathbb{P}}_x$, $x\neq 0$, killed at the origin, then the law of $X^\circ$ is $\mathbb{P}^\circ_{1/x}$.
\end{prop}
Keeping in mind the time-change from \eqref{timechangesolution} gives the unique (possibly exploding) weak solution to the SDE \eqref{2}, the proposition tells us how to transform solutions to the SDE via spatial inversion and time-change into an $h$-transform of the driving stable process, and vice versa. Later on, this will be applied as follows: in order to understand solutions started at infinity, one can equivalently understand the $h$-process started from zero in combination with the behavior of the time-change. Since the $h$-process is a self-similar Markov process, the behavior at zero has been understood in recent years, so it all will boil down to understanding the time-change.


\begin{proof}[Proof of Proposition \ref{prop}]
%
%
%
%
%
%
(i) The Riesz--Bogdan--\.Zak Theorem (Theorem \ref{RBZthrm}) states that under $\P_x$ the transformation $\hat{X}^\circ_t=1/X^\dagger_{\eta_t}$, $t\geq 0$, has law $\hat{\mathbb{P}}_{1/x}^\circ$ with the time-change $\eta_t = \inf\left\{s>0 : \int_0^s |X^\dagger_u|^{-2\alpha}{\rm d}u >t\right\}$. Next, according to the statement, we time-change $1/\hat{X}^\circ_t=X^\dagger_{\eta_t}$ with $\theta$ and show that $1/\hat{X}^\circ_{\theta_t}=X^\dagger_{\eta\circ\theta_t}$ satisfies the time-change relation \eqref{timechangesolution}. To do so, let us consider the concatenation $\eta\circ \theta$ written in terms of $X^\dagger$. Using the chain rule gives 
\begin{align*}
	\frac{d \eta_t}{dt}  = |X^\dagger_{\eta_t}|^{2\alpha}\quad \text{and}\quad	\frac{d \theta_t}{dt} =1/\beta(1/X^\dagger_{\eta\circ{\theta_t}})=|X^\dagger_{\eta\circ {\theta_t}}|^{-2\alpha}\sigma(X^\dagger_{\eta\circ \theta_t})^{\alpha},
\end{align*}
and hence,
\begin{align}\label{5}
	\frac{d\eta\circ\theta_t}{dt} = \left.\frac{d\eta_s}{ds }\right|_{s = \theta_t} \frac{d\theta_t}{dt}=\sigma(X^\dagger_{\eta\circ\theta_t})^{\alpha}.
\end{align}
Defining $\gamma_t=\eta\circ \theta_t$, we note that $\gamma$ satisfies the pathwise equation 
\begin{equation}
\gamma_t=\int_0^t\sigma(X^\dagger_{\gamma_u})^\alpha\,du.
\label{gamma}
\end{equation}
  Applying Lemma \ref{lemm} to $X^\dagger$ with $\gamma$ playing the role of $\upsilon$ and $f (x)= \sigma(x)^{\alpha}$, we see that, trivially, $t_0 = t_1 = \infty$ and \eqref{gamma} has a unique solution almost surely given by 
  \begin{align*}
  	\gamma_t = \inf\left\{s>0 : \int_0^s\sigma(X^\dagger_{u})^{-\alpha} du>t\right\},\quad t\geq 0.
\end{align*} Plugging-in we see that $1/\hat{X}^\circ_{\theta_t}=X^\dagger_{\eta\circ \theta_t}=X^\dagger_{\gamma_t}$ for $t\geq 0$ and the righthand side obviously satisfies the claim.

\smallskip
  
(ii) By assumption $\hat Z=\hat X^\dagger_{\hat \tau}$, where the killed dual process $\hat X^\dagger=-X^\dagger$ has probabilities $\hat \P_x$, $x\in\mathbb{R}$, and $\hat \tau  = \inf\{s>0 : \int_0^s \sigma(\hat{X}^\dagger_u)^{-\alpha} \, d u>t\}$, $t\geq 0$.  Now note that,  $X_t^\circ=1/\hat Z^\dagger_{\vartheta_t}=1/ \hat{X}^\dagger_{\hat\tau\circ \vartheta_t}$. If we can show that, almost surely under $\hat \P_x$, $\hat\tau\circ \vartheta=\hat\eta$, where 
$\hat\eta_t = \inf\{s>0 : \int_0^s |\hat X^\dagger_u|^{-2\alpha}\, du >t\}$, then the proof is complete due to the Riesz--Bogdan--\.Zak theorem. 
To this end, as in part (i), we get from the chain rule
\begin{align}\label{5b}
	\frac{d\hat\tau\circ\vartheta_t}{dt} = \left.\frac{d\hat\tau_s}{ds }\right|_{s = \vartheta_t} \frac{d\vartheta_t}{dt}=\sigma(X^\dagger_{\tau\circ\vartheta_t})^\alpha\beta(1/X^\dagger_{\tau\circ\vartheta_t})=|X^\dagger_{\tau\circ \vartheta_t}|^{2\alpha},
\end{align}
which implies that $\hat\tau\circ\vartheta$ solves $\hat \tau\circ \vartheta_t=\int_0^t|X^\dagger_{\hat\tau\circ\vartheta_u}|^{2\alpha}\,du$. This is the same equation that $\hat\eta$ satisfies. 
Our proof is complete as soon as we show that the equation that both $\hat\tau\circ\vartheta$ and $\hat\eta$ solve has an  almost surely unique solution. 
We do this by 
applying  Lemma \ref{lemm} again but this time to $\hat{X}$ with $\hat\tau\circ\vartheta$ playing the role of $\tau$ and with  $f(x) = |x|^{2\alpha}$. 
The conditions of the Lemma are straightforward to verify, noting  in particular from the Riesz--Bogdan--\.Zak transform that $t_0 = t_1= \tau^{\{0\}}$.
\end{proof}

In a similar way we can apply Chaumont's transformation for spectrally one-sided processes, cf. Theorem \ref{CRBZ}, to obtain the theorem below. On account of its similarity to the one above, we omit the proof. 

\begin{prop}\label{BZuparrow}
Suppose that  $X$ is a spectrally positive stable process with distribution $\mathbb{P}_x$, $x\in\mathbb{R}$ and assume that $\sigma>0$.

\smallskip

(i) Define the time-space transformation
\begin{align}
	{Z}^\dagger_t=\frac{1}{\hat{X}^\uparrow_{\theta_t}}, \qquad t< \int_0^\infty \beta(\hat{X}^\uparrow_u){\,d u},
	\label{RBZ}
\end{align}
where 
\[
\theta_t =\inf\left\{s> 0 : \int_0^s\beta(\hat{X}^\uparrow_u){\,d u}>t\right\}.
\]
If $\hat{X}^\uparrow$ has probabilities $\hat{\mathbb{P}}^\uparrow_{1/x}$, $x> 0$, then $Z^\dagger$ is the time-changed process \eqref{timechangesolution} under $\mathbb{P}_x$ killed at the origin.

\smallskip

(ii) Define the time-space transformation 
	\begin{align}
		X^\uparrow_t=\frac{1}{\hat{Z}^\dagger_{\vartheta_t}}, \qquad t\geq 0,
		\label{XfromZ}
  	\end{align}	
	where 
	\[
	\vartheta_t = \inf\left\{s>0 : \int_0^s\frac{1}{\beta(1/\hat{Z}^\dagger_u)}{\,d u} >t\right\}.
	\]
	If $\hat{Z}^\dagger$ is the time-changed process from \eqref{timechangesolution} under $\hat{\mathbb{P}}_x$, $x> 0$, killed at the origin, then the law of $X^\uparrow$ is $\mathbb{P}^\uparrow_{1/x}$.
\end{prop}
Similarly to the discussion below Proposition \ref{prop}, we will use the proposition to reduce the behavior of solutions to the SDE \eqref{2} driven by a one-sided L\'evy process to the behavior at zero of self-similar Markov processes and the time-change. The situation is easier here, as we only need self-similar Markov processes with positive trajectories for which the theory is more classical.
\subsection{Time-reversal} \label{timereversesubsec}
	It was already part of Feller's \cite{feller1, feller2} analytic treatment of diffusion processes that an entrance point of a diffusion can be related to an exit point of an $h$-transformed diffusion. The general structure behind this was revealed by Hunt \cite{Hunt3, Hunt1and2} who showed how to relate time-reversal and $h$-transforms for Markov processes. Hunt's discrete time arguments were extended to continuous time by Nagasawa. For our purposes, only the results of Section 3 (reversal of Markov processes at $L$-times) of Nagawasa \cite{N} are of importance. Even though all the theory involved is very old, the application to the boundary behavior of the SDE \eqref{2} is only possible due to explicit potential formulas for killed stable processes developed in the past few years.
	\smallskip
	
	Let us first recall some definitions. Suppose that  $Y = (Y_t, t\leq \zeta)$ with probabilities ${\rm\texttt{P}}_x$, $x\in\mathbb{R}$, is a regular Markov process on (a subset of) $\mathbb{R}$ with cemetery state $\Delta$ and killing time $\zeta=\inf\{t>0: Y_t = \Delta\}$. Let us denote by $\mathcal{P}: = (\mathcal{P}_t, t\geq 0)$ the associated semigroup and we will write   ${\texttt P}_\nu = \int_{\mathbb{R}}\nu(da){\texttt P}_a$, for any probability measure $\nu$ on the state space of $Y$.
	
	\smallskip
	
	Suppose that $\mathcal{G}$ is the $\sigma$-algebra generated by $Y$ and  write $\mathcal{G}({\texttt P}_\nu)$ for its completion by the null sets of ${\texttt P}_\nu$. Moreover, write $\overline{\mathcal G} =\bigcap_{\nu} \mathcal{G}({\texttt P}_\nu)$, where the intersection is taken over all probability measures on the state space of $Y$, excluding the cemetery state.
		A finite  random time $\texttt{k}$ is called an $L$-time (generalized last exit time) if
\begin{itemize}
	\item[(i)] $\texttt{k}\leq \zeta$ and $\texttt{k}$ is measurable in $\overline{\mathcal G}$,
	\item[(ii)] $\{s<\texttt{k}(\omega)-t\}=\{s<\texttt{k}(\omega_t)\}$ for all $t,s\geq 0$.
\end{itemize}	
Theorem 3.5 of Nagasawa \cite{N}, shows that, under suitable assumptions on the Markov process, $L$-times form a family of `good times' at which the pathwise time-reversal $\stackrel{_\leftarrow}{Y}_t:=Y_{(\texttt{k}-t)-}, t\in [0,\texttt{k}],$ is again a Markov process. The most important examples of $L$-times are killing times and last  hitting times. To ease the reading, let us state precisely the three main conditions of Nagasawa's duality theorem, one of which is redundant in our setting (and we indicate as such lower down).  

\smallskip

\smallskip

\textbf{(A.3.1)} The potential measure $G_Y(a,\cdot)$ associated to $\mathcal{P}$, defined by the relation 
\begin{equation}
\int_\mathbb{R}f(x)G_Y(a,d x) = \int_0^\infty \mathcal{P}_t[f](a)d t={\texttt E}_a\left[\int_0^\infty f(X_t)\,dt\right],
\label{GY}
\end{equation}
 for bounded and measurable $f$ on $\mathbb{R}$, is a $\sigma$-finite measure. For a $\sigma$-finite measure $\nu$, if we put
\begin{align}\label{a1}
	\mu(A)=\int G_Y(a,A)\, \nu(da)\quad \text{ for }A\in \mathcal B(\R),
\end{align}
	then there exists a Markov transition semigroup, say $\hat{\mathcal{P}}: = (\hat{\mathcal P}_t, t\geq 0)$ such that the corresponding transition semigroup satisfies
	\begin{align}
		\int \mathcal{P}_t[f](x) g(x)\, \mu(dx)=\int f(x) \hat{\mathcal P}_t [g](x)\,\mu(dx),\quad t\geq 0,
		\label{weakdualtity}
	\end{align}
	for bounded, measurable and compactly supported test-functions $f, g$.\smallskip
	
	In other words, (A.3.1) asks for the semigroup $\mathcal P$ to be {\it in weak duality to a semigroup} $\hat{\mathcal P}$  {\it with respect to the measure} $\mu$ taking the form \eqref{a1}.\smallskip

{\bf (A.3.2)} Nagasawa's  second condition pertains to the finiteness of the semigroup $\mathcal{P}$ and its associated resolvents when randomised by initial distribution $\nu$, which in his most general setting, need not be a probability measure. However, this condition is redundant in our setting as we always consider the initial distribution $\nu$ to be a probability measure.  Hence we don't dwell on this condition any further.

\smallskip

\textbf{(A.3.3)} For any continuous test-function $f\in C_0(\R)$, the space of continuous and compactly supported functions,  and $a\in\R$, $\mathcal{P}_t[f](a)$ is right-continuous in $t$ for all $a\in \R$ and, for $q> 0$, $G_{\hat  Y}^{(q)}[f](\stackrel{_\leftarrow}{Y}_t)$ is right-continuous in $t$, where, for bounded and measurable $f$ on $\mathbb{R}$,
 \[
{G}_{\hat Y}^{(q)}[f](a) =\int_0^\infty  e^{-qt}\hat{\mathcal{P}}_t[f](a)d t,\qquad  a\in\mathbb{R}\]
is the $q$-potential associated to $\hat{\mathcal P}$.
\smallskip

\smallskip

 Nagasawa's  duality theorem, Theorem 3.5. of \cite{N}, now reads as follows.

 \begin{theorem}[Nagasawa's duality theorem]\label{Ndual} Suppose that assumptions {\rm{\bf (A.3.1)} } and {\rm{\bf (A.3.3)}} hold. For the given starting probability distribution $\nu$ in {\rm{\bf (A.3.1)} } and any $L$-time $\emph{\texttt{k}}$, the time-reversed process $\stackrel{_\leftarrow}{Y}$ under $\emph{\texttt P}_\nu$ is a time-homogeneous Markov process with transition probabilities
\begin{align}
	\emph{\texttt{P}}_\nu(\stackrel{_\leftarrow}{Y}_t \in A\,|\stackrel{_\leftarrow}{Y}_r, 0<r< s)=\emph{\texttt{P}}_\nu(\stackrel{_\leftarrow}{Y}_t \in A\,|\stackrel{_\leftarrow}{Y}_s)={p}_{\hat{Y}}(t-s,\stackrel{_\leftarrow}{Y}_s,A),\quad \emph{\texttt{P}}_\nu\text{-almost surely},
\end{align}
for all $0<s<t$ and Borel $A$ in $\mathbb{R}$, where ${p}_{\hat{Y}}(u, x, A)$, $u\geq 0$, $x\in\mathbb{R}$, is the transition measure associated to the semigroup $\hat{\mathcal P}$.
\end{theorem}
We will apply Nagasawa's duality theorem to different processes obtained from solutions to the SDE \eqref{2} by killing in different sets which leads to different processes obtained through time-reversal at $L$-times.
\begin{itemize}
	\item[(i)] Proposition \ref{nag}: two-sided jumps, $\alpha>1$, killed at the origin.
	\item[(ii)] Proposition \ref{nag2}: positive jumps, $\alpha>1$, killed at the origin, which is the same as killing at the negative half-line because of the positive jumps.
	\item[(iii)] Proposition \ref{nagAx}: two-sided jumps, $\alpha\in(0,1)$, no killing but explosion.
\end{itemize}
Proofs will be similar, in the sense that, to verify (A.3.1), explicit computations with different killed potential measures will be necessary.\smallskip

Here is the first application of Nagasawa's duality theorem.
\begin{prop}\label{nag}
Consider a stable process with $\alpha\in (1,2)$ and two-sided jumps. 
Suppose that 
 ${{\hat{X}^\circ}}$ has probabilities  $\hat{\mathbb{P}}^\circ_x$, $x\in \mathbb{R}$, defined by the change of measure  \eqref{updownCOM} albeit with respect to the dual $\hat{\mathbb{P}}_x$, $x\in\mathbb{R}$, and the entrance point $0$ from Lemma \ref{zeroenter}.  Define $\hat{Z}^\circ_ t = {{\hat{X}^\circ}}_{\iota_t}$, $t\geq 0$, where the time-change $\iota$ is given by 
\begin{align}\label{tauhat}
	\iota_t = \inf\left\{s>0 : \int_0^s \sigma({{\hat{X}^\circ}}_s)^{-\alpha}ds > t\right\}, \qquad t<\int_0^\infty \sigma({{\hat{X}^\circ}}_s)^{-\alpha}ds.
\end{align}
Suppose that $Z^\dagger$ is the (non-exploding) process from \eqref{timechangesolution} killed on hitting the origin. Then \begin{align}\label{claim}
	Z^\dagger\text{ and } \hat{Z}^\circ \text{ are in weak duality on }\R\backslash \{0\}\text{ with respect to }  \mu(dx)=\sigma(x)^{-\alpha}h(x)dx,
\end{align}
 with $h$ defined in \eqref{h}.
  Moreover:
\begin{itemize}

\item[(i)] The time-reversed process $\hat{Z}^\circ_{(\mbox{\rm $\texttt{k}$} - t)-}$, $t\leq \mbox{\rm $\texttt{k}$}$, under $\hat{\mathbb P}^\circ_0$ 
is a time-homogenous Markov process with transition semigroup which agrees with that of $Z^\dagger$, where $\mbox{\rm $\texttt{k}$}$ is any almost surely finite $L$-time for $\hat{Z}^\circ$.
\item[(ii)] If $\pm\infty$ is an entrance point for $Z$, then the time reversed process $Z^\dagger_{(\mbox{\rm $\texttt{k}$}-t)-}$, $t\leq \mbox{\rm $\texttt{k}$}$,  under $\emph{\rm P}_{\pm\infty}$ is a time-homogenous Markov process with transition semigroup which agrees with that of $\hat{Z}^\circ$, where $\mbox{\rm $\texttt{k}$}$ is any almost surely finite $L$-time for $Z^\dagger$.
\end{itemize}
\end{prop}
\begin{proof}We break the proof of \eqref{claim} in to several steps.

\smallskip

Step 1: At the heart of our proof is weak duality for L\'evy processes killed on hitting sets (in our case, the singleton $\{0\}$) with respect to Lebesgue measure.  Theorem II.1.5 of Bertoin \cite{bertoin} gives us Hunt's  classical duality relation
\begin{align}\label{classicdual}
	p_{X^\dagger}(t,y, {dz}){dy}  =p_{{\hat X}^\dagger}(t,z, {dy}){dz},\qquad y,z\in\mathbb{R}\text{ and }t\geq 0,
\end{align}
 where $p_{X^\dagger}$ is the transition kernel associated to the transition semigroup of the stable process killed at $0$ and $p_{{\hat X}^\dagger}$ is the transition kernel associated to the dual stable process killed at $0$.\smallskip
 
 Step 2: Defining $m(dx)=h(x)dx$ and combining \eqref{classicdual} and   \eqref{updownCOM} with the general formula `$p^h(t,x,dy)={h(y)} p(t,x,dy)/{h(x)}$' for transition kernels of $h$-transformed Markov processes, we obtain
\begin{align}\label{verifiessufficetocheck}
	p_{X^\dagger}(t,y, {dz})m({dy}) &= \frac{h(y)}{h(z)}p_{X^\dagger}(t,y, {dz})h(z){dy} \notag\\
	&=\frac{h(y)}{h(z)}p_{{\hat X}^\dagger}(t,z, {dy})h(z){dz}\notag\\
	&= p_{{\hat{X}^\circ}}(t,z, {dy})m(dz)
\end{align}
 for $y,z\in\mathbb{R}$ 
  and $t\geq 0$. Here $p_{\hat{X}^\circ}(z, dy, t)$ denotes the transition kernel associated to the transition semigroup of ${\hat{X}^\circ}$. Hence, the transition kernels of  $X^\dagger$ and ${\hat{X}^\circ}$ are in weak duality on $\R$ with respect to $m(dx)$.\smallskip

Step 3: The claim \eqref{claim} can now be deduced from general theory for random time-changes. Theorem 4.5 of Walsh \cite{W} states that two Markov processes in weak dualiy remain so when time-changed by the inverse of the same additive functional. The new duality measure is what is known as the Revuz measure of the additive functional with respect to the former duality measure (definition given shortly below). To apply Walsh's result, recall from the definitions, that $Z^\dagger$ (resp. $\hat{Z}^\circ$) are time-changes of $X^\dagger$ (resp. $\hat{X}^\circ$) with the inverse of the additive functional
\begin{align}\label{kk}
	A_t(\omega)= \int_0^t \sigma(\omega_s)^{-\alpha}ds, \qquad t\geq 0,
\end{align}
on the path space $\mathbb{D}([0,\infty), \mathbb{R})$. Theorem 4.5 of Walsh \cite{W} implies that $Z$ and $\hat{Z}^\circ$ are in weak duality with respect to the Revuz measure $\mu$ defined by 
\begin{align}
\int_{\mathbb R}f(x)\mu(dx) = \lim_{t\downarrow0}\int_{\mathbb{R}}m(dz)\mathbb{E}_z\left[ \frac{1}{t}\int_0^t f(X^\dagger_s) dA_s\right]
\label{assured}
\end{align}
for  $f\geq 0 $ bounded and measurable. In order to identify $\mu$, given that the limit \eqref{assured} is assured in Walsh \cite{W}, we can check with the help of Fubini's Theorem, for continuous and compactly supported $f\geq0$, that
\begin{align}
\int_{\mathbb R}f(x)\mu(dx) 
& = \lim_{t\downarrow0} \int_{\mathbb{R}}m(dz)\int_{\mathbb{R}} f(x)\sigma(x)^{-\alpha}\frac{1}{t}\int_0^t p_{X^\dagger}(s, z, dx)ds \notag\\
& = \lim_{t\downarrow0} \frac{1}{t}\int_0^t  \int_{\mathbb{R}}\int_{\mathbb{R}}m(dz)p_{X^\dagger}(s, z, dx)f(x)\sigma(x)^{-\alpha}  \,ds \notag\\
&=\lim_{t\downarrow0} \frac{1}{t}\int_0^t  \int_{\mathbb{R}}\int_{\mathbb{R}}m(dx)p_{\hat{X}^\circ}(s, x, dz)f(x)\sigma(x)^{-\alpha}  \,ds\notag\\
&=\int_{\mathbb{R}}f(x)\sigma(x)^{-\alpha}h(x)dx,
\label{walsh}
\end{align}
where in the third equality we used duality from Step 2 and in the fourth equality we use the fact that $\hat{X}^\circ$ is a conservative process and so $\int_\R p_{\hat{X}^\circ}(s, x, dz) =1$. This finishes the proof of the claim \eqref{claim}.\smallskip

To prove the time-reversal statements (i) and (ii) we check the conditions of Nagasawa's Theorem \ref{Ndual} appealing to the duality established in \eqref{claim}.\smallskip

(i) In order that (A.3.1) holds in the present setting we need to verify that  
\begin{align*}
	\mu(dy) = \int_{\mathbb{R}}\nu({dx})G_{{{\hat{Z}^\circ}}}(x, {dy})\quad \text{ on } \mathcal B(\R),
\end{align*}	
	 where  $\nu= \delta_0$ and $G_{{\hat{Z}^\circ}}(x, {dy})$ is the potential measure of ${\hat{Z}^\circ}$ on $\mathcal B(\R)$ and $\mu$ is the duality measure from \eqref{claim}. To this end, we first calculate the potential measure of ${\hat{X}^\circ}$, denoted by $ G_{\hat{X}^\circ}(0,dy)$. We use for the second equality the very last (unmarked) formula in Section 4.4 of Kuznetsov et al. \cite{KKPW}, Fubini's theorem and substitution to calculate, for bounded and measurable $f\geq 0$,
  \begin{align}\label{hilf2}
 G_{{{\hat{X}^\circ}}}[f](0)&=\int_0^\infty \hat{\mathbb{E}}^\circ_0\big[f({\hat{X}^\circ}_t)\big]\,dt\notag\\
&=\Gamma(-\alpha)\frac{\sin(\alpha \pi \rho)}{\pi}\int_0^\infty {\bf E}_1\big[I_\infty^{-1}f(-(t/{I_\infty})^{1/\alpha})\big]\,dt\notag\\
&\quad +\Gamma(-\alpha)\frac{\sin(\alpha  \pi\hat \rho)}{\pi}\int_0^\infty {\bf E}_{-1}\big[I_\infty^{-1}f((t/{I_\infty})^{1/\alpha})\big]\,dt\notag\\
&=\Gamma(-\alpha)\frac{\sin(\alpha \pi \rho)}{\pi}\int_0^\infty f(-u^{1/\alpha})\,du +\Gamma(-\alpha)\frac{\sin(\alpha  \pi\hat \rho)}{\pi}\int_0^\infty f(u^{1/\alpha})\,du\notag\\
&=\int_\R f(x) h(x)\,dx,
 \end{align}
  with $h$ from \eqref{h} and $I_\infty := \int_0^\infty e^{\alpha\xi_s}ds$ for the underlying MAP $(\xi,J)$, see Section \ref{sec:rssMp}. It follows that $G_{\hat{X}^\circ}(0,dy) = h(y){dy}$ on $\mathcal B(\R)$.

  Since by definition ${\hat{Z}^\circ}$ is a time-change of $\hat{X}^\circ$, from the above we can easily compute the potential measure of $\hat{Z}^\circ$ issued from $0$ by change of variables and the explicit form of $\iota$,
 \begin{align}\label{nupot}
G_{{{\hat{Z}^\circ}}}[f](0)
&=\hat{\mathbb{E}}^\circ_0\left[\int_0^\infty f({\hat{X}^\circ}_{\iota_t})\,dt\right]\notag\\
&=\hat{\mathbb{E}}^\circ_0\left[\int_0^\infty f({\hat{X}^\circ}_t)\sigma({\hat{X}^\circ}_t)^{-\alpha}\,dt\right]\notag\\
&=G_{{{\hat{X}^\circ}}}[f\sigma^{-\alpha}](0)\notag\\
&=\int_\R f(x)\sigma(x)^{-\alpha} h(x)\,dx\notag\\
&=\int_\R f(x)\mu(dx),
\end{align}
for bounded and measurable $f\geq 0$. 
Hence, we obtain that $\mu(dy)=\int_{\mathbb{R}}\delta_0(dx)G_{{\hat{Z}^\circ}}(x, {dy})$ as claimed. Combined with \eqref{claim} we verified assumption \textbf{(A.3.1)} in the current context. The remaining condition of Nagasawa's Theorem \ref{Ndual} are trivially fulfilled since all processes involved have c\`adl\`ag  trajectories. Therefore part (i) of the theorem now follows from Nagasawa's duality theorem.\smallskip

\medskip

(ii) We first show that
\begin{align}\label{and}
	\mu(dy) = \int_{\mathbb{R}}\nu({dx})G_{Z^\dagger}(x, {dy}) \quad \text{ on } \mathcal B(\R),
\end{align}	
	 where  $\nu= \delta_{\pm\infty}$ and $G_{Z^\dagger}(x, {dy})$ is the potential measure of $Z^\dagger$ on $\mathcal B(\R)$ and $\mu$ is the duality measure from \eqref{claim}. To do this, let us first prove that
\begin{align}\label{hilf}
	G_{Z^\dagger}(\pm\infty, A) = \lim_{|x|\to\infty}G_{Z^\dagger}(x,A),\quad \forall A\in \mathcal B(\R),
\end{align}
and then compute the righthand side explicitly. For a bounded Borel set  $A\subset[-{L},{L}]$ and $|x|>{L}$ or $x=\pm\infty$, we have  by the strong Markov property that
\begin{align*}
	G_{Z^\dagger}(x,A) = \int_{-{L}}^{L} G_{Z^\dagger}(z, A){\rm P}_x(Z_{T^{(-{L},{L})}}\in dz),
\end{align*}
where $T^{(-{L},{L})}=\inf\{t\geq 0: |Z_t|\leq {L}\}$. As we have assumed that ${\pm\infty}$ is an entrance point for $Z$, we also have the weak convergence 
\begin{align}\label{hg}
	{\rm P}_{\pm\infty}(Z_{T^{(-{L},{L})}}\in dz) =\lim_{|x|\to\infty}{\rm P}_{x}(Z_{T^{(-\infty,L]}}\in dz)\quad \text{ on }\mathcal B(\R).
\end{align}
See for example Chapter 13 of \cite{Whitt}, using the regularity of stable processes. The claim \eqref{hilf} now follows from the weak convergence \eqref{hg} if $z\mapsto G_{Z^\dagger}(z, A)$ is bounded and continuous on $[-{L},{L}]$. The boundedness and continuity comes from  the explicit form of $G_{X^\dagger}$. The latter is given in  
Theorem II.4.3 of Kyprianou \cite{KALEA}, who proved that $G_{X^\dagger}(x,{dy})$ has a density, say $g_{X^\dagger}(x,y)$, such that 
\begin{align}
g_{X^\dagger}(x,y) &=-\frac{\Gamma(1-\alpha)}{\pi^2}\left(|y|^{\alpha-1}s(y)
 - |y-x|^{\alpha-1} s(y-x) +|x|^{\alpha-1}s(-x)\right),
 \label{daggerpot}
\end{align}
where 
$
s(x) =  \sin(\pi\alpha\rho)\mathbf{1}_{(x>0)} +   \sin(\pi\alpha\rhohat)\mathbf{1}_{(x<0)}$. It follows that, 
\begin{align}
	G_{Z^\dagger}(x,A)&={\rm E}_x\left[\int_0^\infty \mathbf{1}_A(Z^\dagger_s)\,ds\right]\notag\\
	&=\E_x\left[\int_0^\infty \sigma(X^\dagger_s)^{-\alpha} \mathbf{1}_A(X^\dagger_s)\,ds\right]\label{nextblock}\\
	&=G_{X^\dagger}[\sigma^{-\alpha} \mathbf{1}_A](x)\notag\\
	&=-\frac{\Gamma(1-\alpha)}{\pi^2}\int_A\left(|y|^{\alpha-1}s(y) - |y-x|^{\alpha-1} s(y-x) +|x|^{\alpha-1}s(-x)\right)\sigma(y)^{-\alpha}d y.\notag
\end{align}
Using the time-change in \eqref{timechangesolution} with killing at the origin in the potential measure of $Z^\dagger$ together with \eqref{nextblock} and the Riesz--Bogdan--\.Zak transform in Theorem \ref{RBZthrm}, we have, for any bounded open set $A$,
\begin{align}\label{check1}
	G_{Z^\dagger}(\pm\infty,A) 
	 &= \lim_{|x|\to\infty}\mathbb{E}_x\left[\int_0^\infty \mathbf{1}_A(X^\dagger_t)\sigma(X^\dagger_t)^{-\alpha}|X^\dagger_t|^{2\alpha} |X^\dagger_t|^{-2\alpha}dt\right]\notag\\
	&= \lim_{|x|\to\infty}\mathbb{E}_x\left[\int_0^\infty \mathbf{1}_A(X^\dagger_{\eta_s})\sigma(X^\dagger_{\eta_s})^{-\alpha}|X^\dagger_{\eta_s}|^{2\alpha} ds\right]\notag\\
	&=\lim_{|x|\to\infty}\hat{\mathbb{E}}^\circ_{1/x}\left[\int_0^\infty \mathbf{1}_A(1/X_{s})\sigma(1/X_{s})^{-\alpha}|X_{s}|^{-2\alpha} ds\right]\notag\\
	&= G_{\hat{X}^\circ}[g](0),
\end{align}
where $g(x) = \mathbf{1}_A(1/x)\sigma(1/x)^{-\alpha}|x|^{-2\alpha}$. The righthand side was already computed in \eqref{hilf2} as
\begin{align}\label{check2}
	G_{\hat{X}^\circ}[g](0) &= \int_\mathbb{R}\mathbf{1}_A(1/x)\sigma(1/x)^{-\alpha}|x|^{-2\alpha}h(x){dx}\notag\\
&= \int_A\sigma(z)^{-\alpha}|z|^{2(\alpha-1)}h(1/z){dz}\notag\\
&=\int_A\sigma(z)^{-\alpha}h(z){dz},
\end{align}
where in the final equality we used the explicit form of $h$  in \eqref{h} to check that $|z|^{2(\alpha-1)}h(1/z) = h(z)$, for $z\neq 0$. Putting \eqref{check1} and \eqref{check2} together  gives us 
\[
\int_{\mathbb{R}}\nu({dx})G_{Z^\dagger}(x, A) =	 G_{Z^\dagger}(\pm\infty,A)  = \int_{A} \sigma(z)^{-\alpha}h(z){dz}=\mu(A),
\]
which is \eqref{and}. 

The final step consists in deducing from \eqref{and} and Nagaswa's duality theorem the statement of (ii). We note that duality \eqref{claim} of the semigroups was proved in $\R\backslash \{0\}$ only, but we apply Nagasawa's theorem to $\underline{\overline{\R}} \backslash \{0\}$. This is justified by extending the duality measure $\mu$ with zero mass at the additional state. All conditions hold trivially with this extension. The claim of part (ii)  now follows from Theorem \ref{Ndual} as in (i).
\end{proof}
The next proposition offers an analogous result to the first one, but now with respect to entrance from $+\infty$ (resp. $-\infty$) in the spectrally positive (resp. negative) setting. Accordingly, the $h$-transformed process that is involved in the proposition is taken as the stable process conditioned to be positive (resp. negative) which was defined in \eqref{CSP}.

\begin{prop}\label{nag2}
	Suppose that $X$ is  a spectrally positive stable processes with $\alpha\in (1,2)$. Suppose that 
 ${{\hat{X}^\uparrow}}$ has probabilities  $\hat{\mathbb{P}}^\uparrow_x$, $x\geq 0$. Define $\hat{Z}^\uparrow_t = {{\hat{X}^\uparrow}}_{\iota_t}$, $t\geq 0$, where the time-change $\iota$ is given by 
\begin{align}\label{tauhatb}
	\iota_t = \inf\left\{s>0 : \int_0^s \sigma({{\hat{X}^\uparrow}}_s)^{-\alpha}ds > t\right\}, \qquad t< \int_0^\infty \sigma({{\hat{X}^\uparrow}}_s)^{-\alpha}ds.
\end{align}

Suppose that $Z^\dagger$ is the process \eqref{timechangesolution} killed on hitting the origin. Then 
\begin{align}\label{claimb}
	Z^\dagger\text{ and } \hat{Z}^\uparrow \text{ are in weak duality with respect to }  \mu(dx)=\sigma(x)^{-\alpha}h(x)dx\text{ on }[0,\infty),
\end{align}
where $h$ is given by \eqref{speconesidedh}.
Moreover:
\begin{itemize}
\item[(i)] The time-reversed process $\hat{Z}^\uparrow_{(\mbox{\rm $\texttt{k}$} - t)-}$, $t\leq \mbox{\rm $\texttt{k}$}$, with $\hat{Z}^\uparrow_0 = 0$, is a time-homogenous Markov process with transition semigroup which agrees with that of $Z$, where $\mbox{\rm $\texttt{k}$}$ is any almost surely finite $L$-time for $\hat{Z}^\uparrow$.
\item[(ii)] If $+\infty$ is an entrance point for $Z^\dagger$, then the time reversed process $Z^\dagger_{(\mbox{\rm $\texttt{k}$}-t)-}$, $t\leq \mbox{\rm $\texttt{k}$}$, with $Z^\dagger_0 = +\infty$ is a time-homogenous Markov process with transition semigroup which agrees with that of $\hat{Z}^\uparrow$, where $\mbox{\rm $\texttt{k}$}$ is any almost surely finite $L$-time for $Z$.
\end{itemize}
\end{prop}

\begin{proof}
The proof is similar to that of Proposition \ref{nag} but needs adjustment of the involved $h$-transforms. Consequently, different (more classical) results from fluctuation theory are needed. 

\smallskip

We first deal with \eqref{claimb}.
The analogues to Step 1 and Step 2 of the proof of Proposition \ref{nag} are due to Theorem 1 of Bertoin and Savov \cite{BS}. That is to say, the transition measures of $X^\dagger$ and $\hat X^\uparrow$ are in weak duality with respect to $m(dx)=h(x)dx =\Gamma(\alpha)^{-1}x^{\alpha-1}dx$. The analogue of Step 3  in the proof of Proposition \ref{nag} is the same here and hence \eqref{claimb} is verified without appealing to any other further results from fluctuation theory.\smallskip

(i) In order that (A.3.1) holds in the present setting we need to verify that  
\begin{align*}
	\mu(dy) = \int_{[0,\infty)} \nu(da)G_{{{\hat{Z}^\uparrow}}}(a,dy)\quad \text{ on } \mathcal B(\R_+),
\end{align*}	
where $\nu= \delta_0$ and $G_{{{\hat{Z}^\uparrow}}}$ is the potential measure of $\hat{Z}^\uparrow$ and $\mu$ is the duality measure from \eqref{claimb}. We first  calculate  the potential measure $G_{\hat{X}^\uparrow}(0,dy)$ of ${\hat{X}^\uparrow}$ on $\mathcal B(\R)$. Using the expression for the entrance law of $X^\uparrow$ in Theorem 1 of   Bertoin and Yor \cite{BY02} (see also Remark 4 of Chaumont et al. \cite{CKPR}), Fubini's theorem and substitution, we calculate, for bounded and measurable $f\geq 0$, 
  \begin{align}\label{hilf2b}
 G_{{{\hat{X}^\uparrow}}}[f](0)&=\int_0^\infty \hat{\mathbb{E}}^\uparrow_0[f({\hat{X}^\uparrow}_t)]\,dt\notag\\
&=\frac{1}{\alpha {\bf E}[\xi^\uparrow_1]}\int_0^\infty {\bf E}\left[\hat{I}_\infty^{-1}f(-(t/{\hat{I}_\infty})^{1/\alpha})\right]\,dt\notag\\
&=\frac{1}{\Gamma(\alpha)}\int_\R f(x)x^{\alpha-1}\,dx\notag\\
&=\int_\R f(x)h(x)dx
 \end{align}
 with $h$ from \eqref{h} and $\hat{I}_\infty = \int_0^\infty e^{\alpha\hat{\xi}^\uparrow_s}ds$, where the L\'evy process $({\hat{\xi}}^\uparrow, {\bf P})$ is characterised by its exponent in \eqref{xiuparrowpsi}. Thus, $G_{{{\hat{X}^\uparrow}}}(0,dx)=  h(x){dx}$, $x\in\R$. 
 
 \smallskip
 
 Since ${\hat{Z}^\uparrow}$ is a time-change of $\hat{X}^\uparrow$, from the above we can easily compute the potential measure by change of variables, namely
 \begin{align}\label{nupot}
G_{{{\hat{Z}^\uparrow}}}(0, A)
&=\hat{\mathbb{E}}^\uparrow_0\left[\int_0^\infty \mathbf{1}_A({\hat{X}^\uparrow}_{\iota_t})\,dt\right]\notag\\
&=\hat{\mathbb{E}}^\uparrow_0\left[\int_0^\infty \mathbf{1}_A({\hat{X}^\uparrow}_t)\sigma({\hat{X}^\uparrow}_t)^{-\alpha}\,dt\right]\notag\\
&=\int_A \sigma(x)^{-\alpha}h(x)\,dx,\notag\\
&=\mu(A)
\end{align}
for bounded and open sets $A$. Hence, we obtain that $\int_{\mathbb{R}}\nu(dx)G_{{\hat{Z}^\uparrow}}(x, {dy})=\mu(dy)$ which is the condition \textbf{(A.3.1)} of Theorem \ref{Ndual}. Noting again that the other conditions of Nagasawa are trivially fulfilled since all processes involved have c\`adl\`ag  trajectories, the proof is now complete.\smallskip

(ii) In order that (A.3.1) holds in the present setting we need to verify that  
\begin{align*}
	\mu(dy) = \int_{\mathbb{R}}\nu({dx})G_{Z^\dagger}(x, {dy}) \quad \text{ on } \mathcal B(\R_+),
\end{align*}	
	 where  $\nu= \delta_{+\infty}$ and $G_{Z^\dagger}(x, {dy})$ is the potential measure of $Z^\dagger$ on $\mathcal B(\R_+)$ and $\mu$ is the duality measure from \eqref{claimb}. As in the proof of Proposition \ref{nag} (ii), it is straightforward to show that 
\begin{align}\label{hilfb}
	G_{Z^\dagger}(+\infty, A) = \lim_{x\to+\infty}G_{Z^\dagger}(x,A),\quad \forall A\in \mathcal B(\R_+).
\end{align}
Indeed, for $x>{L}$ or $x=+\infty$ and $A\subset[0,{L}]$ we have (recall that $X$ is assumed to only creep downwards) by the strong Markov property that
\begin{align*}
	G_{Z^\dagger}(x,A) = \int_{0}^{L}{\rm P}_x(Z_{T^{(-\infty,L]}}\in dz)G_{Z^\dagger}(z, A)= G_{Z^\dagger}({L},A),
\end{align*}
where $T^{(-\infty,L]}=\inf\{t\geq 0: Z_t\leq {L}\}=\inf\{t\geq 0: Z_t= {L}\}$. The claim \eqref{hilfb} now follows. 

Now we combine \eqref{hilfb}, the time-change in Theorem \ref{zthrm} and the Chaumont's transform in Theorem \ref{CRBZ}, to get, for bounded and open set $A$, 
\begin{align*}
G_{Z^\dagger}(+\infty, A)
&= \lim_{x\to\infty}\mathbb{E}_x\left[\int_0^\infty \mathbf{1}_A(X^\dagger_t)(X_t^\dagger)^{2\alpha} (X^\dagger_t)^{-2\alpha}\sigma(X^\dagger_t)^{-\alpha}dt\right]\\
&= \lim_{x\to\infty}\mathbb{E}_x\left[\int_0^\infty \mathbf{1}_A(X^\dagger_{\gamma_s})(X^\dagger_{\gamma_s})^{2\alpha} \sigma(X^\dagger_{\gamma_s})^{-\alpha}ds\right]\\
&=\lim_{x\to\infty}\hat{\mathbb{E}}^\uparrow_{1/x}\left[\int_0^\infty \mathbf{1}_A(1/X_{s})X_{s}^{-2\alpha}\sigma(1/X_s)^{-\alpha} ds\right]\\
&= G_{\hat{X}^\uparrow}[g](0),
\end{align*}
where $g(x) = \mathbf{1}_A(1/x)x^{-2\alpha}\sigma(1/x)^{-\alpha}$ and (as above) the continuity at the origin of $G_{\hat{X}^\uparrow}$ is a consequence of $0$ being an entrance boundary for $\hat X^\uparrow$, see Lemma \ref{zeroenter2}. The righthand side was already computed in \eqref{hilf2b}. Plugging it into the right-hand side above gives us  for bounded $A\in\mathcal{B}(\R)$, 
\begin{align*}
G_{Z^\dagger}(+\infty, A) 
&= \int \mathbf 1_A(1/x)x^{-2\alpha}\sigma(1/x)^{-\alpha}h(x){dx}\\
&= \int \mathbf 1_A(z)z^{2(\alpha-1)} \sigma(z)^{-\alpha}h(1/z){dz}\\
&=\int_A \sigma(z)^{-\alpha}h(z){dz},
\end{align*}
where in the final equality we used the explicit form of $h$ to obtain $z^{2(\alpha-1)}h(1/z) = h(z)$ for $z\neq 0$.
 We can now conclude that $G_{Z^\dagger}(+\infty, dy)  = \mu(dy)$ on $\mathbb{R_+}$ which is the condition \textbf{(A.3.1)} of Theorem \ref{Ndual}. The claim in part (ii) now follows from Theorem \ref{Ndual} of Nagasawa as before with the same slight adjustment mentioned in the final paragraph at the end of the proof of Proposition \ref{nag}.
\end{proof}

\section{Entrance from infinity, the impossible cases}\label{impossible}
This first section of the main proof gathers the cases where entrance from infinity is impossible irrespectively of $\sigma$, i.e. a cross appears in the table of Theorem \ref{main}. Recall that entrance stands for enterable but not exit. All proofs are indirect and based on the triviality of certain limiting hitting distributions (overshoots, inshoots) of stable processes for which explicit formulas are available.\smallskip

 Recall that for $x\in \R$, ${\rm P}_x$ denotes the law of the unique weak solution to the SDE \eqref{2} issued from $x$, $\P_x$ denotes the law of the stable process issued from $x\in\R$ and ${\rm P}_x$ can be expressed via the time-change \eqref{timechangesolution} in terms of $\P_x$. To study ${\rm P}$ for infinite entrance points we use the strong Markov property (consequence of Feller assumption) at first hitting times and then use Proposition \ref{pr} to obtain formulas in terms of the stable process. First hitting distributions under ${\rm P}_x$ are identical to those under $\P_x$ as the time-change does not influence the jump sizes. Note that, since $\sigma>0$ is assumed continuous, $\sigma$ is bounded away from zero within all compact sets. Hence, the time-change in \eqref{timechangesolution} does not level off in $\R$ so that solutions to the SDE \eqref{2} visit the same sets as the driving stable process.

\subsection{Entrance from $+\infty$, two-sided jumps, $\alpha\in(0,2)$}\label{proof0}
In this first proof we show that divergence of overshoots for stable L\'evy processes implies that under ${\rm P}_{+\infty}$ trajectories would jump instantaneously from $+\infty$ to $-\infty$ which contradicts continuous entry. We consider $\overline{\R}=(-\infty,+\infty]$ and assume $({\rm P}_x,{x\in\bar \R})$ is a Feller extension of $({\rm P}_x, x\in\R)$, satisfying ${\rm P}_{+\infty}(\lim_{t\downarrow 0} Z_t=+\infty)=1$. 
Recall that the Feller property of the extension implies the strong Markov property which we apply to the first hitting times $T^{(-\infty,L]}=\inf\{t\geq 0: Z_t\leq {L}\}$ for ${L}\in \R$.

\smallskip

Using the time-change representation \eqref{timechangesolution}, we find that, for all ${{L}}\in \N$ and compact sets $A\subset \R$,
\begin{align}\label{ab}
	\lim_{z\to+\infty}{\rm P}_z(Z_{T^{{(-\infty,L]}}}\in A)= \lim_{z\to+\infty}\P_z(X_{\tau^{(-\infty,{{L}}]}}\in A)=0,
\end{align} 
where $\tau^{(-\infty,{{L}}]} = \inf\{t<0: X_t\leq L\}$ and we have used that the ranges of $Z$ and $X$ agree and the fact  that 
$X$ has no stationary overshoots (recall the discussion around \eqref{C2}).
This last claim can be verified directly by recalling the classical result which states that 
\begin{align}\label{leif}
\P_z(X_{\tau^{(-\infty,{{L}}]}}\leq {{L}}-y) = \frac{\sin(\pi\alpha\rhohat)}{\pi}\int_0^{y/(z-{{L}})}t^{-\alpha\rho}(1+t)^{-1}dt, \qquad z\leq {{L}}.
\end{align}
See for example Equation (2) of  Rogozin \cite{Rog} for the above formula.
For bounded and  measurable $A$, define the auxiliary function
\begin{align*}
	f(z)=\begin{cases}
		{\rm P}_z(Z_{T^{{(-\infty,L]}}}\in A)&\text{  if }z\in\R,\\
		0&\text{  if }z=+\infty,
	\end{cases}
\end{align*}
so that  $0\leq f\leq 1$. Thanks to \eqref{ab} and the explicit overshoot distribution \eqref{leif},  $f$ is continuous on $\overline{\R}$.
Hence, for every $\epsilon>0$, there is some ${{L}}$ so that $0\leq f(z)\leq \epsilon$ for all $z>{{L}}$. Applying the strong Markov property at $T^{(-\infty,L']}$ for ${{L}'}>{{L}}$ gives 
\begin{align*}
	{\rm P}_{+\infty}(Z_{T^{{L}}}\in A)
	&=\lim_{{{L}'}\to +\infty} \int {\rm P}_y(Z_{T^{{L}}}\in A)\, {\rm P}_{+\infty} (Z_{T^{(-\infty,L']}}\in dy)\\
	&=\lim_{{{L}'}\to +\infty}\left( \int_{y>{{L}}} f(y)\, {\rm P}_{+\infty} (Z_{T^{(-\infty,L']}}\in dy)+\int_{y\leq {{L}}} f(y)\, {\rm P}_{+\infty} (Z_{T^{(-\infty,L']}}\in dy)\right)\\
	&\leq \epsilon+\lim_{{{L}'}\to +\infty}{\rm P}_{+\infty} (Z_{T^{(-\infty,L']}}\leq {{L}})\\
	&=\epsilon,
\end{align*}
where the final equality follows since trajectories enter from infinity continuously by assumption. Hence, $\lim_{{L}'\to \infty}{\rm P}_{+\infty}(Z_{T^{(-\infty,L']}}\in A)=0$ for every bounded and measurable subsets $A$ of $\R$ which implies that under ${\rm P}_{+\infty}$ no compact subset of $\R$ is visited.
\subsection{Entrance from $-\infty$, two-sided jumps, $\alpha\in(0,2)$}\label{proof1}
The proof follows the same lines as before with $\underline{\R}=[-\infty,+\infty)$, replacing $T^{(-\infty,L]}$ by $T^{[L,\infty)}=\inf\{t\geq 0: Z_t\geq {L}\}$ and using, for all ${{L}}\in \N$ and $A\subset [-{{L}},{{L}}]$, the continuous function
\begin{align*}
	f(z)=\begin{cases}
		{\rm P}_z(Z_{T^{[L,\infty)}}\in A)&\text{  if }z\in\R,\\
		0&\text{  if }z=-\infty,
	\end{cases}
\end{align*}
with analogous formulas forcing an instantaneous jump from $-\infty$ to $+\infty$

\subsection{Entrance from $\pm \infty$, two-sided jumps,  $\alpha\in(0,1)$}\label{proof2}
The proof follows the same idea as in Section \ref{proof0} replacing overshoots by `inshoots' into compact intervals and then using that transience of stable processes for $\alpha\in(0,1)$ does not allow to reach arbitrary compact sets from infinity.

\smallskip

The differences are the use of first hitting times $T^{(-{L},{L})}=\inf\{t\geq 0: Z_t\in (-{L},{L})\}$, the auxiliary function
\begin{align*}
	f(z)=\begin{cases}
		{\rm P}_z(Z_{T^{(-L,L)}}\in A)&\text{  if }z\in\R,\\
		0&\text{  if }|z|=\pm\infty,
	\end{cases}
\end{align*}
on $\overline{\underline \R}$ and the argument for continuity of $f$. Here, $f$ is continuous in the interior of $\overline{\underline \R}$ due to the explicit form of ${\rm P}_z(Z_{T^{(-L,L)}}\in A) = \mathbb{P}_z(X_{\tau^{(-{{L}},{{L}})}}\in A) = \mathbb{P}_{z/{{L}}}(X_{\tau^{(-1,1)}}\in A/{{L}})$ given in Theorem 1.1 of Kyprianou et al. \cite{KPW}.  Specifically, it says that, for $\alpha\in(0,1)$,
\begin{align}
&\p_{x}\bigl(X_{\tau^{(-1,1)}} \in  d y\bigr)
= \frac{\sin(\pi\alpha\rhohat)}{\pi}
    (1+x)^{\alpha\rho}(1+y)^{-\alpha\rho}
    (x-1)^{\alpha\rhohat}
    (1-y)^{-\alpha\rhohat}
    (x-y)^{-1} d y.
    \label{striplowalpha}
\end{align}
Continuity of $f$ at $\pm\infty$ is due to the transience of stable L\'evy processes for $\alpha\in(0,1)$, so that, using again the time-change representation \eqref{timechangesolution}, $\lim_{|z|\to\infty} {\rm P}_z(Z_{T^{(-L,L)}}\in A)=0$ which implies that under ${\rm P}_{\pm\infty}$ no compact subset of $\R$ is visited.

\subsection{Entrance from $\pm \infty$ or $-\infty$, spectrally positive jumps, $\alpha\in(0,2)$}\label{proof4} 

First note that spectral positivity excludes the case that $\alpha = 1$ (which is necessarily symmetric). We therefore only need to deal with the cases $\alpha\in(0,1)\cup(1,2)$.

\smallskip

On account of the fact that we know the law of the overshoot of $X$ into $({L},\infty)$, see e.g. again Rogozin \cite{Rog},  we can apply a similar argument to the one in \eqref{ab} and deduce that
$
	\lim_{z\to-\infty}{\rm P}_z(Z_{T^{[L,\infty)}}\in A)=0
$
for all compact sets $A$ so that entrance from $-\infty$ is impossible.

\smallskip

Next, we consider the limit of ${\rm P}_z(Z_{T^{(-L,L)} }\in A)$ as $|z|\to\infty$ for all compact sets $A$.
When $\alpha\in(0,1)$, the process $X$ is a subordinator and hence the paths of $Z$ are monotone increasing. Therefore the aforesaid limit does not exist.
On the other hand, when $\alpha>1$, we can appeal to the spectrally positive analogue of \eqref{striplowalpha}, see Proposition 1.3 of \cite{KPW} or \cite{Port67}. This tells us that, for $z<-1$,
\begin{align*}
  \p_z(X_{\tau^{(-1,1)}} \in d y)&= \frac{\sin \pi(\alpha-1)}{\pi}
  (|z|-1)^{\alpha-1} (1+y)^{1-\alpha} (|z|+y)^{-1} \dd y \\
  & \quad +\delta_{-1}(\dd y) \frac{\sin \pi(\alpha-1)}{\pi}
  \int_0^{\frac{|z|-1}{|z|+1}} t^{\alpha-2} (1-t)^{1-\alpha} \, d t,
\end{align*}
and $\p_z(X_{\tau^{(-1,1)}} =1) = 1$ for $z>1$ (positive jumps).
With the help of scaling, it is therefore clear that limits of ${\rm P}_z(Z_{T^{(-L,L)}}\in A) = \p_z(X_{\tau^{(-{L},{L})}} \in A)$ do not exist. Indeed, one need only compare the probabilities  ${\rm P}_z(Z_{T^{(-L,L)}}={L}) $ as $z\to\infty$ and $z\to-\infty$.

\subsection{Entrance from $\pm \infty$ or $+\infty$, spectrally negative jumps, $\alpha\in(0,2)$}\label{proof6}
The proof is analogous to the one above.

\subsection{Entrance from $+\infty$, spectrally positive jumps, $\alpha\in(0,1)$}\label{proof4.5}
By virtue of the increasing nature of the paths in this setting, entrance at $+\infty$ is trivially impossible.

\subsection{Entrance from $-\infty$, spectrally negative jumps, $\alpha\in(0,1)$}\label{proof8}
By virtue of the decreasing nature of the paths in this setting, entrance at $+\infty$ is trivially impossible.

\section{Entrance from $\pm\infty$, two-sided jumps, $\alpha\in(1,2)$}\label{proof3}
In this section we discuss the main arguments of the article for which we have seen significant   preparation in the earlier sections. Proofs of Section \ref{proof5} go along the lines.\smallskip

We break the proof into necessity and sufficiency of the integral test	
\begin{align}\label{integral}
	I^{\sigma,\alpha}(\R)=\int_\R \sigma(x)^{-\alpha}|x|^{\alpha-1}\,{dx}<\infty
\end{align}	
for $\pm\infty$ as an entrance point.\smallskip

\textit{Idea for necessity:} Suppose solutions enter from infinity. Since for $\alpha>1$ solutions will hit the origin almost surely (as they are time-changes of the stable process which hits points) we can time-reverse at the first hitting time of $0${\color{black}; see Figure \ref{fig2}}. From Proposition \ref{nag} we know the dynamics of the reversed process. It is a time-change under $\hat{\mathbb{P}}^\circ_{0}$, the stable process conditioned to avoid $0$. Since the conditioned process itself is conservative, necessarily the time-change \eqref{tauhat} needs to explode under $\hat{\mathbb{P}}^\circ_{0}$. Recall $\hat{\mathbb{P}}^\circ_{0}$ is well-defined due to Lemma \ref{zeroenter}. Hence, we obtain the necessity of an almost surely finite perpetual integral under $\hat{\mathbb{P}}^\circ_{0}$. Since the conditioned process is a self-similar Markov process we can employ the Lamperti representation for the positive part and the negative part (alternatively the Lamperti--Kiu transformation to the entire process) to get two almost surely finite perpetual integrals over two L\'evy processes with positive finite means and local times. For such perpetual integrals we can employ the article \cite{DK} to obtain an integral test which gives \eqref{integral}.\smallskip

\begin{figure}[h]
        \includegraphics[scale=0.5]{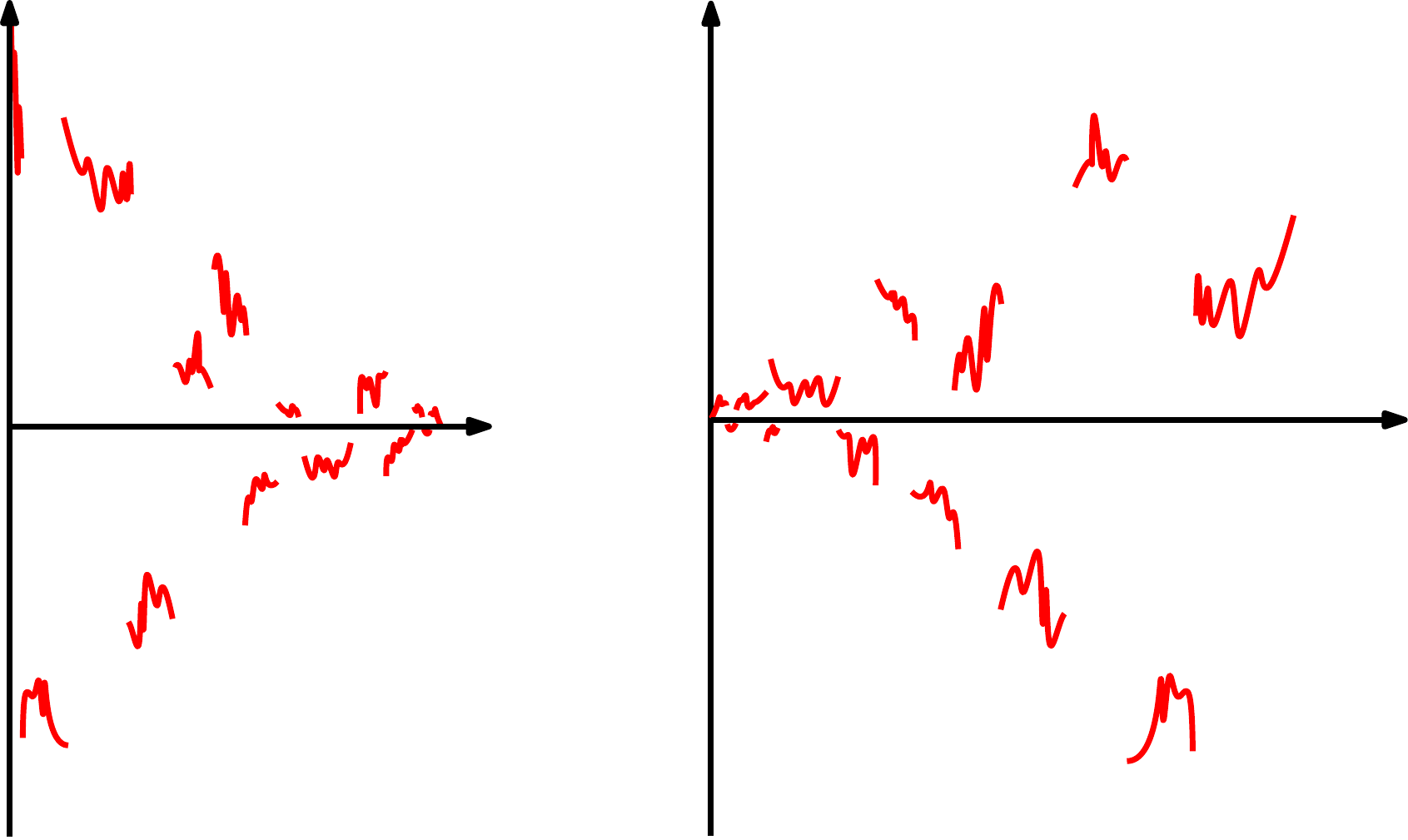}
\caption{Time-reversing SDE entering at $\pm\infty$ to give time-change of $h$-transform $\hat{X}^\circ$ entering at $0$.}
\label{fig2}
\end{figure}

\textit{Proof of necessity:} Let us assume $\pm\infty$ is an entrance point for the SDE \eqref{2} in the sense of Definition \ref{Fellerdef}.  Necessarily we must have under $\emph{\rm P}_{\pm\infty}$ that $T^{(-{L},{L})} = \inf\{t>0: |Z_t|<{L}\}<\infty$ with positive probability for some ${L} >0$ and that this probability tends to 1 as ${L}\to\infty$. From Proposition \ref{pr} and the recurrence of stable processes for $\alpha\in (1,2)$, we also know that $\zeta=\inf\{t>0: Z_t = 0\}$ is almost surely finite when $Z$ is issued from any point in $\R$ (this uses the assumption that $\sigma$ is positive and continuous, hence, the time-change cannot level off in $\R$). It follows by the strong Markov property that the first hitting time of zero $\zeta$ is finite almost surely under $\emph{\rm P}_{\pm\infty}$. We also note that $\zeta$ is an $L$-time for the SDE killed at $0$. Hence, we will consider the time-reversal under $\emph{\rm P}_{\pm\infty}$ from $\texttt{k} =\zeta$. 

\smallskip

As we have assumed that $\pm\infty$ is an entrance point for $\emph{\rm P}_{\pm\infty}$, Proposition \ref{nag} (ii) tells us that $\pm\infty$ is accessible in an almost surely finite time for $\hat{Z}^\circ$, where $\hat{Z}^\circ_ t = {{\hat{X}^\circ}}_{\iota_t}$ with $\hat{Z}^\circ_0 = 0$. The conservative process  ${{\hat{X}^\circ}}$ has probabilities  $\hat{\mathbb{P}}^\circ_x$, $x\in \mathbb{R}$,  and the time-change $\iota$ is given by 
\begin{align*}
	\iota_t = \inf\left\{s>0 : \int_0^s \sigma({{\hat{X}^\circ}}_s)^{-\alpha}ds > t\right\}, \qquad t<\int_0^\infty \sigma({{\hat{X}^\circ}}_s)^{-\alpha}ds.
\end{align*}
The finite-time accessibility of $\pm\infty$ for $\hat{Z}^\circ$ and the fact that $\hat{X}^\circ$ is a conservative process implies that 
the time-change has to explode in finite time or, equivalently,
\begin{align}\label{prove this}
	\int_0^\infty \sigma({{\hat{X}^\circ}}_s)^{-\alpha}ds<\infty
\end{align}
 almost surely under $\hat{\mathbb{P}}^\circ_{0}$. The first exit time of $\hat{X}^\circ$ from $(-\varepsilon,\varepsilon)$, for any $\varepsilon>0$, occurs before $\hat{X}^\circ$ reaches $\pm\infty$. Moreover, appealing to  \eqref{2sideexit} in combination with the  $h$-transform that defines $\hat{\mathbb{P}}^\circ_{0}$, it is clear that the law of the overshoot of $\hat{\mathbb{P}}^\circ_{0}$ outside of $(-\varepsilon, \varepsilon)$ is absolutely continuous with respect to Lebesgue measure. Hence, it follows that \eqref{prove this} is almost surely finite under $\hat{\mathbb{P}}^\circ_{x}$, for Lebesgue almost every  $x\in\mathbb{R}$. In what follows we continue with two such $x>0$ and $x<0$.\smallskip 
\smallskip

To show that the necessary almost sure finiteness in \eqref{prove this} implies  the finiteness of the integral test  \eqref{integral}, we need to introduce a path transformation of $\hat{X}^\circ$. We note that in the spirit of the example in Section \ref{pssmpexamples}, we can censor out the negative parts of the path of $\hat{X}^\circ$ to create a positive self-similar Markov process, say $\hat{X}^{\circ\hspace{-1pt}>}$. That is to say 
\begin{align}
	\hat{X}^{\circ\hspace{-1pt}>}_t = \hat{X}^\circ_{\hat{\gamma}^\circ_t},\quad \text{ with } \hat{\gamma}^\circ_t = \inf\left\{s>0 : \int_0^s \mathbf{1}_{(\hat{X}^\circ_u<0)}du>t\right\}.
	\label{Ycen}
\end{align}
Let us write $\hat{\xi}^{\circ\hspace{-1pt}>}$ for the L\'evy process appearing in Lamperti's representation \eqref{pssMpLamperti} of $\hat{X}^{\circ\hspace{-1pt}>}$.
The finiteness of \eqref{prove this} implies the almost sure finiteness of  the integrals
\begin{align}\label{censored}
\int_0^\infty \sigma({{\hat{X}^\circ}}_t)^{-\alpha}\mathbf{1}_{({{\hat{X}^\circ}}_t>0)}{dt} &= \int_0^\infty  \sigma({{\hat{X}^{\circ\hspace{-1pt}>}}}_s)^{-\alpha} {ds}\notag\\
&=\int_0^\infty  \sigma( e^{\hat{\xi}^{\circ\hspace{-1pt}>}_{\hat\varphi_u}})^{-\alpha} {du}\notag\\
&=\int_0^\infty  \sigma( e^{\hat{\xi}^{\circ\hspace{-1pt}>}_{v}})^{-\alpha}  e^{\alpha\hat{\xi}^{\circ\hspace{-1pt}>}_{v}}{dv}.
\end{align}
To the (almost surely finite) righthand side we will apply \cite{DK} to obtain the integral test \eqref{integral}. The result of \cite{DK} that we apply states the following: If $\xi$ is a L\'evy process with local times and finite positive mean, then
\begin{align*}
	\texttt{P}\left(\int_0^\infty f(\xi_s)\,ds<\infty\right)=1\quad \Longleftrightarrow \quad \int_0^\infty f(x)\,dx<\infty.
\end{align*}
We will now check that $\hat{\xi}^{\circ\hspace{-1pt}>}$ has local times (equivalently: $\hat{\xi}^{\circ\hspace{-1pt}>}$ hits points, compare for instance Theorem 7.12 of \cite{Kbook} and Theorem V.1 of \cite{bertoin}) and finite positive mean.\smallskip

(i) {\it Local times}. Note that, for the stable process, as $\alpha\in(1,2)$, we have $\hat{\mathbb{P}}_x(\tau^{\{y\}}<\infty)=1$ for all $x,y\in\R$, where $\tau^{\{y\}}= \inf\{t>0: X_t = y\}$. It follows from \eqref{updownCOM} (albeit with $X$ replaced by $X^\dagger$) that $\hat{\mathbb{P}}^\circ_x(\hat{\tau}_\circ^{\{y\}}<\infty)>0$ for all $x,y\in\R$, where $\hat{\tau}_\circ^{\{y\}} = \inf\{t>0: \hat{X}^\circ_t = y\}$. But then the censored processes hit points (same range) and also the L\'evy processes through the Lamperti transformation hit points (exponential change of space, time-change irrelevant). Hence, $\hat{\xi}^{\circ\hspace{-1pt}>}$ has local times.\smallskip

(ii) {\it Finite positive mean}. We can derive the characteristic exponent of $\hat{\xi}^{\circ\hspace{-1pt}>}$ from the characteristic exponent of, say $\hat\xi^{>}$, the L\'evy process that lies behind the stable process $\hat{X}^\dagger$, which has been negatively censored. Indeed, from \eqref{censoredpsi}, its characteristic exponent takes the form 
\begin{equation}
\hat{\Psi}^{>}(z) = \frac{\Gamma(\alpha\rhohat - \iu{z})}{\Gamma(-\iu{z})}
    \frac{\Gamma(1 - \alpha\rhohat + \iu{z})}{\Gamma(1 - \alpha + \iu{z})}, \qquad z\in \R.
    \label{censoredpsi2}
\end{equation}
On account of the fact that, for $t\geq 0$ fixed, $\omega\mapsto  \inf\left\{s>0 : \int_0^s \mathbf{1}_{(\omega_u<0)}du>t\right\}$ is a sequence of almost surely finite stopping times under $\hat{\mathbb{P}}_x$, $x\neq 0$,  as well as the same being true of the time-change in the Lamperti transform \eqref{pssMpLamperti} for the process $\hat{X}^{\circ\hspace{-1pt}>}$,
the Doob $h$-transform that defines $\hat{X}^\circ$ is tantamount to an Esscher transform (exponential change of measure) on $\hat\xi^{>}$. In particular, note that $\Psi_{\hat\xi^{>}}(-\iu(\alpha-1))=0$ and $\exp((\alpha-1)\hat\xi^{>}_t)$, $t\geq0$, is a $\hat{\mathbb{P}}$-martingale. It follows that the characteristic exponent of $\hat{\xi}^{\circ\hspace{-1pt}>}$ takes the form
\begin{equation}
\hat{\Psi}^{\circ\hspace{-1pt}>}(z) = \frac{\Gamma(1-\alpha\rho - \iu{z})}{\Gamma(1-\alpha -\iu{z})}
    \frac{\Gamma(\alpha\rho + \iu{z})}{\Gamma(\iu{z})} , \qquad z\in \R.
    \label{hypexp}
\end{equation}
By computing  $-{\rm i} \Psi'_{\hat{\xi}^{\circ\hspace{-1pt}>}}(0)$ we can verify directly that the mean of $\hat{\xi}^{\circ\hspace{-1pt}>}_1$ is finite.\smallskip

With local times and finite positive mean we apply Theorem 1 of \cite{DK} for which the starting value of $\hat{\xi}^{\circ\hspace{-1pt}>}$ is irrelevant. This tells us that 
\[
\int_0^\infty  \sigma( e^{\hat{\xi}^{\circ\hspace{-1pt}>}_{v}})^{-\alpha}  e^{\alpha\hat{\xi}^{\circ\hspace{-1pt}>}_{v}}{dv}<\infty\,\,\,\text{a.s.}
\quad\Longleftrightarrow\quad \int_0^\infty \sigma( e^y)^{-\alpha}  e^{\alpha y}{dy} = \int_1^\infty \sigma(x)^{-\alpha} x^{\alpha-1}dx<\infty.
\]
The analogous argument in which we censor away the positive parts of $\hat{X}^\circ$ (the negative of this censored process is a pssMp)  shows that 
\[
\int_0^\infty \sigma({{\hat{X}^\circ}}_t)^{-\alpha}\mathbf{1}_{({{\hat{X}^\circ}}_t<0)}{dt}<\infty\,\,\,\text{a.s.} \quad\Longleftrightarrow \quad\int_{-\infty}^0\sigma(x)^{-\alpha}|x|^{\alpha-1}{dx}=\infty.
\]
We thus conclude that 
\[
\int_0^\infty \sigma({{\hat{X}^\circ}}_t)^{-\alpha}{dt} <\infty \,\,\,\text{a.s.}\quad \Longleftrightarrow \quad\int_{|x|>1}\sigma(x)^{-\alpha}|x|^{\alpha-1}{dx}<\infty.
\]
The integral test \eqref{integral} thus follows from \eqref{prove this}.

\subsection*{Idea for sufficiency} From Proposition \ref{prop} we know that the SDE started from $x$ with law $\emph{\rm P}_{x}$ can be built under spatial-inversion ($x\mapsto 1/x$) as a time-change of the h-transformed (conditioned) process $\hat{X}^\circ$ started in $1/x$. The natural guess is to construct $\emph{\rm P}_{\pm\infty}$ as spatial inversion of the same time-change of $\hat{X}^\circ$ started from $0$. Two facts need to be established: the limit law $\hat{\mathbb{P}}^\circ_0=\lim_{x\to 0}\hat{\mathbb{P}}^\circ_x$ needs to be well-defined and the time-change needs to be well-defined under $\hat{\mathbb{P}}^\circ_0$. The first follows from \cite{DDK} as explained in Section \ref{sec:rssMp}, the latter by computing the expectation of the time-change which leads to the integral test \eqref{integral}. Finally, we show that the semigroup extension defined like this is indeed a Feller extension of $(\emph{\rm P}_{x}:x\in\R)$ to $\overline{\underline \R}$. Since $\emph{\rm P}_{\pm\infty}$ is constructed explicitly through space-inversion and time-change from $\hat{\mathbb{P}}^\circ_0$, under which trajectories leave $0$ continuously, we see immediately that under $\emph{\rm P}_{\pm\infty}$ paths almost surely start from infinity continuously{\color{black}; see Figure \ref{fig3}}.

\begin{figure}[h]
                \includegraphics[scale=0.5]{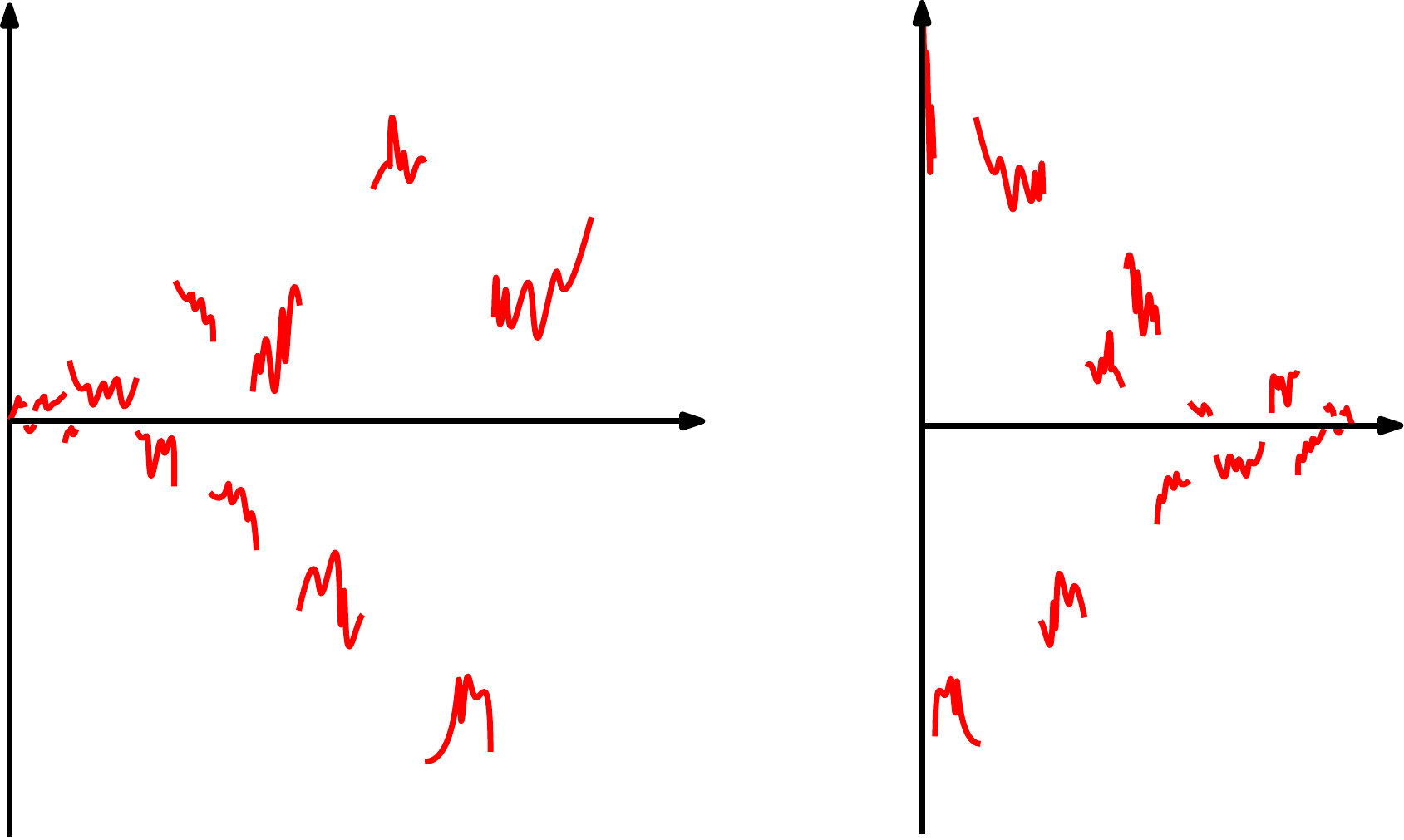}
  \caption{Space inversion and time-change of $h$-transform entrance law $\hat{\mathbb{P}}^\circ_0$ to give SDE started at $\pm\infty$.}
  \label{fig3}
\end{figure}
\subsection*{Proof of sufficiency} Suppose the integral test \eqref{integral} is satisfied. We first use \eqref{integral} to prove that 
\begin{align}\label{abc}
	\hat{\mathbb{E}}^\circ_0\left[ \int_0^\infty\beta(\hat{X}^\circ_u){du}\right] <\infty
\end{align}
with $\beta(x)=\sigma(1/x)^{-\alpha}|x|^{-2\alpha}$ for $x\neq 0$. 

\smallskip

Recalling that when $\hat{X}^\circ$ is negatively censored as in \eqref{Ycen}, as a positive self-similar Markov process, thanks to the type of underlying L\'evy process described in \eqref{hypexp}, the origin is left instantaneously and not hit again; see for example the discussion in \cite{CKPR}. A similar statement holds when $\hat{X}^\circ$ is positively censored. It follows that under $\hat{\mathbb{P}}^\circ_0$, the origin is left instantaneously and $0$ is not hit again, thus, the integral is well-defined 
but possibly infinite.\smallskip

 To prove \eqref{abc}, note that, for each fixed $t>0$, $\omega\mapsto \int_0^t \beta(\omega_s){\rm d}s$ is a continuous functional in the Skorohod topology. Using that $x\mapsto \mathbb P^\circ_x$ is weakly continuous, Fatou's Lemma and $\beta \geq 0$, we first get
\begin{align}\label{NAnnoying}
	\hat{\mathbb{E}}^\circ_0\left[ \int_0^t\beta(\hat{X}^\circ_u){du}\right]\leq\lim_{|x|\to0}\hat{\mathbb{E}}^\circ_x\left[ \int_0^t\beta(\hat{X}^\circ_u){du}\right]<\lim_{|x|\to0}\hat{\mathbb{E}}^\circ_x\left[ \int_0^\infty\beta(\hat{X}^\circ_u){du}\right]
\end{align}
for all $t\geq 0$. Hence, to prove \eqref{abc} we show that the righthand side of \eqref{NAnnoying} is finite. Recalling that $\hat{X}^\circ$ is an $h$-transform of $X^\dagger$, using $\hat{h}$ defined as in \eqref{h} albeit the roles of $\rho$ and $\hat\rho$ are interchanged,   and the general $h$-transform formula for potential measures `$G^h(x,dy)={h(y)}G(x,dy)/{h(x)} $' yields
\begin{align}\label{takelimitx}
\begin{split}
\hat{\mathbb{E}}^\circ_x\left[ \int_0^\infty\beta(\hat{X}^\circ_u){du}\right] & = \int_{\mathbb{R}}G_{\hat{X}^\circ}(x,{dy})\sigma(1/y)^{-\alpha}|y|^{-2\alpha}\\
&=\int_{\mathbb{R}}G_{\hat{X}^\dagger}(x,{dy})\frac{\hat{h}(y)}{\hat{h}(x)}\sigma(1/y)^{-\alpha}|y|^{-2\alpha}.
\end{split}
\end{align}
  In order to take the limit in \eqref{takelimitx} as $|x|\to0$, we can appeal to the expression for $G_{X^\dagger}(x,{dy})$. Recall from \eqref{daggerpot} that  $G_{X^\dagger}(x,{dy})$ has a density
\begin{align}
g_{X^\dagger}(x,y) &=-\frac{\Gamma(1-\alpha)}{\pi^2}\left(|y|^{\alpha-1}s(y)
 - |y-x|^{\alpha-1} s(y-x) +|x|^{\alpha-1}s(-x)\right),
 \label{daggerpot2}
\end{align}
where 
$
s(x) =  \sin(\pi\alpha\rho)\mathbf{1}_{(x>0)} +   \sin(\pi\alpha\rhohat)\mathbf{1}_{(x<0)}$.  It was also noted there that, following classical potential theory (see also Theorem 6.5 of \cite{G}), 
\begin{align}
\frac{|y|^{\alpha-1}s(y)- |y-x|^{\alpha-1} s(y-x) +|x|^{\alpha-1}s(-x)}{|y|^{\alpha-1}(s(y)+ s(-y))}
=\frac{g_{X^\dagger}(x,y)}{g_{X^\dagger}(y,y)}
=\mathbb{P}_x(\tau^{\{y\}}<\tau^{\{0\}})\leq 1,
\label{proba}
\end{align}
for $\tau^{\{y\}} = \inf\{t>0: X_t = y\}$. Using the assumption that
\begin{align*}
	\int_{\R}\sigma(1/y)^{-\alpha}|y|^{-\alpha -1}{dy}= \int_{\R}\sigma(z)^{-\alpha}|z|^{\alpha-1}{dz}<\infty
\end{align*}	
and $\alpha\in(1,2)$ together with \eqref{daggerpot2} and \eqref{proba}, we compute, with a floating unimportant constant $C$,
which can take different values in each line, 
\begin{align*}
& \lim_{|x|\to 0}-\frac{\pi^2}{\Gamma(1-\alpha)}\hat{\mathbb{E}}^\circ_x\left[ \int_0^\infty\beta(\hat{X}^\circ_u){du}\right] \\
&=	\lim_{|x|\to0}-\frac{\pi^2}{\Gamma(1-\alpha)}\int_{\mathbb{R}}G_{X^\dagger}(x,{dy})\frac{h(y)}{h(x)}\sigma(1/y)^{-\alpha}|y|^{-2\alpha}\notag\\
&=\lim_{|x|\to0}\int_{\R} \frac{s(-y)\left(|y|^{\alpha-1}s(y) - |y-x|^{\alpha-1} s(y-x) +|x|^{\alpha-1}s(-x)\right)}{s(-x)|x|^{\alpha -1}}\frac{\sigma(1/y)^{-\alpha}}{|y|^{\alpha + 1}}{dy}\notag\\
 &\leq C  \int_{\R}\lim_{|x|\to0} \frac{\left(|y|^{\alpha-1}s(y)
 - |y-x|^{\alpha-1} s(y-x) +|x|^{\alpha-1}s(-x)\right)}{|x|^{\alpha -1}}\frac{\sigma(1/y)^{-\alpha}}{|y|^{\alpha+1}}{dy}\notag\\
 &\leq  C\int_{\R}\lim_{|x|\to0}|x|^{2-\alpha}{\rm sign}(x)\frac{
\left(|y|^{\alpha-1} - |y-x|^{\alpha-1}  \right)}{x}\frac{1}{|y|^{\alpha+1}}\sigma(1/y)^{-\alpha}{dy}\notag\\
&\quad+C\int_{\R}\frac{1}{|y|^{\alpha+1}}\sigma(1/y)^{-\alpha}{dy}\notag\\
&=C\int_{\R}{|z|^{\alpha-1}}\sigma(z)^{-\alpha}{dz}\notag\\
&<\infty,
\end{align*}
where in the first inequality we have used dominated convergence in combination with \eqref{proba} and the righthand side was assumed to be finite. Hence, \eqref{abc} is verified.\medskip

Now we come to the crucial step. We write down explicitly the process that plays the role of the SDE \eqref{2} started from infinity. First, note that \eqref{abc} implies that $\hat{\mathbb{P}}^\circ_0$-almost surely $\int_0^\infty\beta(\hat{X}^\circ_u){du}<\infty$. In turn, this implies that 
the time-change $(\theta_t, t\geq 0)$, in Propostion \ref{prop} (i) explodes in finite time.
Moreover, on account of the fact that $(\hat{X}^\circ,\hat{\mathbb{P}}_0^\circ)$ is well-defined, cf. Lemma \ref{zeroenter},  the space-time transformation
\begin{align}
	{Z}^\dagger_t=\frac{1}{\hat{X}^\circ_{\theta_t}}, \qquad t< \int_0^\infty \beta(\hat{X}^\circ_u){\,d u},
	\label{RBZ4}
\end{align}
where 
\[
\theta_t =\inf\left\{s> 0 : \int_0^s\beta(\hat{X}^\circ_u){\,d u}>t\right\}
\]
is well-defined under $\hat{\mathbb{P}}^\circ_0$.

\smallskip

  Given the conclusion of Proposition \ref{prop} (i), it thus follows that we have constructed 
a candidate for the Feller extension of $({\rm P}_z, z\in \R)$ with ${\rm P}_{\pm\infty}$ defined as \eqref{RBZ4} under $\hat{\mathbb P}^\circ_0$. Note that trajectories enter instantaneously with alternations between $+\infty$ and $-\infty$ as trajectories under  $\hat{\mathbb P}^\circ_0$ leave $0$ instantaneously with alternations of sign. We still need to verify, for the extension at $\pm\infty$, the weak continuity of $({\rm P}_z, z\in \underline{\overline{\R}})$ in the Skorokhod topology and the Feller property. Note, the latter  means that,
for continuous and bounded $f$  on $\underline{\overline{\R}}$, we need
\begin{align}
\lim_{|x|\to \infty}{\rm E}_x[f(Z^\dagger_t)] ={\rm E}_{\pm\infty}[f(Z^\dagger_t)]\qquad \text{ and }\qquad \lim_{t\to 0} {\rm E}_{\pm\infty}[f(Z^\dagger_t)] = f(\pm\infty).
\label{Felleratpminf}
\end{align}

Propostion \ref{prop} and the definition of ${\rm P}_{\pm\infty}$ allows us to equivalently write \eqref{Felleratpminf} as

\begin{equation}
 \lim_{|x|\to \infty}\hat{\mathbb{E}}^\circ_x[f(1/\hat{X}^\circ_{\theta_t})]= \hat{\mathbb{E}}^\circ_0[f(1/\hat{X}^\circ_{\theta_t})]\qquad \text{ and }\qquad \lim_{t\to 0} \hat{\mathbb{E}}^\circ_0[f(1/\hat{X}^\circ_{\theta_t})] = f(\pm\infty).
 \label{Felleratpminf2}
\end{equation}

Thanks to  \eqref{abc}  and the continuity  of composition, first hitting and $\theta$ with respect to the Skorokhod topology for sufficiently regular processes, cf. Chapter 13 of Whitt \cite{Wh}, the weak continuity and the Feller property follow from the Skorokhod continuity of $X^\circ$ from Lemma \ref{zeroenter}.

\section{Entrance from $+\infty$, spectrally positive, $\alpha\in(1,2)$}\label{proof5}
The entire proof is along the lines of the previous section, albeit that we work with the duality relation explored in Proposition \ref{nag2}, replacing $\hat{\mathbb{P}}^\circ_0$ by $\hat{\mathbb{P}}^\uparrow_0$, in order to show the sufficiency and necessity of the condition 
\begin{align}\label{SPLPintegral}
	I^{\sigma,\alpha}(\R_+)=\int_0^\infty\sigma(x)^{-\alpha}x^{\alpha-1}{dx}<\infty.
\end{align}
The proof for entrance from $-\infty$ in the spectrally negative regime is analogous.

\subsection*{Proof of  Necessity} Suppose that  $+\infty$ is an entrance point. Then the duality of $\hat{Z}^\uparrow$ and $Z^\dagger$ in Proposition \ref{nag2} means that, by time reversing $Z$ from its first hitting of the origin, $+\infty$ must be accessible for $\hat{Z}^\uparrow$. Reasoning in a similar way to the `necessity' part of the proof in Section \ref{proof3}, we must  have that 
\begin{align}
	\int_0^\infty \sigma({{\hat{X}^\uparrow}}_s)^{-\alpha}ds<\infty
	\label{prove this2}
\end{align}
almost surely under $\hat{\p}^\uparrow_{x}$, for $x\geq 0$. Recalling that ${\hat{X}^\uparrow}$ is a positive self-similar Markov process we use Lamperti's representation to rewrite \eqref{prove this2} as a perpetual integral of the L\'evy process $\hat{\xi}^\uparrow$ discussed at the end of Section \ref{pssmpexamples}. The L\'evy process hits point (because $X$ and hence $X^\uparrow$ do), thus, has local times. The L\'evy process also has finite positive mean as can be seen similarly to the proof in Section \ref{proof3} using the characteristic exponent \eqref{xiuparrowpsi}. Hence, Theorem 1 of \cite{DK} is applicable to deduce via change of variables that \eqref{prove this2} holds if and only if \eqref{SPLPintegral} holds. 
\medskip

\subsection*{Proof of Sufficiency} We are again guided by the sufficiency argument in Section \ref{proof3}. We appeal to the representation in Proposition \ref{BZuparrow} to provide a candidate for ${\rm P}_{+\infty}$ built from $1/\hat{X}^\uparrow_{\theta_t}$, $t\geq 0$, under $\hat{\mathbb{P}}^\uparrow_0$ from Lemma \ref{zeroenter2}. For this to work we need to ensure that $\int_0^\infty\beta(\hat{X}^\uparrow_u){du}<\infty$, $\hat{\mathbb{P}}^\uparrow_0$-almost surely. As in Section \ref{proof3} this will be achieved by proving
\begin{align*}
\hat{\mathbb{E}}^\uparrow_0\left[ \int_0^\infty\beta(\hat{X}^\uparrow_u){du}\right] <\infty.
\end{align*}
To this end, let us write 
$G_{\hat{X}^{\ddagger}}(x, dy)$, $x,y>0$, for the potential measure of $\hat X$ killed on entering $(-\infty,0)$. Appealing to Corollary 8.8 and Exercise 8.2 of \cite{Kbook},  it is shown that, up to a multiplicative constant, 
\begin{equation}
G_{\hat{X}^\ddagger}(x,{dy}) =\left( x^{\alpha - 1} - (x-y)^{\alpha-1}\mathbf{1}_{(x\geq y)}\right)d y, \qquad  x,y\geq 0.
\label{Gdd}
\end{equation}
 Thus, we have that, up to a multiplicative constant on the left-hand side,
\begin{align*}
\hat{\mathbb{E}}^\uparrow_x\left[ \int_0^\infty\beta(\hat{X}^\uparrow_u){du}\right] 
&=\int_{0}^\infty G_{\hat{X}^\ddagger}(x,{dy})\frac{h(y)}{h(x)}\sigma(1/y)^{-\alpha}y^{-2\alpha}\\
&=\int_{0}^\infty \sigma(1/y)^{-\alpha}y^{-\alpha-1}dy
-\int_0^x \frac{(x-y)^{\alpha-1}}{x^{\alpha-1}} \sigma(1/y)^{-\alpha}y^{-\alpha-1}dy
\end{align*}
so that thanks to Fubini's Theorem and  Fatou's Lemma
\begin{align*}
\hat{\mathbb{E}}^\uparrow_0\left[ \int_0^\infty\beta(\hat{X}^\uparrow_u){du}\right]&\leq  \int_0^\infty\lim_{x\downarrow0} \hat{\mathbb{E}}^\uparrow_x\left[\beta(\hat{X}^\uparrow_u)\right] {du}\\
&
 = \int_{0}^\infty \sigma(1/y)^{-\alpha}y^{-\alpha-1}dy \\
 &= \int_{0}^\infty \sigma(z)^{-\alpha}z^{\alpha-1}dz\\
 &<\infty
\end{align*}
as required.

\smallskip
Now we come to the construction. We write down explicitly the process that plays the role of the SDE \eqref{2} started from infinity. Since we proved that $\int_0^\infty\beta(\hat{X}^\uparrow_u){du}$ is almost surely finite under $\hat{\mathbb{P}}^\uparrow_0$, the time-change $\theta_t$, $t\geq 0$, in Propostion \ref{BZuparrow} explodes in finite time. Moreover, on account of the fact that $(\hat{X}^\uparrow,\hat{\mathbb{P}}_0)$ is well-defined, cf. Lemma \ref{zeroenter2}, the space-time transformation
\begin{align}\label{RBZ5}
	{Z}^\dagger_t=\frac{1}{\hat{X}^\uparrow_{\theta_t}}, \qquad t< \int_0^\infty \beta(\hat{X}^\uparrow_u){\,d u},
\end{align}
where 
\[
\theta_t =\inf\left\{s> 0 : \int_0^s\beta(\hat{X}^\uparrow_u){\,d u}>t\right\}
\]
is well-defined under $\hat{\mathbb{P}}^\uparrow_0$. Volkonskii's Theorem, Corollary to Theorem 2 of \cite{Vol}, ensures that 
the right-hand side of \eqref{RBZ5} is a strong Markov process.  Given the conclusion of Proposition \ref{BZuparrow} (i), it thus follows that we have constructed 
a candidate for the Feller extension of $({\rm P}_z, z\in \R)$
 with ${\rm P}_{+\infty}$ defined by \eqref{RBZ5} under  $\hat{\mathbb P}^\uparrow_0$. Note that trajectories come down from $+\infty$ continuously as trajectories under $\hat{\mathbb P}^\uparrow_0$ leave zero continuously and are non-negative.\smallskip

Checking for the Feller property of $Z^\dagger$ when entering at $+\infty$  we again follow the reasoning in Section \ref{proof3} and appeal to the representation in Proposition \ref{BZuparrow} to conclude that it suffices to check that for continuous and bounded $f$ on $[0,\infty]$
\[
\lim_{|x|\to \infty}{\rm E}_x[f(Z^\dagger_t)] = \hat{\mathbb{E}}^\uparrow_0[f(1/\hat{X}^\uparrow_{\theta_t})]\qquad \text{ and }\qquad \lim_{t\to 0}{\rm E}_x[f(Z^\dagger_t)]=\lim_{t\to 0} \hat{\mathbb{E}}^\uparrow_0[f(1/\hat{X}^\uparrow_{\theta_t})] = f(+\infty).
\]
As in Section \ref{proof3}, this follows as a consequence of the Feller property of $\hat{X}^\uparrow$ at $0$, Lemma \ref{zeroenter2}. 
The Skorokhod continuity of $({\rm P}_x , x\in\overline{\R})$ also follows in an easy and  similar manner to the proof at the very end of Section \ref{proof3}.

\section{Entrance from $\pm\infty$, $\alpha=1$}\label{a=1}
Now we come to the more delicate case of $\alpha=1$. The sufficiency proof is similar to the ones before, the  proof of necessity  must be different. There are two reasons why additional arguments are needed. Since the Cauchy process does not hit points (has no local times) the time-reversal from points does not work unchanged and the 0-1 law for perpetual integrals of \cite{DK} is not applicable. To circumvent these difficulties we develop a different approach here, built upon a general theory of transience for Markov processes highlighted by Getoor \cite{G}. The general result for transient Markov processes we will use is developed in the Appendix to avoid distraction from the job at hand in this section.


%
%
%
%
%
%
%
%
%
%

\subsection*{Proof of necessity} We start with an auxiliary lemma, which will be used as part of the proof of necessity thereafter.
We need to compute the potential measure of the extension killed upon first entry to $(-1,1)$. This is a consequence of recent work on killed stable processes given in the lemma below, which is stated under the additional assumptions of the necessary part of the proof of entrance from $\pm\infty$ with $\alpha = 1$. 
\begin{lemma}\label{resolvent}
	Suppose that  $Z^\star$ the unique solution to the SDE \eqref{2} (resp. the extension to infinity) killed upon first entry into $(-1,1)$. Then the potential measure is
	\begin{align*}
		G_{Z^\star}(x,dy)=\begin{cases}
			\sigma(y)^{-1}\frac{1}{\pi}\big(\log(|y|+(y^2-1)^{1/2})\big)\,dy& \text{if }x=\pm\infty\\
			\sigma(y)^{-1}\frac{1}{\pi}\big(\log(|\frac{1-xy}{x-y}|+((\frac{1-xy}{x-y})^2-1)^{1/2}\big)\,dy& \text{if }x\in\R\backslash (-1,1),
		\end{cases}
	\end{align*}
	for $|y|\geq 1$.
\end{lemma}	
\begin{proof}
	The formula for $x\in \R\backslash (-1,1)$ follows from Theorem B of Profeta and Simon \cite{PS} or Theorem II.3.3. of Kyprianou \cite{KALEA} (the potential density of the killed Cauchy process), the factor $\sigma^{-1}$ comes from the time-change and substitution in the time-integral defining the potential measure.\smallskip
	
	For $x=\pm\infty$, we can use the assumed Skorokhod continuity as in \eqref{hg} and reason as in \eqref{check1} with $x\notin A$ to deduce that 
	$G_{Z^\star}(\pm\infty,A) =\lim_{|x|\to\infty}G_{Z^\star}(x,A)$, for all bounded Borel sets in $\overline{\underline{\mathbb{R}}}\backslash(-1,1)$. In turn this gives the statement of the lemma.
\end{proof}
To finish the proof of necessity of Theorem \ref{main} in the case $\alpha = 1$, recall that $Z^\star$ is the unique solution to the SDE \eqref{2} killed upon first entry into $(-1,1)$. We first check the assumptions of Proposition \ref{P1} for $Z^\star$. If $x\in \R\backslash (-1,1)$, then $Z^\star$  hits $(-1,1)$ in finite time because of the time-change representation from Proposition \ref{pr}, the (set)recurrence of the Cauchy process and the assumption that $\sigma>0$ is continuous (thus, locally bounded away from zero by a constant). Hence, ${\rm P}_x(\zeta^\star<\infty)=1$ for $x\in \R\backslash (-1,1)$, where $\zeta^\star$ is the lifetime of $Z^\star$. For $x=\pm\infty$ we apply  \eqref{hg}, which is equally valid for $\alpha = 1$, to deduce that $\rm P_{\pm\infty}(\zeta^\star<\infty)=1$, by set recurrence. Since, by definition, $\texttt{E}_x[f(\zeta^\star)]=\texttt{E}_x[f(T^{(-1,1)})]$, where $T^{(-1,1)}= \inf\{t>0: Z_t\in(-1,1)\}$, the continuity of $x\mapsto \texttt{E}_x[f(\zeta^\star)]$ for $f$ bounded continuous follows from the assumed weak continuity in the Skorokhod topology of the extension of $Z$ and Chapter 13 of \cite{Whitt}. Applying Proposition \ref{P1} in the appendix, we obtain $G_{Z^\star}(\pm\infty,K)<\infty$ for all $K$ compact. Choosing $K=\overline{\underline{
\R}}\backslash (-1,1)$, we have from Lemma \ref{resolvent}
\begin{align*}
	\int_{(-1,1)^c}\sigma(y)^{-1}\frac{1}{\pi}\big(\log(|y|+(y^2-1)^{1/2})\big)\,dy<\infty
\end{align*}
from which the integral test 
\begin{align*}
	 I^{\sigma,1}= \int_\R \sigma(y)^{-1}\log |y|d y<\infty
\end{align*}
follows because $\sigma$ is bounded away from $0$ on compacts.

\subsection*{Proof of sufficiency} 
We want to prove that the condition 
\begin{equation}
	I^{\sigma,1}=\int_\R \sigma(x)^{-1}\log |x| dx<\infty
\label{NS}
\end{equation}
implies that $\pm\infty$ is an entrance point. The construction is identical to the one in Section \ref{proof3} but simpler as the $h$-function for $\alpha=1$ becomes $h=1$ so that $\hat{X}^\circ = X$. Specifically we relate via Proposition \ref{prop} the entrance of $Z$ at $\pm\infty$ to the entrance of the Cauchy process at $0$. Note that the Cauchy process leaves zero continuously and never returns. In analogy to the final paragraphs of Section \ref{proof3} the guess for the limiting law will be
\begin{align}\label{ddd}
	{Z}_t=\frac{1}{X_{\theta_t}}, \quad t\geq 0,
\end{align}
under $\P_0$, where 
\[
\theta_t =\inf\left\{s> 0 : \int_0^s\beta(X_u){\,d u}>t\right\}.
\]
To show that $\theta$ is well-defined for all $t\geq 0$ we proved in Section \ref{proof3} that $\int_0^t\beta(\hat{X}^\circ_u){\,d u}<\infty$ almost surely by checking $\hat{\mathbb{E}}^\circ_0\left[ \int_0^\infty\beta(\hat{X}^\circ_u){du}\right] <\infty$ in \eqref{abc}. Controlling the integral up to $t$ by the integral up to $\infty$ is too coarse here as the latter is infinite due to the (set)recurrence of the Cauchy process. What we do instead is to show that $\int_0^{\tau^{(-a,a)^c}}\beta(X_u){\,d u}<\infty$ almost surely for all $a>0$. Since $\lim_{a\to\infty}\tau^{(-a,a)^c}=\infty$ almost surely, as a consequence we obtain $\int_0^t\beta(X_u){\,d u}<\infty$ almost surely.\smallskip

As in Section \ref{proof3}, to verify $\int_0^{\tau^{(-a,a)^c}}\beta(X_u){\,d u}<\infty$ almost surely, we prove finiteness of the expectation under $\mathbb{P}_0$. To this end, considering only $a=1$ for notational convenience and write  $(X^\bullet_t  , t<\tau^{(-1,1)^c})$ for the process $X$ killed on exiting $(-1,1)$. Recalling that $G_{X^\bullet}$  denotes  its potential measure, we compute
\begin{align*}
\mathbb{E}_0\left[ 
	\int_0^{\tau^{(-1,1)^c}} \beta(X_u)\, {d u}\right]&=\int_{-1}^1\beta(y)G_{X^\bullet}(0,dy)\\
	&=\frac{1}{\Gamma(\alpha/2)^2}\int_{-1}^1\beta(y)\int_1^{1/|y|} (s^2-1)^{-1/2} ds \,dy\\
	&\leq -\frac{1}{\Gamma(\alpha/2)^2}\int_{-1}^1\sigma(1/y)^{-1}|y|^{-2}\log|y| dy\\
	&=\int_{|z|\geq 1}\sigma(z)^{-1}\log|z| dz\\
	&\leq I^{\sigma,1}<\infty,
\end{align*}
where we have taken advantage of the explicit form of $G_{X^\bullet}$; see for example Blumenthal et al. \cite{BGR}.\smallskip

The rest of the sufficiency proof goes along the arguments of Section \ref{proof3} with the guessed limit \eqref{ddd} under $\P_0$. Using the above to see that the time-change in \eqref{ddd} is well-defined the argument is as in Section \ref{proof3}.

\section{Explosion}\label{sec6}
We only give the arguments for two-sided jumps, the one sided cases are modifications just as Section \ref{proof5} is a modification of Section \ref{proof3}, e.g. by replacing $X^\circ$ by $X^\uparrow$.
	

\subsection*{Non-explosion for $\alpha\geq 1$} Recall from Proposition \ref{pr} that for initial condition $x\in\R$, under the stable law $\P_x$, the time-change $Z_t:=X_{\tau_t}$ is the unique solution to the SDE \eqref{2} up to the killing time $T= \int_0^\infty \sigma(X_s)^{-\alpha}ds$ which is a perpetual integral. To show that solutions do not explode we only need to verify that $\P_x(T=\infty)=1$. But this is a direct consequence of the (set)recurrence of stable processes for $\alpha\geq 1$.

\subsection*{Explosion and non-explosion for $\alpha\in(0,1)$} Just as in the argument for $\alpha\geq 1$, a 0-1 law $\P_x(T<\infty)\in \{0,1\}$ for the perpetual integral $T=\int_0^\infty \sigma(X_s)^{-\alpha}ds$ depending on $\alpha$ and $\sigma$ would be sufficient to style the remainder of the  proof. The 0-1 law for perpetual integrals is not hard to prove (see Lemma 5 of \cite{DK}) but we cannot provide a direct characterization of $\alpha$ and $\sigma$ that leads to respective probabilities of  $0$ or $1$. Instead, we appeal again to our understanding of how expectation of the perpetual integral serves as an equivalent marker of almost sure convergence. In the `sufficient' direction, this is straightforward in the `necessary' direction, we will again use our variant of  Getoor's characterisation of transience, given in Proposition \ref{P1} of the Appendix. 
	
\subsection*{Necessity} The main idea here will be to use a mixture of space inversion together with time reversal to convert the 
event of explosion into an event of entrance for a familiar transient process that lives on $\R$ (Proposition \ref{nagAx} below). As such, the latter will allow us to invoke Proposition \ref{P1}, whose conclusion can be reinterpreted as ensuring the desired integral test holds.

\smallskip

Recall from Section \ref{background} that, when $\alpha\in(0,1)$, the stable process does not hit points (hence, $X=X^\dagger$) and its Doob $h$-transform using $h$ from \eqref{h} corresponds to conditioning the process to be continuously absorbed (in finite time) at the origin. For the next proposition recall that $\beta (x) = \sigma(1/x)^{-\alpha}|x|^{-2\alpha}$ for $x\neq 0$.

 \begin{prop}\label{nagAx} 
 Suppose that $\alpha\in(0,1)$, the stable process $X$ has two-sided jumps and the solution $Z$ to \eqref{2} explodes for all points of issue.
Under $\mathbb{P}_x, x\in\R$, define $V_t = X_{\iota_t}$ for $t<\int_0^\infty \beta(X_s)ds$, where
\begin{align}\label{iota}
	\iota_t = \inf\left\{s>0 : \int_0^s \beta(X_s)ds > t\right\}
\end{align}
and let $\hat{V}^\circ_ t ={Z^{-1}_t}$ for $t<T:=\int_0^\infty \sigma(X_s)^{-\alpha}ds$. Then 
\begin{equation}
\label{third} V\text{ is in weak duality with  }\hat{V}^\circ \text{ with respect to }\mu(dx)=\beta(x)h(x)dx,
\end{equation}
where $h$ is given by \eqref{h}.
Moreover, when  $Z$ is issued from the origin,
the time-reversal $(\hat{V}^\circ_{(T-t)-}, t\leq T)$  is a time-homogenous Markov process with transition probabilities which agree with that of $V$ started in $0$.

\end{prop}
\begin{proof}
The proof is similar in spirit to that of Proposition \ref{nag} so we only highlight the main points.

\smallskip

Proposition \ref{prop} tells us that $\hat{V}^\circ_ t=Z^{-1}_t  = {{\hat{X}^\circ}}_{\theta_t}$, $t<  \int_0^\infty \beta({{\hat{X}^\circ}}_s)ds$, where the time-change $(\theta_t, t\geq 0)$ is given by 
\begin{align}\label{tauhatAx}
	\theta_t = \inf\left\{s>0 : \int_0^s \beta({{\hat{X}^\circ}}_s)ds > t\right\}.
\end{align}
Since $Z$ is assumed to explode at the finite $T$, $\hat{X}^\circ_{\theta_\cdot}$ is absorbed at $T$.

\smallskip

The proof of the weak duality \eqref{third} follows by the use of Revuz measures, as in the proof of  \eqref{claim}, as soon as we can show that $\hat{X}^\circ$ and $X$ are in weak duality  with respect to $ h(x)dx$. This was already shown, however, in \eqref{verifiessufficetocheck}.

\smallskip

For the final part, we note that $\hat{V}^\circ= {{\hat{X}^\circ}}_{\theta_\cdot}$ is a Markov process that hits the origin at the explosion time $T$ of $Z$. As before, we want to apply Nagasawa's Duality Theorem \ref{Ndual}. As usual, the verification of {\bf (A.3.3)} is 
straightforward (appealing to dominated convergence). Taking account of \eqref{third}, to verify {\bf (A.3.1)}, we are required to check that, for all bounded and measurable $f$ which is compactly supported in the domain $\overline{\underline\R}\backslash\{0\}$ of $\hat{V}^\circ$,
\begin{align}
{\rm E}_0\left[\int_0^T f(1/Z_t)dt\right]=\int_\R f(x)\beta(x)h(x)dx.
\label{showthisformu}
\end{align}
Writing $G_X$ for the potential measure of $X$, we have $G_X(0,dx) =  h(x)dx$; see e.g.  Theorem I.1.4 in  Kyprianou \cite{KALEA}. We may thus write 
\begin{align}
 {\rm E}_0\left[\int_0^T f(1/Z_t)dt\right]&=
\mathbb{E}_0\left[
\int_0^\infty f(1/X_s) \sigma(X_t)^{-\alpha}dt\right] \notag\\
&= \int_\R  f(1/x) \sigma(x)^{-\alpha}h(x) dx\notag\\
& = \int_\R  f(y) \sigma(1/y)^{-\alpha} h(1/y) y^{-2}dy\notag\\
&=\int_\R  f(y) \sigma(1/y)^{-\alpha} h(y)|y|^{-2\alpha}dy\notag\\
&=\int_\R  f(y) \beta(y)h(y)dy,
\label{compactfimportant}
\end{align}
where, just as in \eqref{check2}, we have used that $h(1/y)|y|^{-2} = h(y)|y|^{-2\alpha}$. Note in particular that the compact support in $\overline{\underline\R}\backslash\{0\}$ of $f$ ensures that the right-hand side of \eqref{compactfimportant} is finite. The requirement \eqref{showthisformu} thus holds and hence the proof is complete. 
\end{proof}

Let us now return to the proof of necessity for the case $\alpha\in(0,1)$ in Theorem \ref{zthrm} for which we aim to use Proposition \ref{P1}. Recall the notion $(X^\bullet_t  , t<\tau^{(-1,1)^c})$ for the stable process $X$ killed on first exiting $(-1,1)$. Accordingly
$
X^\bullet_{\iota_\cdot}
$ denotes the process $V = X_{\iota_\cdot}$ killed on first exiting $(-1,1)$. Let us denote the killing time by $\zeta^\bullet$ and note that $\zeta^\bullet= 
\int_0^{\tau^{(-1,1)^c}}\beta(X_s)ds$. When $X^\bullet_{\iota_\cdot}$ is issued from a point $x\neq 0$, the aforementioned integral representation of $\zeta^\bullet$ and  the fact that $|X^\bullet_{\iota_\cdot}|$ is almost surely bounded away from the origin and 1 implies that  $\zeta^\bullet$  is almost surely finite. For $x=0$ the almost sure finiteness of $\zeta^\bullet$ is a consequence of the assumed explosion of $Z$ and the time-reversal statement in Proposition \ref{nagAx}. In total, the assumed explosion implies $\P_x(\zeta^\bullet\in (0,\infty))=1$ for all $x\in (-1,1)$, which is property (a) of Proposition \ref{P1}. 

\smallskip

{\color{black}
Property (b) of Proposition \ref{P1}, requires the weak continuity of $\zeta^\bullet$. 
Weak continuity is clear when the point of issue is away from the origin, as the trajectory of $X$ is bounded away from the origin; recall that the integrand of $ 
\int_0^{\tau^{(-1,1)^c}}\beta(X_s)ds$ (which equals $\zeta^\bullet$)  is explosive if $|X_s|\to0$. Weak continuity of $\zeta^\bullet$ at zero is a more complicated issue but, fundamentally, is a consequence of the assumed Skorokhod continuity of the explosion time $T$. 
\smallskip

To see why, we use duality, $h$-transforms and dominated convergence. First, note that the converse to the duality and spatial inversion in Proposition \ref{nagAx} (analogously to Propositions \ref{nag} and \ref{nag2}) is that, if we take the process $V= X_{\iota}$ issued from $x\in (-1,1), x\neq 0,$ and time reverse it from its last passage out of $(-1,1)$, say $\ell^{(-1,1)}$, the resulting process is equal in law to the process $\hat{V}^{\circ,(x)}$, defined as $1/Z^{\circ,(1/x)}$,  where $Z^{\circ,(1/x)}$ is the Doob $h$-transform of $X$ with the $h$-function $y\mapsto h(y-1/x)$ on $\mathbb{R}$, where $h$ is given by  \eqref{h} (i.e. $X$ conditioned to hit $1/x$ continuously), and time changed in the same way as \eqref{timechangesolution}. The initial condition of $Z^{\circ,(1/x)}$ is $\varpi_x(dy): =\mathbb{P}_x(1/X_{\ell^{(-1,1)}-}\in dy)$, $y\in(-1,1)^c$. Reasoning similarly to that of Step 1 of the proof of Proposition \ref{nag}  shows that  $X$ and $\hat{X}^\circ$ are in weak duality  and  we can also identify 
$
\mathbb{E}_x[f(X_{\ell^{(-1,1)}-}) ]=\lim_{|y|\to\infty}
\hat{\mathbb{E}}_y[f({X}_{\tau^{(-1,1)}})\hat{h}(X_{\tau^{(-1,1)}}-x)/\hat{h}(y-x) ]
$; see similar calculations in \cite{KV}.
It follows from the explicit formula \eqref{striplowalpha} that $\varpi_x$ is absolutely continuous with respect to the Lebesgue measure, for each $x\in (-1,1)$, as well as that $(\varpi_x, x\in(-1,1))$ forming a weakly continuous family of measures. We will use these preparatory remarks to prove $$\lim_{|x|\to0}\mathbb{P}_x(t<\zeta^\bullet)=\mathbb{P}_{0}(t<\zeta^\bullet), \quad t\geq 0.$$ Define 
 \[
 H_x(y, t):=  \mathbb{E}_{y}\left[\mathbf{1}_{(t<T)} {|xX_{\tau_t}-1|^{\alpha -1} }{|xy-1|^{1-\alpha}}\right],\quad y\in (-1,1)^c, x\in(-1,1), t>0,
 \]  
 so that, due to the duality and spatial inversion mentioned above, $$\mathbb{P}_x(t<\zeta^\bullet) = \int_{(-1,1)^c} H_x(y,t)\,\varpi_x(dy).$$ In order to deal with the limit of $\mathbb{P}_x(t<\zeta^\bullet)$ for $|x|\to 0$, we first prove that
\begin{equation}
	\lim_{|x|\to 0}\int_{(-1,1)^c} H_x(y,t)\,\varpi_x(dy)= \lim_{|x|\to0} \int_{(-1,1)^c} H_0(y,t)\, \varpi_x(dy),
\label{varpi1}
\end{equation}
and then use weak continuity of the measures $(\varpi_x, x\in(-1,1))$ and continuity of $H_0$ to complete the argument. Note that the Doob $h$-transform in the definition of $H_x(y,t)$ is applied at the almost surely finite stopping times $(\tau_t, t\geq 0)$ which remains a martingale transform e.g. by Theorem III.3.4 of \cite{JacodShiryaev}.\smallskip

Let us start to prove \eqref{varpi1}. As an $h$-transform, $H_x(y, t)$ is a probability and hence bounded in $[0,1]$.
To verify \eqref{varpi1} we show $\lim_{|x|\to0}\sup_{ |y|\in[1, N]}|H_x(y, t)- H_0(y, t)| = 0$ for any $N>1$ which then allows us to replace $H_x$ by $H_0$ in \eqref{varpi1}. To this end, using the spatial homogeneity of $(X, \mathbb{P})$, we can choose $\delta>0$ sufficiently small such that, for given $\varepsilon>0$,
\begin{align}
&\quad \sup_{ |y|\in[1, N]}|H_x(y,t)- H_0(y,t)| \notag\\
&= \sup_{ |y|\in[1, N]}\left|\mathbb{E}_{y}\left[\mathbf{1}_{(t<T)} \frac{|xX_{\tau_t}-1|^{\alpha -1} }{|xy-1|^{\alpha-1}}\right]-  \mathbb{P}_y(t<T)\right|\notag\\
&\leq  \mathbb{E}_{0}\left[\sup_{ |y|\in[1, N]}\mathbf{1}_{(t<T^{(y)}, \, \inf_{s\geq 0}|y+X_{s}|>\delta)}\left| \frac{|x + (xX_{\tau^y_t}/y)-(1/y)|^{\alpha -1} }{|x-(1/y)|^{\alpha-1}}-1\right|\right]\notag\\
&\quad +\mathbb{E}_{0}\left[\sup_{ |y|\in[1, N]}\mathbf{1}_{(t<T^{(y)}, \, \inf_{s\geq 0}|y+X_{s}|\leq \delta)}\left| \frac{|x + (xX_{\tau^y_t}/y)-(1/y)|^{\alpha -1} }{|x-(1/y)|^{\alpha-1}}-1\right|\right],
\label{delta}
\end{align}
where $T^{(y)}= \int_0^\infty \sigma(y+X_u)^{-\alpha}du$ and  $\tau^y_t = \inf\{s>0: \int_0^s \sigma(y+X_u)^{-\alpha}du> t\}$.  Note that the continuity of $\sigma$ and the restriction of $y\in[1,N]$ ensures that $\underline{c}_Nt \leq \tau^y_t\leq \overline{c}_Nt $ for constants $\underline{c}_N, \overline{c}_N$, depending on $N$. Next, we note that, for each fixed $u>0$, Doob's martingale inequality and the fact that $X$ is known to have  absolute moments of all orders in $(-1,\alpha),$ ensures that, for $p>1$ sufficiently close to 1, fixed $u>0$ and $z\in\mathbb{R}$, $ \mathbb{E}_0[\sup_{s\leq u}|z+X_s|^{p(\alpha-1)}]\leq c_p \mathbb{E}_0[|z+X_u|^{p(\alpha-1)}]<\infty$, for some unimportant constant $c_p\in(0,\infty)$.  As a consequence, when $x\in[-1/(2N),1/(2N)]$ and $|y|\in[1,N]$, there are constants $b_1^N$ and $b_2^N$ such that 
\begin{align*}
\mathbf{1}_{(t<T^{(y)})}\left| \frac{|x + (xX_{\tau^y_t}/y)-(1/y)|^{\alpha -1} }{|x-(1/y)|^{\alpha-1}}-1\right|
&\leq  b_1^N
\textstyle{
\sup_{s\leq b_2^N t}|X_s|^{\alpha -1} 
}
+ 1.
\end{align*}
For the first summand on the right-hand side of \eqref{delta}, we may now appeal to dominated convergence and take limits as $|x|\to0$ inside the expectation, noting that  the term between the modulus signs in the previous display tends to zero. The second summand of the right-hand side of \eqref{delta} vanishes for $\delta\to 0$ directly with dominated convergence. The desired  $\lim_{|x|\to0}\sup_{ |y|\in[1, N]}|H_x(y,t )- H_0(y,t )|  = 0$ now follows.\smallskip

To both verify and identify the limit in \eqref{varpi1}, we now note that the just-proved uniform continuity of $H_x(y,t ) $ implies that, for a given choice of $\varepsilon$, by taking $N$ sufficiently large such that $\varpi_0([-N,N]^c)<\varepsilon$,
\begin{align}
&\limsup_{|x|\to0}\left|\int_{(-1,1)^c}H_x(y,t )\, \varpi_x(dy)  -\int_{(-1,1)^c} H_0(y,t)\, \varpi_x(dy)\right|\notag \\
&<\limsup_{|x|\to0}\varepsilon\varpi_x([-N,N]) + 2\varpi_x([-N,N]^c)\notag\\
&\leq \varepsilon + 2\varpi_0([-N,N]^c) <3\varepsilon.
\label{almostthere}
\end{align}
Hence, \eqref{varpi1} is proved. To compute the righthand side of \eqref{varpi1} we need continuity of $y\mapsto H_0(y,t)=\P_y(T>t)$ for all $t\geq 0$ fixed, which is a consequence of the weak convergence assumption if the explosion time $T$ has no atoms. A variant of Proposition \ref{nagAx} states that the SDE started from $y$ and reversed from explosion is equal in law to the stable process $X$ issued at the origin and conditioned to hit $1/y$ via an $h$-transform using \eqref{h}, with the time change $\iota$ in \eqref{iota}. It follows that, for $y\neq 0$, 
\begin{equation}
\P_y(T>t)= \mathbb{E}_0[\mathbf{1}_{(\iota_t<\infty)} |y X_{\iota_t} -1|^{\alpha-1}].
\label{noatoms}
\end{equation} 
Dominated convergence (recall $X$ has   absolute moments in $(-1,\alpha)$) together with quasi-left/right-continuity of $X$ and the fact that  $( \iota_t, t\geq 0)$ is a  continuous additive functional  ensures that $\P_y(T>t)$ has no discontinuities for any $y\neq 0$, $t>0$.  Hence, from  \eqref{varpi1}, the continuity of $H_0$ and the weak continuity of $(\varpi_x, x\in(-1,1))$, we have  $$\lim_{|x|\to0}\mathbb{P}_x(t<\zeta^\bullet) =\lim_{|x|\to0}\int_{(-1,1)^c}H_0(y,t )\,\varpi_x(dy)  = \mathbb{P}_{\varpi_0}(t<T) =  \mathbb{P}_{0}(t<\zeta^\bullet),\quad t\geq 0,$$ where the final equality follows from the duality of $V$ and $1/Z$ from Proposition \ref{nagAx}. Portmanteau's Theorem now ensures that  we have the desired weak convergence in property (b) of Proposition \ref{P1}.
}
\smallskip

Both  conditions of Proposition \ref{P1} are thus met and hence, we may deduce  as a conclusion of that proposition that, for all $0<\varepsilon<1$, 
\begin{align}
\infty>\mathbb{E}_{0}\left[\int_0^\infty \mathbf{1}_{(|X^\bullet_{\iota_s}|\leq \varepsilon)} ds\right]  &=
\mathbb{E}_{0}\left[\int_0^{\tau^{(-1,1)^c}} \mathbf{1}_{(|X_{u}|\leq \varepsilon)} \beta(X_u)du\right]
=\int_{[-\varepsilon, \varepsilon]} \beta(x)G_{X^\bullet}(0,dx).
\label{mustbefinite}
\end{align}
From Theorem II.2.3 in \cite{KALEA}, equivalently Theorem B of \cite{PS}, it is known that, for $\alpha\in(0,1)$, $G_{X^\bullet}(0,dx)$ has a density which is asymptotically equivalent to $h$ times a constant at $0$.
%
From \eqref{mustbefinite} we thus have that 
\begin{equation}
\int_{[-\varepsilon, \varepsilon]} \beta(x)h(x)dx<\infty.
\label{mustbefinite2}
\end{equation}
Changing variables as in \eqref{compactfimportant} gives the desired integral test $I^{\sigma,\alpha}(\R)= \int_\R\sigma(y)^{-\alpha} |y|^{\alpha-1}dy<\infty.$


\subsection*{Sufficiency} 
First note that
 \begin{align}\label{potential}
	 \mathbb{E}_x\left[\int_{0}^\infty \sigma(X_t)^{-\alpha}dt\right] = \int_\mathbb{R} \sigma(y)^{-\alpha}G_X(x,dy)=\int_\mathbb{R} \sigma(y)^{-\alpha}h(x-y)dy,
 \end{align}
  where, as before,  $G_X$ is the potential measure of $X$ and $h$ is the the free potential density of $X$ given in \eqref{h}. The righthand side is finite for all $x\in \R$ if and only if
  \begin{align}\label{SDEexists}
	I^{\sigma,\alpha}(\R)= \int_\R\sigma(y)^{-\alpha} |y|^{\alpha-1}dy<\infty.
 \end{align}
Hence, if the assumed integral test holds, then the perpetual integral $T=\int_{0}^\infty \sigma(X_t)^{-\alpha}dt$ has finite expectation, thus, is finite $\P_x$-almost surely. Proposition \ref{pr} implies that for all initial conditions the unique solution to the SDE \eqref{2} almost surely explodes in finite time. {\color{black} As soon as we know that $T$ is almost surely finite, identity \eqref{noatoms} ensures there is Feller explosion.}

\appendix

\section{A transience result for Markov processes}
	We  develop a version of a result of Getoor \cite{G} on transience of Markov processes. Our version imposes stronger regularity assumptions than the main results in Getoor but gives a stronger statement also. We adopt here the same notation as in Section \ref{timereversesubsec}. Let us suppose $Y = (Y_t, t\leq \zeta)$ is a Markov process with state space $E$, cemetery state $\partial$, killing time $\zeta=\inf\{t>0: Y_t = \partial\}$ and transition (sub)probabilities $({\rm\texttt{P}}_x, x\in E)$. Denote by $\mathcal{P}: = (\mathcal{P}_t, t\geq 0)$ the associated transition semigroup and by $U_Y[f]=\int_0^\infty \mathcal{P}_t[f]\,\d t$ the potential operator. 
 The potential measure induced by $U_Y$ is denoted by $G_Y(\cdot,\cdot)$, as in the sections above. As usual, every function on $E$ is set to $0$ at the cemetery state $\partial$. 
 We prove the following proposition:
 

\begin{prop}\label{P1}Suppose $Y$ is Markov on $E$ with killing time $\zeta$ so that
	\begin{itemize}
		\item[(a)] ${\emph{\texttt{P}}}_x(\zeta\in (0,\infty))=1$ for all $x\in E$,
		\item[(b)] $x\mapsto \emph{\texttt{E}}_x[f(\zeta)]$ is continuous for all continuous and bounded functions on $[0,\infty)$, i.e. the killing time is weakly continuous in the initial condition.
	\end{itemize}
	Then $G_Y(y,K)<\infty$ for all $y\in E$ and $K\subseteq E$ compact.
\end{prop}
The proof is based on a simple lemma motivated by Lemma 3.1 of Getoor \cite{G}. The lemma identifies simple functions on $E$ for which the potential operator can be computed explicitly.
\begin{lem}\label{L1}
	Define $h_a(x)={\emph{\texttt{P}}}_x(\zeta\leq a)$ for all $x\in E$. Then we have $U_Y[ h_a](x)= \emph{\texttt{E}}_x[\zeta\wedge a]$.
\end{lem}
\begin{proof}
Using that $\mathcal{P}_a \mathbf{1}_E(x)=\texttt{P}_x(X_a\in E)=\texttt{P}_x(\zeta>a)=1- h_a(x)=\mathbf{1}_E-h_a(x)$ for $x\in E$ we find, for $x\in E$,
\begin{align*}
	U_Y[  h_a](x)
	&= \lim_{t\to\infty}\int_0^t \mathcal{P}_s [ h_a](x)\,\d s\\
	&= \lim_{t\to\infty} \int_0^t  \big(\mathcal{P}_s [\mathbf{1}_E](x)-\mathcal{P}_s[ \mathcal{P}_a [\mathbf{1}_E] ](x)\big)\,\d s\\
	&= \lim_{t\to\infty} \Big(\int_0^t  \mathcal{P}_s [\mathbf{1}_E](x)\,\d s -\int_0^t  P_{s+a} [\mathbf{1}_E ](x)\,\d s\Big)\\
	&= \lim_{t\to\infty} \Big(\int_0^t  \mathcal{P}_s [\mathbf{1}_E](x)\,\d s -\int_a^{t+a} \mathcal{P}_s [\mathbf{1}_E] (x)\,\d s\Big)\\
	&=\int_0^a \mathcal{P}_s [\mathbf{1}_E] (x)\,\d s-\lim_{t\to\infty}\int_t^{t+a} \texttt{P}_x(\zeta>s)\,\d s.
\end{align*}
By continuity of measures and Assumption (a) we have $\lim_{s\to\infty} \texttt{P}_x(\zeta>s)=\texttt{P}_x(\zeta=\infty)=0$, hence, 
\begin{align*}
	U_Y[h_a](x)= \int_0^a \texttt{P}_x(\zeta > s)\,\d s=\texttt{E}_x \left[\int_0^a \mathbf{1}_{(s<\zeta) }\,\d s\right]=\texttt{E}_x[\zeta \wedge a], 
\end{align*}
which completes the proof. \end{proof}

We can now give the proof of the proposition.

\begin{proof}[Proof of Proposition \ref{P1}]
 Fix $x\in E$. From (a) there is a constant $a_x>0$ so that $h_{a_x}(x)>0$. Now fix a continuous bounded $f_x$ on $E$ such that
\begin{align*}
	\mathbf 1_{(z\leq a_x)}\leq f_x(z)\leq \mathbf 1_{(z\leq a_x+1)},\quad \text{for all }z\in E,
\end{align*}
so that (monotonicity), for all $y\in E$,
\begin{align*}
	h_{a_x}(y)=\texttt{E}_y[\mathbf 1_{(\zeta \leq a_x)}]\leq \texttt{E}_y[f_x(\zeta)]\leq \texttt{E}_y[\mathbf 1_{(\zeta \leq a_x+1)}]=h_{a_x+1}(y).
\end{align*}
Now we have from the choice of $a_x$ and the lemma (monotonicity of potential operator for the second)
\begin{itemize}
	\item $g_x(x)>0$ with $g_x(y)=\texttt{E}_y[f_x(\zeta)]$,
	\item $U_Y g_x(y)\leq U_Y h_{a_x+1}(y)<a_x+1$ for all $y\in E$.
\end{itemize}
Using the continuity of $g_x$ we proved that for all $x$ there is an open neighborhood $O_x$ of $x$ on which $g_x$ is strictly positive and dominates an indicator $\epsilon_x \mathbf 1_{O_x}$ for some $\epsilon_x>0$. This implies $G_Y(y,O_x)<\eps_x^{-1}(a_x+1)$ for all $y\in E$. As 	$\cup_{x\in K} O_x$
is a covering for the compact set $K\subset E$,
there is a finite subcovering $K\subset \cup_{i=1}^n O_{x_i}$. Hence, 
\begin{align*}
	G_Y (y,K)\leq \sum_{i=1}^n G_Y(y,O_i)<\sum_{i=1}^n\eps^{-1}_{x_i}(a_{x_i}+1)<\infty,
\end{align*}
for all $y\in E$.
\end{proof}

\section*{Acknowledgements} Both authors would like to thank Jean Bertoin for introducing them to the problem. Part of this work was carried out when the AEK was on sabbatical at ETH Z\"urich and he would like thank the Forschungsinstitut f\"ur Mathematik for its hospitality. Both authors are especially grateful to an anonymous referee who read an initial version of this manuscript with great care, offering very helpful comments that led to its improvement. We would also like to thank Cyril Labb\'e for discussions. 

\bibliography{biblio}{}

\begin{thebibliography}{10}

\bibitem{bertoin}
J.~Bertoin.
\newblock {\em L\'evy processes}, volume 121 of {\em Cambridge Tracts in
  Mathematics}.
\newblock Cambridge University Press, Cambridge, 1996.

\bibitem{BS}
J.~Bertoin and M.~Savov.
\newblock Some applications of duality for {L}\'evy processes in a half-line.
\newblock {\em Bull. Lond. Math. Soc.}, 43(1):97--110, 2011.

\bibitem{BY02}
J.~Bertoin and M.~Yor.
\newblock The entrance laws of self-similar {M}arkov processes and exponential
  functionals of {L}\'evy processes.
\newblock {\em Potential Anal.}, 17(4):389--400, 2002.

\bibitem{BGR}
R.~M. Blumenthal, R.~K. Getoor, and D.~B. Ray.
\newblock On the distribution of first hits for the symmetric stable processes.
\newblock {\em Trans. Amer. Math. Soc.}, 99:540--554, 1961.

\bibitem{BZ}
K.~Bogdan and T.~{\.Z}ak.
\newblock On {K}elvin transformation.
\newblock {\em J. Theoret. Probab.}, 19(1):89--120, 2006.

\bibitem{CC}
M.~E. Caballero and L.~Chaumont.
\newblock Weak convergence of positive self-similar {M}arkov processes and
  overshoots of {L}\'evy processes.
\newblock {\em Ann. Probab.}, 34(3):1012--1034, 2006.

\bibitem{CPP}
M.~E. Caballero, J.~C. Pardo, and J.~L. P\'erez.
\newblock Explicit identities for {L}\'evy processes associated to symmetric
  stable processes.
\newblock {\em Bernoulli}, 17(1):34--59, 2011.

\bibitem{C96}
L.~Chaumont.
\newblock Conditionings and path decompositions for {L}\'evy processes.
\newblock {\em Stochastic Process. Appl.}, 64(1):39--54, 1996.

\bibitem{Loicnotes}
L.~Chaumont.
\newblock An introduction to self-similar processes.
\newblock {\em Lecture notes}, 2013.

\bibitem{CD}
L.~Chaumont and R.~A. Doney.
\newblock On {L}\'evy processes conditioned to stay positive.
\newblock {\em Electron. J. Probab.}, 10:no. 28, 948--961, 2005.

\bibitem{CKPR}
L.~Chaumont, A.~Kyprianou, J.~C. Pardo, and V.~Rivero.
\newblock Fluctuation theory and exit systems for positive self-similar
  {M}arkov processes.
\newblock {\em Ann. Probab.}, 40(1):245--279, 2012.

\bibitem{CPR}
L.~Chaumont, H.~Pant\'{i}, and V.~Rivero.
\newblock The {L}amperti representation of real-valued self-similar {M}arkov
  processes.
\newblock {\em Bernoulli}, 19(5B):2494--2523, 2013.

\bibitem{Chy-Lam}
O.~Chybiryakov.
\newblock The {L}amperti correspondence extended to {L}\'evy processes and
  semi-stable {M}arkov processes in locally compact groups.
\newblock {\em Stochastic Process. Appl.}, 116(5):857--872, 2006.

\bibitem{DDK}
Steffen Dereich, Leif D\"{o}ring, and Andreas~E. Kyprianou.
\newblock Real self-similar processes started from the origin.
\newblock {\em Ann. Probab.}, 45(3):1952--2003, 2017.

\bibitem{DK}
L.~D\"oring and A.~E. Kyprianou.
\newblock Perpetual integrals for {L}\'evy processes.
\newblock {\em J. Theoret. Probab.}, 29(3):1192--1198, 2016.

\bibitem{EthKur}
S.~N. Ethier and T.~G. Kurtz.
\newblock {\em {M}arkov processes}.
\newblock Wiley Series in Probability and Mathematical Statistics: Probability
  and Mathematical Statistics. John Wiley \& Sons, Inc., New York, 1986.
\newblock Characterization and convergence.

\bibitem{feller1}
W.~Feller.
\newblock The parabolic differential equations and the associated semi-groups
  of transformations.
\newblock {\em Ann. of Math. (2)}, 55:468--519, 1952.

\bibitem{feller2}
W.~Feller.
\newblock The general diffusion operator and positivity preserving semi-groups
  in one dimension.
\newblock {\em Ann. of Math. (2)}, 60:417--436, 1954.

\bibitem{G}
R.~K. Getoor.
\newblock Continuous additive functionals of a {M}arkov process with
  applications to processes with independent increments.
\newblock {\em J. Math. Anal. Appl.}, 13:132--153, 1966.

\bibitem{GV}
S.~E. Graversen and J.~Vuolle-Apiala.
\newblock {$\alpha$}-self-similar {M}arkov processes.
\newblock {\em Probab. Theory Relat. Fields}, 71(1):149--158, 1986.

\bibitem{Hunt3}
G.~A. Hunt.
\newblock Markoff processes and potentials. {I}, {II}.
\newblock {\em Illinois J. Math.}, 1:44--93, 316--369, 1957.

\bibitem{Hunt1and2}
G.~A. Hunt.
\newblock Markoff processes and potentials. {III}.
\newblock {\em Illinois J. Math.}, 2:151--213, 1958.

\bibitem{JacodShiryaev}
J.~Jacod and A.~N. Shiryaev.
\newblock {\em Limit theorems for stochastic processes}, volume 288 of {\em
  Grundlehren der Mathematischen Wissenschaften [Fundamental Principles of
  Mathematical Sciences]}.
\newblock Springer-Verlag, Berlin, second edition, 2003.

\bibitem{Kallenberg}
O.~Kallenberg.
\newblock Fondations of modern probability.
\newblock {\em Springer, Heidelberg}, 1997.

\bibitem{KaratzasShreve}
I.~Karatzas and S.~E. Shreve.
\newblock {\em Brownian motion and stochastic calculus}, volume 113 of {\em
  Graduate Texts in Mathematics}.
\newblock Springer-Verlag, New York, 1988.

\bibitem{Kiu}
S.~W. Kiu.
\newblock Semistable {M}arkov processes in {${\bf R}^{n}$}.
\newblock {\em Stochastic Process. Appl.}, 10(2):183--191, 1980.

\bibitem{KSc}
P.~K\"uhner and A.~Schnurr.
\newblock Time change equations for l\'evy-type processes.
\newblock {\em appears in Stochastic Process. Appl.}, 2017.

\bibitem{KKPW}
A.~Kuznetsov, A.~E. Kyprianou, J.~C. Pardo, and A.~R. Watson.
\newblock The hitting time of zero for a stable process.
\newblock {\em Electron. J. Probab.}, 19:no. 30, 26, 2014.

\bibitem{Kbook}
A.~E. Kyprianou.
\newblock {\em Fluctuations of {L}\'evy processes with applications}.
\newblock Universitext. Springer, Heidelberg, second edition, 2014.
\newblock Introductory lectures.

\bibitem{deep1}
A.~E. Kyprianou.
\newblock Deep factorisation of the stable process.
\newblock {\em Electron. J. Probab.}, 21:Paper No. 23, 28, 2016.

\bibitem{KALEA}
A.~E. Kyprianou.
\newblock Stable {L}\'{e}vy processes, self-similarity and the unit ball.
\newblock {\em ALEA Lat. Am. J. Probab. Math. Stat.}, 15(1):617--690, 2018.

\bibitem{KPW}
A.~E. Kyprianou, J.~C. Pardo, and A.~R. Watson.
\newblock Hitting distributions of {$\alpha$}-stable processes via path
  censoring and self-similarity.
\newblock {\em Ann. Probab.}, 42(1):398--430, 2014.

\bibitem{KRSe}
A.~E. Kyprianou, V.~Rivero, and B.~\c~Seng\"ul.
\newblock Conditioning subordinators embedded in {M}arkov processes.
\newblock {\em Stochastic Process. Appl.}, 127(4):1234--1254, 2017.

\bibitem{KV}
A.~E. Kyprianou and S.~M. Vakeroudis.
\newblock {\color{red}Stable windings at the origin}.
\newblock {\em Stochastic Process. Appl.}, 128(12):4309--4325, 2018.

\bibitem{KRS}
Andreas~E. Kyprianou, V\'{\i}ctor~M. Rivero, and Weerapat Satitkanitkul.
\newblock Conditioned real self-similar {M}arkov processes.
\newblock {\em Stochastic Process. Appl.}, 129(3):954--977, 2019.

\bibitem{L72}
J.~Lamperti.
\newblock Semi-stable {M}arkov processes. {I}.
\newblock {\em Z. Wahrscheinlichkeitstheorie und Verw. Gebiete}, 22:205--225,
  1972.

\bibitem{Pei}
P.-S. Li.
\newblock A continuous- state polynomial branching process.
\newblock {\em Stochastic Process. Appl., to appear. (see also
  arXiv:1609.09593)}, 2016.

\bibitem{N}
M.~Nagasawa.
\newblock Time reversions of {M}arkov processes.
\newblock {\em Nagoya Math. J.}, 24:177--204, 1964.

\bibitem{N2}
M.~Nagasawa.
\newblock Note on pasting of two {M}arkov processes.
\newblock {\em S\'eminaire de probabilit\'e}, 10:532--535, 1976.

\bibitem{peskir}
G.~Peskir.
\newblock On boundary behaviour of one-dimensional diffusions: From brown to
  feller and beyond.
\newblock {\em In: William Feller, Selected Papers II, Springer}, pages 77--93,
  2015.

\bibitem{Port67}
S.~C. Port.
\newblock Hitting times and potentials for recurrent stable processes.
\newblock {\em J. Analyse Math.}, 20:371--395, 1967.

\bibitem{PS}
C.~Profeta and T.~Simon.
\newblock On the harmonic measure of stable processes.
\newblock In {\em S\'eminaire de {P}robabilit\'es {XLVIII}}, volume 2168 of
  {\em Lecture Notes in Math.}, pages 325--345. Springer, Cham, 2016.

\bibitem{Rog}
B.~A. Rogozin.
\newblock Distribution of the position of absorption for stable and
  asymptotically stable random walks on an interval.
\newblock {\em Teor. Verojatnost. i Primenen.}, 17:342--349, 1972.

\bibitem{Cas}
J.~A. van Casteren.
\newblock On martingales and feller semigroups.
\newblock {\em Results in Mathematics}, 21(3):274--288, May 1992.

\bibitem{Vol}
V.~A. Volkonski\u\i.
\newblock Random substitution of time in strong {M}arkov processes.
\newblock {\em Teor. Veroyatnost. i Primenen}, 3:332--350, 1958.

\bibitem{VG}
J.~Vuolle-Apiala and S.~E. Graversen.
\newblock Duality theory for self-similar processes.
\newblock {\em Ann. Inst. H. Poincar\'e Probab. Statist.}, 22(3):323--332,
  1986.

\bibitem{W}
John~B. Walsh.
\newblock {M}arkov processes and their functionals in duality.
\newblock {\em Z. Wahrscheinlichkeitstheorie und Verw. Gebiete}, 24:229--246,
  1972.

\bibitem{Flo}
F.~Werner.
\newblock Concatenating and pasting of right processes.
\newblock {\em arXiv:1801.02595}, 2017.

\bibitem{Wh}
W.~Whitt.
\newblock Some useful functions for functional limit theorems.
\newblock {\em Math. Oper. Res.}, 5(1):67--85, 1980.

\bibitem{Whitt}
W.~Whitt.
\newblock {\em Stochastic-process limits}.
\newblock Springer Series in Operations Research. Springer-Verlag, New York,
  2002.
\newblock An introduction to stochastic-process limits and their application to
  queues.

\bibitem{z2}
P.~A. Zanzotto.
\newblock On solutions of one-dimensional stochastic differential equations
  driven by stable l\'evy motion.
\newblock {\em Stochastic Process. Appl.}, 68:209--228, 1997.

\bibitem{z}
P.~A. Zanzotto.
\newblock On stochastic differential equations driven by a {C}auchy process and
  other stable {L}\'evy motions.
\newblock {\em Ann. Probab.}, 30(2):802--825, 2002.

\end{thebibliography}
\bibliographystyle{plain}

\begin{table}[h]
\begin{tabular}{l l l}
&\multicolumn{1}{c}{\bf Glossary of some commonly used notation}&\\
&&\\
\hline
Notation & Description & Location \\
\hline
&&\\
$(X, \mathbb{P}_x)$, $\Psi$ & stable process and exponent &  \eqref{Psi_alpha_rho_parameterization_process} \\
$\hat{\square}$ & any quantity $\square$ built pathwise or in law from $-X$ as opposed to $X$ &various\\
$\tau^D$ & first entry time into $D$ of $X$ & various\\
$X^\dagger$ & stable process killed on hitting 0&\eqref{Xdagger}\\
$(X^\circ, \mathbb{P}_x^\circ)$& stable process conditioned to absorb continuously at/avoid the origin&\eqref{attract}/\eqref{avoid}\\
$(X^\uparrow, \mathbb{P}^\uparrow_x) $& stable process conditioned to stay positive and exponent&\eqref{CTSP}\\
$h$ & Doob $h$-{\color{black}function} in definiton of $X^\circ$ and $X^\uparrow$ & \eqref{h} and \eqref{speconesidedh}\\
$\hat{X}^{\circ\hspace{-1pt}>}$&negatively censored $\hat{X}^\circ$&\eqref{Ycen}\\
$\hat{X}^\ddagger$& $\hat{X}$ killed on entering $(-\infty,0)$&above \eqref{Gdd}\\
$X^\bullet$ &  $X$ killed on exiting $(-1,1)$   for $\alpha \leq  1$&Sections \ref{a=1} and \ref{sec6} \\
 & & \\
\hline
& & \\
$(\hat{V}^\circ, \hat{\mathbb{P}}_x)$ &time changed process $X_{\theta_\cdot}$& Proposition \ref{nagAx}\\
$(Z, {\rm P}_x)$ & time-changed process $X_{\tau_\cdot}$ (weak solution of SDE)& {\color{black}Proposition \ref{pr}}\\
$Z^\dagger$ & $Z$ killed on hitting the origin& various\\
$\hat{Z}^\circ$ &Nagasawa dual to $Z^\dagger$& Proposition \ref{nag}\\
$Z^\star$ & Z killed on entering $(-1,1)$  for $\alpha =1$ & Lemma \ref{resolvent} \\
$\hat{Z}^\uparrow$ & Nagasawa dual to $Z^\dagger$ with spectral positivity &Proposition \ref{nag2}\\
$T^{D}$ &first entry time of $Z$ into $D$&\S\ref{proof1}, \S\ref{proof2} \\
 & & \\
\hline
& & \\
$\overline{\R}$, $\underline{\R}$, $\overline{\underline{\R}}$ & extenstion of $\R$ using $+\infty$, $-\infty$ and $\pm\infty$ & \eqref{3R}\\
$(Y, \texttt{P}_y)$ on $S$ &  general Markov process & \eqref{semigroupMarkov}\\
$(\mathcal{P}_t, t\geq 0)$ & semigroup of general Markov process& Definition \ref{Fellersemigroupdef}\\
$\mu,\nu$ & Nagasawa duality measures (used for duals of $Z^\dagger$, $\hat{Z}^\circ$, $\hat{Z}^\uparrow$ and $Z^\updownarrow$)  &various\\
$m,\nu$ & Nagasawa duality measures (used for duals of $X^\dagger$, $\hat{X}^\circ$, $\hat{X}^\uparrow$ and $X^\updownarrow$)  &various\\
$G_Y$ & potential measure of Markov process $Y$  (e.g. for $X^\dagger$, $X^\circ$ etc.)& \eqref{GY}\\
$p_Y(t, x,dy)$ & transition measure of Markov process $Y$ (e.g. for $X^\dagger$, $X^\circ$ etc.) & e.g. \eqref{classicdual}\\
 & & \\
\hline
& & \\
$(\mathcal{X}, {P}_x)$ & (positive) self-similar Markov process & \S\ref{pssMpsect}, \S\ref{sec:rssMp}\\
$(\xi, \mathbf{P})$ & L\'evy process underlying Lamperti transform &\S\ref{pssMpsect}\\
$((\xi, J), \mathbf{P}_{x,i})$ & MAP underlying Lamperti--Kiu transform & \S\ref{sec:rssMp}\\

$(\xi^{>},\mathbf{P}^>$), $\Psi^{>}$  & L\'evy process underlying censored stable process   and exponent&\eqref{censoredpsi}\\
$(\xi^{|\cdot|}, \mathbf{P}^{|\cdot|})$, $\Psi^{|\cdot|}$ & L\'evy process underlying $|X|$ when $\rho = 1/2$ and exponent& \eqref{a}\\
$\xi^\uparrow$, $\Psi^\uparrow$ & L\'evy process underlying  $X^\uparrow$ and exponent& \eqref{Psiuparrow}\\
$\xi^\dagger$, $\Psi^\dagger$  & L\'evy process underlying $X^\dagger$ in spectrally positive case and exponent& \eqref{xipsi}\\
$\hat{\xi}^{\circ\hspace{-1pt}>}$, $\hat{\Psi}^{\circ\hspace{-1pt}>}$& L\'evy process underlying $\hat{X}^{\circ\hspace{-1pt}>}$&below \eqref{Ycen}\\
  & & \\
\hline

\end{tabular}
\label{table-notation}
\end{table}

\end{document}